\documentclass[11pt]{article}%
\usepackage{mathrsfs} 
\usepackage[english]{babel}

\usepackage{color}

\usepackage{bbm}
\usepackage{amsfonts}

\usepackage{geometry}%
\usepackage{amsmath}%
\usepackage{amssymb}%
\usepackage{graphicx}
\providecommand{\U}[1]{\protect\rule{.1in}{.1in}}

\newtheorem{theorem}{Theorem}

\newtheorem{corollary}[theorem]{Corollary}

\newtheorem{definition}[theorem]{Definition}
\newtheorem{hypothesis}[theorem]{Hypothesis}

\newtheorem{lemma}[theorem]{Lemma}

\newtheorem{proposition}[theorem]{Proposition}
\newtheorem{remark}[theorem]{Remark}

\newenvironment{proof}[1][Proof]{\noindent\textbf{#1.} }{\ \rule{0.5em}{0.5em}}

\newcommand{\R} {\ensuremath {\mathbb{R}}}
\newcommand{\N} {\ensuremath {\mathbb{N}}}
\newcommand{\E} {\ensuremath {\mathbb{E}}}
\newcommand{\Q} {\ensuremath {\mathbb{Q}}}

\newcommand{\calL} {\ensuremath {\mathcal{L}}}

\newcommand{\PP} {\ensuremath {\mathbb{P}}}

\newcommand{\calF} {\ensuremath {\mathcal{F}}}

\newcommand{\dd}  {\ensuremath{\mathrm{d}}}

\def\hh{\vskip 2mm \noindent}
\def\vv{\vskip 2mm }

\begin{document}

\title{Regularity of Stochastic Kinetic Equations}
\author{Ennio Fedrizzi \footnote{I2M, Aix-Marseille Univ, CNRS, UMR 7373, Centrale Marseille, 13453 Marseille, France. Email address: {\tt ennio.fedrizzi@centrale-marseille.fr}.}, 
Franco Flandoli\footnote{Dipartimento di Matematica, Largo Bruno Pontecorvo 5, Universit\`a 
  di Pisa, 56127 Pisa, Italy. Email address: {\tt flandoli@dma.unipi.it}.}, 
Enrico Priola \footnote{ Dipartimento di Matematica, Università di Torino, via Carlo Alberto 10, 10123 Torino, Italy. Email address: {\tt enrico.priola@unito.it}}, 
Julien Vovelle \footnote{Univ Lyon, Université Claude Bernard Lyon 1, CNRS UMR 5208, Institut Camille Jordan, 43 blvd. du 11 novembre 1918, F-69622 Villeurbanne cedex, France    Email address: {\tt vovelle@math.univ-lyon1.fr}}}
\date{}

\maketitle

\abstract{ We consider regularity properties of stochastic kinetic equations with multiplicative noise and drift term which belongs to a space of mixed regularity ($L^p$-regularity in the velocity-variable and Sobolev regularity in the space-variable).
 We prove that, in contrast with the deterministic case,
 the SPDE  admits a  unique  weakly differentiable solution 
 which preserves a certain degree of Sobolev regularity of the initial condition 
 without developing discontinuities.
To prove the result  we also study  the related degenerate Kolmogorov equation in Bessel-Sobolev spaces and construct a suitable stochastic flow. }

\section{Introduction }

We consider the linear Stochastic Partial Differential Equation (SPDE) of
kinetic transport type 
\begin{equation}
\dd_{t}f+\left(  v \cdot D_{x}f+F\cdot D_{v}f\right)  \dd t+D_{v}f\circ
\dd W_{t}=0,\qquad f \big|_{t=0}=f_{0}\label{SPDEintro}%
\end{equation}
and the associated stochastic characteristics described by the stochastic
differential equation (SDE)
\begin{equation}
\label{eq-SDE}
\begin{cases}
\dd X_t &  =V_t \dd t,\qquad \dd V_t = F \left(  X_t,V_t \right)  \dd t + \dd W_{t} \\
X\left(  0\right)   &  =x_{0}, \;\;\qquad V\left(  0\right)  =v_{0}.
\end{cases}
\end{equation}
Here $t\in\left[  0,T\right]  $, $\left(  x,v\right)  \in\mathbb{R}^{d}%
\times\mathbb{R}^{d}$, $f:\left[  0,T\right]  \times\mathbb{R}^{d}%
\times\mathbb{R}^{d}\rightarrow\mathbb{R}$,  $f_{0}:\mathbb{R}^{d}%
\times\mathbb{R}^{d}\rightarrow\mathbb{R}$,
$F:\mathbb{R}^{d}
\times\mathbb{R}^{d}\rightarrow\mathbb{R}^d$,
$x_{0},v_{0}\in\mathbb{R}^{d}$ and
$\left(  W_{t}\right)  _{t\geq0}$ is a $d$-dimensional Brownian motion defined
on a complete filtered probability space $\big(  \Omega,\mathcal{F},\left(
\mathcal{F}_{t}\right)  _{t\geq0},\PP \big)  $;\ the operation $D
_{v}f\circ \dd W_{t}=\sum_{\alpha=1}^{d}\partial_{v_{\alpha}}f\circ
\dd W_{t}^{\alpha}$ will be understood in the Stratonovich sense, in order to
preserve (a priori only formally) the relation $\dd f\left(  t,X_{t}%
,V_{t}\right)  =0$, when $\left(  X_{t},V_{t}\right)  $ is a solution of the
SDE; we use Stratonovich not only for this mathematical 
convenience, but also
because, in the spirit of the so called Wong-Zakai principle, the Stratonovich
sense is the natural one from the physical view-point as a limit of correlated
noise with small time-correlation. The physical meaning of the SPDE
(\ref{SPDEintro}) is the transport of a scalar quantity described by the function
$f\left(  t,x,v\right)  $ (or the evolution of a density $f\left(
t,x,v\right)  $, when $\operatorname{div}_{v}F=0$, so that $F D
_{v}f=\operatorname{div}_{v}\left(  Ff\right)  $), under the action of a fluid
- or particle - motion described by the SDE (\ref{eq-SDE}), where we have two
force components: a ``mean" (large scale) component $F\left(  x,v\right)  $,
plus a fast fluctuating perturbation given by $\frac{\dd W_{t}}{\dd t}$. Under
suitable assumptions and more technical work one can consider more elaborate
and flexible noise terms, space dependent, of the form $\sum_{k=1}^{\infty
}\sigma_{k}\left(  x\right)  \dd W_{t}^{k}$ (see \cite{CoFl},
\cite{DelarueFlaVincenzi}\ for examples of assumptions on a noise with this
structure and \cite{FalGawVerg} for physical motivations), but for the purpose
of this paper it is sufficient to consider the simplest noise $\dd W_{t}%
=\sum_{k=1}^{d}e_{k} \, \dd W_{t}^{k}$, $\left\{  e_{k}\right\}  _{k=1,...,d}$ being
an orthonormal base of $\mathbb{R}^{d}$.

Our aim is to show that noise has a regularizing effect on both the SDE
(\ref{eq-SDE}) and the SPDE (\ref{SPDEintro}), in the sense that it provides results
of existence, uniqueness and regularity under assumptions on $F$ which are
forbidden in the deterministic case. Results of this nature have been proved
recently for other equations of transport type, see for instance
\cite{FGP10}, \cite{FF13b}, \cite{Fla Saint Flour}, \cite{BFGM14}, but here,
for the first time, we deal with the case of ``degenerate" noise, because
$\dd W_{t}$ acts only on a component of the system. It is well known that the
kinetic structure has good ``propagation" properties from the $v$ to the $x$
component; however, for the purpose of regularization by noise one needs
precise results which are investigated here for the first time and are
technically quite non trivial. 
Let us describe more precisely the result proved here. First, we investigate
the SDE (\ref{eq-SDE}) under the assumption (see below for more details) that
$F$ is in the mixed regularity space   $L^{p}\left(  \mathbb{R}_{v}^{d};W^{s,p}\left(
\mathbb{R}_{x}^{d}; \R^d\right)  \right)  $ for some $s\in (\frac{2}{3},1)$ and $p > 6d$;
this means that we require
$$
\int_{\R^d} \| F(\cdot , v) \|_{ W^{s, p} 
}^p \,  \dd v< \infty \, ,
$$
where $W^{s, p}  = W^{s, p} (\R^{d}; \R^d)$ is a fractional Sobolev space
(cf. Hypothesis \ref{hyp-holder} and the comments after this assumption; see also Sections 3.1 and 3.2 for more details).
Thus our drift is only $L^p$ in the ``good'' $v$-variable in which the noise acts and has Sobolev regularity in the other $x$-variable. This is particularly clear   in the special case of 
\begin{equation}
\label{dopo}
F(x,v) = \varphi (v) G(x) \, ,
\end{equation}
where $G \in W^{s, p} (\R^{d}; \R^d)$, $\varphi \in L^p(\R^d)$ and 
 $p > 6d$ with $s\in (\frac{2}{3},1).$
 Just to mention 
in the  case of full-noise action, the best known assumption to get pathwise uniqueness (cf. \cite{KR05}) is that  $F$ must belong to  $L^p (\R^{N}; \R^N)$, $p>N$ (in our case $N= 2d$).

According to a general scheme  (see \cite{V}, \cite{KR05}, \cite{FGP10}, \cite{FF11}, \cite{Fla Saint Flour} \cite{P1}, \cite{FF13a}, \cite{FF13b}, \cite{BFGM14}, \cite{CdR}, \cite{WaZh1}, \cite{WaZa2})
to study regularity  properties of the stochastic characteristics 
one   first needs to establish  precise regularity results for solutions to associated  Kolmogorov 
equations. In our case such equations are
 degenerate elliptic equations
 of the type 
\begin{equation}
\label{kol2}
  \lambda {\psi} (x,v) - \frac{1}{2} \triangle_v {\psi}(x,v) - v \cdot D_x {\psi} (x,v) 
 - F(x,v) \cdot D_v {\psi}(x,v) =    g(x,v) \, ,  
\end{equation} 
where $\lambda >0$ is 
given (see Section \ref{sec: regularity Bessel}). 
We prove an optimal regularity result 
 for \eqref{kol2} 
in Bessel-Sobolev spaces 
(see  Theorem \ref{PDE-1!}).  Such result requires   basic $L^p$-estimates proved in  \cite{Bo} and \cite{BCLP} and   non-standard interpolation techniques for functions from $\R^d$ with values in Bessel-Sobolev  spaces (see in particular the proofs of Theorem \ref{cos} and Lemma \ref{corsa}).

The results of Section \ref{sec:PDE} are exploited in Section \ref{sec:SDE} to prove existence of strong solutions to \eqref{eq-SDE} and pathwise uniqueness. Moreover, we can also construct a continuous stochastic flow, injective and surjective, hence a flow of homeomorphisms. 
These maps are locally $\gamma$-H\"{o}lder continuous
for every $\gamma\in\left(  0,1\right)  $. We cannot say that they are
diffeomorphisms;\ however, we can show that for any $t$ and $\PP$-a.s. the random variable $Z_t=(X_t, V_t)$ admits a distributional derivative with respect to $z_0=(x_0,v_0)$. Moreover, for any $t$ and $p >1$, the weak derivative $D_z Z_t \in L^p_{loc} (\Omega \times \R^{2d})$ (i.e., $D_z Z_t \in L^p (\Omega \times K)$, for any compact set $K \subset \R^{2d}$; see Theorem \ref{teo SDE-differentiable flow}).  These results are a generalization to the kinetic (hence degenerate noise) case of theorems in \cite{FF13b}.

Well-posedness for kinetic SDEs \eqref{eq-SDE} 
with non-Lipschitz drift has been recently investigated: strong existence and uniqueness  have  been recently proved in  
 \cite{CdR}  and \cite{WaZh1}. 
Moreover, a stochastic flow of diffeomorphisms 
has  been obtained in \cite{WaZa2} even with a multiplicative noise. In \cite{WaZh1} and \cite{WaZa2} the drift  is assumed to be $\beta$-H\"{o}lder continuous in the $x$-variable with $\beta>\frac{2}{3}$ and Dini continuous in the $v$-variable. The results here are more general  even concerning the regularity in the $x$-variable (see also Section 2.1). We stress  that  well-posedness is not true without noise, as the counter-examples given by Propositions \ref{Prop 1 Appendix} and \ref{Prop 2 Appendix}  show. 

Based on our results on the stochastic flow, 
we prove in Section \ref{sec: SKE} that if the initial condition $f_{0}$
is sufficiently smooth, the SPDE 
(\ref{SPDEintro}) admits a   weakly differentiable solution and provide a representation formula (see Theorem \ref{Theo regul}). Moreover,  the solution of equation (\ref{SPDEintro}) in the spatial variable is of class
$W_{loc}^{1,r}\left(  \mathbb{R}^{2d}\right)  $, for every
$r\geq1$, $\PP$-a.s., at every time $t\in\left[  0,T\right]  $.\ Such regularity
result is not true without noise: \ Proposition \ref{Prop 2 Appendix} gives an
example where solutions develop discontinuities from smooth initial
conditions and  with drift in the class considered here.
Moreover,  assuming in addition that $\mathrm{div}_v F \in L^\infty(\R^{2d})$
we prove  uniqueness of weakly differentiable solutions (see Theorem \ref{teo uniqueness}).

The results presented here may also serve as a preliminary for the
investigation of properties of interest in the theory of
kinetic equations, where again we see a regularization by noise. In a
forthcoming paper we shall investigate the mixing property%
\[
\left\Vert f_{t}\right\Vert _{L_{x}^{\infty}\left(  L_{v}^{1}\right)  }\leq
C\left(  t\right)  \left\Vert f_{0}\right\Vert _{L_{x}^{1}\left(
L_{v}^{\infty}\right)  }, 
\]
with $C\left(  t\right)  $ diverging as $t\rightarrow0$, to see if it holds when the noise is present in comparison to the deterministic case (cf. \cite{GS02} and \cite{H-K}). Again the theory of
stochastic flows, absent without noise under our assumptions, is a basic
ingredient for this analysis.


The paper is constructed as follows. We begin by introducing in the next section some necessary notation and presenting some examples that motivate our study. In Section \ref{sec:PDE} we state some well-posedness results for an associated degenerate elliptic equation (see Theorem \ref{PDE-1!}, which contains the main result of this section). These results will be used in Section \ref{sec:SDE} to solve the stochastic equation of characteristics associated to \eqref{SPDEintro}. This is a degenerate stochastic equation, but we can prove existence and uniqueness of strong solutions (see Theorem \ref{strong 1!}), generating a weakly differentiable flow of homeomorphisms (see Theorems \ref{teo-flow} and \ref{teo SDE-differentiable flow}). Using all these tools, we can finally show in Section \ref{sec: SKE} that the stochastic kinetic equation \eqref{SPDEintro} is well-posed in the class of weakly differentiable solutions.

\section{Notation and Examples}\label{sec preliminaries}

We will either use a dot or $\langle \ ,\, \rangle$ to denote the scalar product in $\R^d$ and $|\cdot|$ for the Euclidian norm. Other norms will be denoted by $\| \cdot \|$, and for the sup norm we shall use both $\| \cdot \|_\infty$ and $\| \cdot \|_{L^\infty(\R^{d})}$. $C_b(\R^{d})$ denotes the Banach space of all real continuous and bounded functions $f : \R^d \to \R$ endowed with the sup norm; $C_b^1 (\R^{d}) \subset C_b(\R^{d})$ is the subspace of all functions wich are differentiable on $\R^{d}$ with bounded and continuous partial derivatives on $\R^{d}$; for $\alpha\in\R_+\backslash \N$, $C^\alpha (\R^d) \subset C^0$ is the space of $\alpha$-H\"older continuous functions on $\R^d$; $C_c^{\infty} (\R^{d}) \subset C_b(\R^{d})$ is the space of all infinitely differentiable functions with compact support. $C, c, K$ will denote different constants, and we use subscripts to indicate the parameters on which they depend.

Throughout the paper, we shall use the notation $z$ to denote the point $(x,v) \in \R^{2d}$. Thus, for a scalar function $g(z): \R^{2d} \to \R$, $D_z g$ will denote the vector in $\R^{2d}$ of derivatives with respect to all variables $z=(x,v)$, $D_x g \in \R^d$ denotes the vector of derivatives taken only with respect to the first $d$ variables and similarly for $D_v g(z)$. We will have to work with spaces of functions of different regularity in the $x$ and $v$ variables: we will then use subscripts to distinguish the space and velocity variables, as in Hypothesis \ref{hyp-holder}.

Let us state the regularity assumptions we impose on the force field $F$.

\begin{hypothesis}\label{hyp-holder}
  The function $F:\R^{2d}\to \R^d$ is a Borel function such that
 \begin  {equation} \label{fra2}
 \int_{\R^d} \| F(\cdot , v) \|_{ H^{s}_p }^p\, \dd v< \infty
\end{equation}
where  $s \in (2/3, 1)$ and $p > 6d$. We write that
$F \in L^p\left( \mathbb{R}_{v}^{d};H^{s}_p\left(
\mathbb{R}_{x}^{d}; \R^d\right) \right)$.
\end{hypothesis}
{  The Bessel space $H^{s}_p = H^{s}_p (\R^{d}; \R^d)$ is defined by the Fourier transform (see Section 3).
According to Remark \ref{in4}, condition \eqref{fra2}  can also be rewritten
using  the related fractional Sobolev spaces $W^{s,p}(\R^d; \R^d)$ instead of $H^{s}_p (\R^{d}; \R^d)$. In the sequel we will also write $H^{s}_p (\R^{d})$ instead of $H^{s}_p (\R^{d}; \R^d)$ when no confusion may arise.

\subsection{Examples}

Without noise, when $F$ is only in the space 
$L^{p}\left(\mathbb{R}^{d}_v;H_{p}^{s}\left(  \mathbb{R}^{d}_x\right)  \right)  $ for some $s>\frac{2}{3}$ and $p>6d$, the equation for the  characteristics 
\begin{align}
x^{\prime} &  =v,\qquad v^{\prime}=F\left(  x,v\right)  \label{character}\\
x\left(  0\right)   &  =x_{0},\qquad v\left(  0\right)  =v_{0}\nonumber
\end{align}
and the associated kinetic transport equation
\begin{equation}
D_{t}f+v \cdot D_{x}f+F \cdot D_{v} f=0,\qquad f|_{t=0}=f_{0}%
\label{transport det}%
\end{equation}
may have various types of pathologies. We shall mention here some of them in
the very simple case of  $d=1$,
\begin{align} \label{coun}
F\left(  x,v\right)   &  = \pm \theta\left(  x,v\right)  sign\left(  x\right)
\left\vert x\right\vert ^{\alpha}\\ \nonumber
\text{for some }\alpha &  \in\left(  \frac{1}{2},1\right)  ,  \quad  \theta\in
C_{c}^{\infty}\left(  \mathbb{R}^{2}\right)  .
\end{align}

First, note that  this function belongs to $L^{p}\left(
\mathbb{R}_v;H_{p}^{s}\left(  \mathbb{R}_x\right)  \right)  $ for 
for some $s>\frac
{2}{3}$ and $p>6$ (to check this fact one can use the
Sobolev embedding theorem:\ $H_{q}^{1}\left(  \mathbb{R}\right)  \subset
H_{p}^{s}\left(  \mathbb{R}\right)  $ if $\frac{1}{p}=\frac{1}{q}-1+s$).

Thus $F$ satisfies our Hypothesis \ref{hyp-holder}. 
On the other hand 
when $\alpha\in\left(  \frac{1}{2},\frac{2}{3}\right)  $, the function
$sign\left(  x\right)  \left\vert x\right\vert ^{\alpha}$ is not in
$C_{loc}^{\gamma}\left(  \mathbb{R}\right)  $ for any $\gamma>2/3$ and  the
results of \cite{WaZh1}, \cite{WaZa2} do not apply.


 Let us come to the description of the pathologies of characteristics and
kinetic equation when $F\left(  x,v\right)  = \pm \theta\left(  x,v\right)
\left\vert x\right\vert ^{\alpha}$.

\begin{proposition}\label{prop ex1}
\label{Prop 1 Appendix}In $d=1$, if $\theta\in C_{c}^{\infty}\left(
\mathbb{R}^{2}\right)  $, $\theta=1$ on $B\left(  0,R\right)  $ for some
$R>0$, $F\left(  x,v\right)  =\theta\left(  x,v\right)  sign\left(  x\right)
\left\vert x\right\vert ^{\alpha}$, then system (\ref{character}) with initial
condition $\left( x_0 , 0 \right)  $ has infinitely many solutions. In particular,
for small time (depending on $R$ and $\alpha$), $\left( x_{t},v_{t}\right)
=\left( x_0 +  At^{\beta},A\beta t^{\beta-1}\right)  $, with $\left(  \beta
,A\right)  $ satisfying (\ref{beta A}) below, and also $A=0$, are solutions.
\end{proposition}

\begin{proof}
Let us check that $\left(  x_{t},v_{t}\right)  =\left(  At^{\beta},A\beta
t^{\beta-1}\right)  $ with the specified values of $\left(  \beta,A\right)  $
and a small range of $t$, are solutions. We have $x_{t}^{\prime}=v_{t}$,
\begin{align*}
v_{t}^{\prime}-F\left(  x,v_{t}\right)   &  =A\beta\left(  \beta-1\right)
t^{\beta-2}-sign\left(  x_{t}\right)  \left\vert x_{t}\right\vert ^{\alpha}\\
&  =A\beta\left(  \beta-1\right)  t^{\beta-2}-sign\left(  A\right)  \left\vert
A\right\vert ^{\alpha}t^{\alpha\beta}=0
\end{align*}
for $\alpha\beta=\beta-2$ and $A\beta\left(  \beta-1\right)  =sign\left(
A\right)  \left\vert A\right\vert ^{\alpha}$, namely%
\begin{equation}
\beta=\frac{2}{1-\alpha},\qquad A=\pm\left(  \frac{1}{\beta\left(
\beta-1\right)  }\right)  ^{\frac{1}{1-\alpha}}=\pm\left(  \frac{\left(
1-\alpha\right)  ^{2}}{2\left(  1+\alpha\right)  }\right)  ^{\frac{1}%
{1-\alpha}}. \label{beta A}
\end{equation}
\end{proof}

With a little greater effort one can show, in this specific example, that
every solution $\left(  x_{t},v_{t}\right)  $ from the initial condition
$\left(  0,0\right)  $ has, for small time, the form $\left(  x_{t}%
,v_{t}\right)  =\left(  A\left(  t-t_{0}\right)  ^{\beta},A\beta\left(
t-t_{0}\right)  ^{\beta-1}\right)  1_{t\geq t_{0}}$ for some $t_{0}\geq0$, or
it is $\left(  x_{t},v_{t}\right)  =\left(  0,0\right)  $ ($\left(
\beta,A\right)  $ always given by (\ref{beta A})) and that existence and
uniqueness holds from any other initial condition, even from points of the
form $\left(  0,v_{0}\right)  $, $v_{0}\neq0$, around which $F$ is not
Lipschitz continuous. Given $T>0$ and $R>0$ large enough, there is thus, at
every time \thinspace$t\in\left[  0,T\right]  $, a set $\Lambda_{t}%
\subset\mathbb{R}^{2}$ of points  ``reached from $\left(  0,0\right)  $", which
is the set
\[
\Lambda_{t}=\left\{  \left(  A\left(  t-t_{0}\right)  ^{\beta},A\beta\left(
t-t_{0}\right)  ^{\beta-1}\right)  \in\mathbb{R}^{2}:t_{0}\in\left[
0,t\right]  \right\}  .
\]
Using this family of sets one can construct examples of non uniqueness for the
transport equation (\ref{transport det}), because a solution $f\left(
t,x,v\right)  $ is not uniquely determined on $\Lambda_{t}$. However, these
examples are not striking since the region of non-uniqueness, $\cup_{t\geq
0}\Lambda_{t}$, is thin and one could say that uniqueness is restored by a
modification of $f$ on a set of measure zero. But, with some additional effort, it is also possible to construct un example with $F\left(  x,v\right)  = \pm \theta\left(  x,v\right) \left\vert x\right\vert ^{\alpha}$. In this case, for some negative $m$ (depending on $R$ and $\alpha$), one can construct infinitely many solutions $(x_t, v_t)$ starting from any point in a segment $(x_0,0)$, $x_0 \in [m,0)$. Indeed, $(x_t, v_t) = (x_0, 0)$ is a solution, but there are also solutions leaving $(x_0, 0)$ which will have $v_t>0$, at least for some small time interval. Then one obtains that the solution $f(t,x,v)$ is not uniquely determined on a set of positive Lebesgue measure.

More relevant, for a simple class of drift as the one above, is the phenomenon
of loss of regularity. Preliminary, notice that, when $F$ is Lipschitz
continuous, system (\ref{character}) generates a Lipschitz continuous flow
and, using it, one can show that, for every Lipschitz continuous
$f_{0}:\mathbb{R}^{2}\rightarrow\mathbb{R}$, the transport equation
(\ref{transport det}) has a unique solution in the class of continuous
functions $f:\left[  0,T\right]  \times\mathbb{R}^{2}\rightarrow\mathbb{R}$
that are Lipschitz continuous in $\left(  x,v\right)  $, uniformly in $t$.
The next proposition identifies an example with non-Lipschitz $F$ where this
persistence of regularity is lost. More precisely, even starting from a smooth
initial condition, unless it has special symmetry properties, there is a
solution with a point of discontinuity. This pathology is removed by noise,
since we will show that with sufficiently good initial condition, the unique
solution $f(t,z)$ is of class $W_{loc}^{1,r}(\R^2)$ for every $r\ge1$ and $t\in[0,T]$ a.s., hence in particular continuous. However, in the stochastic case, we do not know whether the
solution is Lipschitz under our assumptions, whereas presumably it is under
the stronger H\"{o}lder assumptions on $F$ of \cite{WaZa2}.

\begin{proposition}
\label{Prop 2 Appendix}In $d=1$, if $\theta\in C_{c}^{\infty}\left(
\mathbb{R}^{2}\right)  $, $\theta=1$ on $B\left(  0,R\right)  $ for some
$R>0$, $F\left(  x,v\right)  = \theta\left(  x,v\right)  sign\left(  x\right)
\left\vert x\right\vert ^{\alpha}$, then system (\ref{character}) has a unique
local solution on any domain not containing the origin, for every initial condition. For every $t_{0}>0$ (small enough with
respect to $R$), the two initial conditions $\left(  At_{0}^{\beta}, -A\beta
t_{0}^{\beta-1}\right)  $ with $\left(  \beta,A\right)  $ given by
(\ref{beta A}) produce the solution 
\begin{equation*}
\left(  x_{t},v_{t}\right)  =\left(  A\left(
t_{0}-t\right)  ^{\beta}, \, -A\beta\left(  t_{0}-t\right)  ^{\beta-1}\right)  
\end{equation*}
for $t\in\left[  0,t_{0}\right]$, and 
$
\left(  x_{t_{0}},v_{t_{0}}\right)  =\left(  0,0\right) \,.
$
As a consequence, the transport equation
(\ref{transport det}) with any smooth $f_{0}$ such that $f_{0}\left(
At_{0}^{\beta}, \, -A\beta t_{0}^{\beta-1}\right)  \neq f_{0}\left(  -At_{0}%
^{\beta}, \, A\beta t_{0}^{\beta-1}\right)  $ for some $t_{0}>0$, has a solution 
with a discontinuity at time $t_{0}$ at position $\left(
x,v\right)  =\left(  0,0\right)  $.
\end{proposition}

\begin{proof}
The proof is elementary but a full proof is lengthy. We limit ourselves to a
few simple facts, without proving that system (\ref{character}) is forward
well posed (locally in time) and the transport equation (\ref{transport det}) is also well posed
in the set of weak solutions. We only stress that the claim $\left(
x_{t_{0}},v_{t_{0}}\right)  =\left(  0,0\right)  $ when
the initial condition is $\left(  At_{0}^{\beta}, \, -A\beta t_{0}^{\beta
-1}\right)  $ can be checked by direct computation (as in the previous
proposition) and the discontinuity of the solution $f$ of (\ref{transport det}%
) is a consequence of the transport property, namely the fact that whenever
$f$ is regular we have
\begin{equation}
f\left(  t,x_{t},v_{t}\right)  =f_{0}\left(  x_{0},v_{0}\right)
\label{transp identity}%
\end{equation}
where $\left(  x_{t},v_{t}\right)  $ is the unique solution with initial
condition $\left(  x_{0},v_{0}\right)  $. Hence we have this identity for
points close (but not equal) to the coalescing ones mentioned above, where the
forward flow is regular and a smooth initial condition $f_{0}$ gives rise to a
smooth solution; but then, from identity (\ref{transp identity}) in nearby
points, the limit%
\[
\lim_{\left(  x,v\right)  \rightarrow\left(  0,0\right)  }f\left(
t_{0},x,v\right)
\]
does not exists if $t_{0}$ is as above and $f_{0}\left(  At_{0}^{\beta}, \,-A\beta
t_{0}^{\beta-1}\right)  \neq f_{0}\left(  -At_{0}^{\beta}, \, A\beta t_{0}%
^{\beta-1}\right)  $.
\end{proof}

\section{Well-posedness for degenerate Kolmogorov equations in Bessel-Sobolev spaces}\label{sec:PDE}

\subsection{Preliminaries on functions spaces and interpolation theory}

Here we collect basic facts on Bessel 
and Besov 
spaces (see \cite{B}, \cite{T} and \cite{S} for more details).  In the sequel if $X$ and $Y$ are real Banach spaces then $Y \subset X$ means that $Y$ is  continuously embedded in $X$.

The Bessel (potential) spaces  are defined as follows (cf. \cite{B} page 139 and \cite{S} page 135). For the sake of simplicity  we only consider  $p \in [2, \infty)$ and $s \in \R_+$.

First one considers the Bessel potential $J^s $,
$$
J^s f = {\cal F}^{-1} [ (1+ |\cdot|^{2})^{s/2}) {\cal F} f] 
$$
where $\cal F$ denotes the Fourier transform of a distribution $f \in {\cal S}' (\R^d)$, $d \ge 1$. Then we introduce
\begin{equation*}
H^{s}_p(\R^d) = \{ f \in {\cal S}' (\R^d) \; :\;  
J^s f \in L^p(\R^d)
 \}
\end{equation*}
(clearly $H^0_p(\R^d) = L^p(\R^d)$). This is a Banach space endowed with the norm $\| f\|_{H^s_p} = $ $ \| J^s f \|_p $, where $\| \cdot\|_p$ is the usual norm of $L^p(\R^d)$
 (we identify functions with coincide a.e.). It can be proved that 
\begin{equation} \label{st}
H^{s}_p (\R^d)= \{ f \in L^p (\R^d) \; :\;   
   {\cal F}^{-1} [ |\cdot|^{s} \,  {\cal F} f] \in L^p(\R^d)  \}
\end{equation}
and an equivalent norm in $H^{s}_p (\R^d)$ is 
\begin{equation*}
 \| f\|_{s,p} =  \| f\|_p +
  \|  {\cal F}^{-1} [ |\cdot|^{s} \,  {\cal F} f]  \|_p 
 \backsimeq  \|  {\cal F}^{-1} [ (1+ |\cdot|^{s}) \,  {\cal F} f]  \|_p. 
\end{equation*} 
To show this characterization one can use  that 
$$
(1 + 4 \pi^2 |x|^2)^{s/2}  = (1+ (2 \pi |x|^s)) \, [{\cal F}\phi(x) +1], \;\;\; x \in \R^d,
$$
for some $\phi \in L^1(\R^d)$ (see page 134 in \cite{S}), and basic properties of convolution and Fourier transform.
We note that
\begin{equation}\label{H=W}
H^{k}_p (\R^d) = W^{k,p} (\R^d)
\end{equation}   
 if $k \ge 0$ is an integer with equivalence of norms (here $W^{k,p} (\R^d)$ is the usual Sobolev space; $W^{0,p}(\R^d) = L^{p}(\R^d)$); see  Theorem 6.2.3 in \cite{B}.  
However if $s$ is not an integer we only have (see Theorem 6.4.4 in \cite{B}  or \cite{S} page 155) 
\begin{equation} \label{den}
 H^{s}_p (\R^d) \subset W^{s,p} (\R^d)
\end{equation} 
where  $W^{s,p} (\R^d)$ is a fractional  Sobolev space (see below). We have (cf. Theorem 6.2.3 in \cite{B}) 
\begin{equation*}
H^{s_2}_p (\R^d) \subset H^{s_1}_p (\R^d)
\end{equation*} 
if $s_2 > s_1$ and, moreover,  $C_c^{\infty} (\R^d)$ is dense in any $H^{s}_p (\R^d)$. 

\smallskip    One can compare Bessel spaces with   Besov spaces $B^{s}_{p, q} (\R^d) $ (see, for instance,  Theorem 6.2.5 in \cite{B}).
Let   $p, q \ge 2$, $s \in (0, 2)$, to simplify notation. 

If $s \in (0,1)$  then  $ B^{s}_{p, q} (\R^d)$ consists of functions $f \in L^p(\R^d)$ such that
$$
[f]_{B^{s}_{p, q}} = \Big ( 
\int_{\R^d} \frac{\dd h}{ |h|^{d + s q} } \Big( \int_{\R^d} 
 |f(x+h) - f(x)|^p \dd x \Big)^{q/p}
\Big)^{1/q} < \infty \, .
$$
Thus we have 
\begin{equation} \label{rft}
  B^{s}_{p, p} (\R^d ) = W^{s,p} (\R^d)
\end{equation}
with equivalence of norms. However if $s =1$, 
$ B^{1}_{p, q} (\R^d)$ consists of all functions $f \in L^p(\R^d)$ such that
$$
[f]_{B^{1}_{p, q}} = \Big ( 
\int_{\R^d} \frac{\dd  h}{ |h|^{d +  q} } \Big( \int_{\R^d} 
 |f(x+2h) - 2 f(x+h) + f(x)|^p \dd x
 \Big)^{q/p}  \Big)^{1/q} < \infty \, .
$$
Thus we only have  $B^{1}_{p, p} (\R^d ) \subset W^{1,p} (\R^d)$. Note that 
$B^{s}_{p, q} (\R^d )$ is a Banach space endowed with the norm: $\| \cdot\|_{p}$ $+[ \cdot]_{B^{s}_{p, q}}  $.
Similarly, if $s \in (1,2)$,  
then  $ B^{s}_{p, q} (\R^d)$ consists of functions $f \in W^{1,p} (\R^d)$ such that
\begin{equation}
\label{r55}
[f]_{B^{s}_{p, q}} = \sum_{i=1}^d\Big ( 
\int_{\R^d} \frac{\dd h}{ |h|^{d + s q} } \Big( \int_{\R^d} 
 |\partial_{x_i} f(x+h) - \partial_{x_i}f(x)|^p \dd x
 \Big)^{q/p}  \Big)^{1/q} < \infty \, .
\end{equation}
Moreover, $C_c^{\infty}(\R^d)$ is dense in any $B^{s}_{p, q} (\R^d)$ and 
\begin{equation} \label{incl} 
B^{s_2}_{p, q}(\R^d) \subset B^{s_1}_{p, q}(\R^d) , \;\;\; 0 < s_1 < s_2 <2,\;\;\; p \ge 2 \, .
\end{equation} 
We also have the following result (cf. Theorem 6.4.4 in \cite{B})
\begin{equation} \label{ma1}
B^{s}_{p, 2} (\R^d) \subset H^s_p(\R^d) \subset B^{s}_{p, p} (\R^d) \, ,
\end{equation} 
$s \in (0,2)$, $p\ge 2$. Next we state  a  result for which we have not  found  a precise reference in the literature. This is useful to give an equivalent formulation to   Hypothesis \ref{hyp-holder} (cf. Remark \ref{in4}).
The proof is given in Appendix. 
\begin{proposition}\label{df3}
Let $p >2$, $s, s' $ such that  $0<s < s' <1$. We have
\begin{equation*}
 W^{s',p} (\R^d)   \subset  B^{s}_{p, 2} (\R^d) \subset  H^s_p(\R^d) \, . 
\end{equation*}
 \end{proposition}
It is important to notice that Besov spaces are real interpolation spaces (for the definition of interpolation spaces  $(X, Y)_{\theta, q}$ with $X$ and $Y$ real Banach spaces and $Y \subset X$ see Chapter 1 in \cite{L} or \cite{B}). As a particular case  of   Theorem 6.2.4 in \cite{B} we have
 for $0\le s_0 < s_1 \le 2 $, $\theta \in (0,1)$, $p \ge 2$,
\begin{equation} \label{ber1}
(H^{s_0}_p (\R^d), H^{s_1}_p (\R^d))_{\theta, p} =  B^{s}_{p, p} (\R^d)
\end{equation} 
 with $s = (1- \theta) s_0 + \theta s_1$.  Moreover, it holds (see Theorem 6.4.5 in \cite{B}): 
\begin{equation} \label{int11}
(B^{s_0}_{p,p} (\R^d), B^{s_1}_{p,p} (\R^d))_{\theta, p} =  B^{s}_{p, p} (\R^d)
\end{equation} 
 with $0< s_0 < s_1<2 $, $s = (1- \theta) s_0 + \theta s_1$, $\theta \in (0,1)$.

\subsection{Interpolation of functions with values in Banach spaces }

We follow Section VII in \cite{LP} and \cite{C}. Let $A_0$ be a real Banach space. We will  consider the Banach space $L^p (\R^d ; A_0)$, $1\le p < \infty$, $d \ge 1$.  As usual this consists of all strongly measurable functions $f$ from  $\R^d$ into $A_0$ such that the  real valued function $\| f(x)\|_{A_0}$ belongs to $L^p(\R^d)$. We have
$$
\| f\|_{L^p (\R^d ; A_0)} = \Big( \int_{\R^d} \| f(x)\|_{A_0} \dd x  \Big)^{1/p}, \;\;\; f \in L^p (\R^d ; A_0) \, .
$$
If  $A_1$ is another real Banach spaces with $A_1 \subset A_0$  we  can define  the Banach space
$$
L^p (\R^d ; (A_0, A_1)_{\theta, q} ) \, ,
$$
by using the interpolation space $ (A_0, A_1)_{\theta, q}$,
 $ q \in (1, \infty)$, $p \ge 1$ and $\theta \in (0,1)$. One can prove that 
\begin{equation} \label{in2}
\big( L^p (\R^d ; A_0),  
L^p (\R^d ;  A_1)
\big)_{\theta, q} = L^p (\R^d ; (A_0, A_1)_{\theta, q} ) \, .
\end{equation} 
with equivalence of norms (see \cite{LP} and \cite{C}).  In the sequel we will often use, for $s\ge 0$,
$p \ge 2$, 
 \begin {equation} \label{hsp}
L^p (\R^d ; H^{s}_p (\R^d)) \, . 
\end{equation}
We will often identify this space with the Banach space
 $L^p\left( \mathbb{R}_{v}^{d};H^{s}_p\left(
\mathbb{R}_{x}^{d}\right) \right)$
 of all measurable functions $f(x,v)$,
$f: \R^d \times \R^d \to \R$ such that $f(\cdot, v) \in H^{s}_p (\R^d)$, for a.e. $v \in \R^d$, and, moreover (see \eqref{st})
\begin{equation} \label{h2}
\int_{\R^d} \|f( \cdot,v) \|_{H^{s}_p}^p \, \dd v
=  \int_{\R^d} \dd v \int_{\R^d} |{\cal F}^{-1}_x [ (1+ |\cdot|^{s}) {\cal F}_x f( \cdot,v) ] (x) |^p \dd x < \infty
\end{equation}
(here ${\cal F}_x$ denotes the partial Fourier transform in the $x$-variable; as usual we identify functions which coincide a.e.).   As a norm we  consider 
\begin{equation} \label{chie}
 \| f\|_{L^p (\R^d_v ; H^{s}_p (\R^d_x))}  = \Big( \int_{\R^d} \|f(\cdot,v) \|_{H^{s}_p}^p \dd v \Big)^{1/p}.
\end{equation} 
Also  $L^p (\R^d_v ; L^p (\R^d_x) )$ can be identified with $L^p(\R^{2d})$. Similarly, we can define   
$
L^p (\R^d_v ; B^{s}_{p,p} (\R^d_x)).
$ 
 Using   \eqref{ma1} we have
\begin{equation} \label{fr3}
L^p (\R^d ; H^{s}_p (\R^d)) \subset L^p (\R^d ; B^{s}_{p,p} (\R^d)),
\end{equation} 
  $p \ge 2$, $0< s<2$. Finally using \eqref{in2} and \eqref{int11} we get for $0<s_0 < s_1<2$, $\theta \in (0,1)$, $p\ge 2$,
\begin{equation} \label{int12}
\big( L^p (\R^d ; (B^{s_0}_{p,p} (\R^d)),  
L^p (\R^d ;  (B^{s_1}_{p,p} (\R^d))
\big)_{\theta, p} = L^p (\R^d ;  B^{s}_{p, p} (\R^d) ),
\end{equation} 
 with $s = (1- \theta) s_0 + \theta s_1$. 

In the sequel when no confusion may arise,  we will simply write 
$L^p(\R^d)$ instead of 
$L^p(\R^d ; \R^k)$, $k \ge 1$, $p \in [1, \infty)$. Thus a function $U : \R^d \to \R^k$ belongs to $L^p(\R^d)$ if all its components $U_i \in L^p(\R^d)$, $i =1, \ldots,k$. Moreover, 
$
\| U\|_{L^p} $ $= \Big  ( \sum_{i=1}^k \| U_i\|_{L^p}^p \Big)^{1/p}.
$ 
This convention about vector-valued functions will be used for other function spaces as well.

 \begin{remark} \label{in4} 
 {\em  Proposition \ref{df3} and formula \eqref{den} show that Hypothesis \ref{hyp-holder} is equivalent to the following one:
$F:\R^{2d}\to \R^d$ is a Borel function such that
\begin {equation} \label{ass1}
 \int_{\R^d} \| F(\cdot , v) \|_{W^{s,p}}^p \dd v < \infty \, ,
\end{equation}
where  $s \in (2/3, 1)$
 and $p > 6d.$ 
}
\end{remark}

\subsection{Regularity results  in Bessel-Sobolev spaces}\label{sec: regularity Bessel}

Here $\R^N = \R^{2d}$ and $z = (x,v) \in \R^d \times \R^d$.  Let also $p \in (1, \infty)$, $s \in (0,1)$ and  $\lambda>0$. This section is devoted to the study of the equation
\begin{gather*}
\lambda {\psi} (z) - \frac{1}{2} \triangle_v {\psi}(z) - v \cdot D_x {\psi} (z) 
 - F(z) \cdot D_v {\psi}(z) =    g(z)    \,    \nonumber \\
\qquad = \lambda {\psi} (z)
  - \frac{1}{2} \mathrm{Tr} \big(Q D^2 {\psi} (z) \big) -   \langle Az, 
D{\psi}(z)\rangle   -  \langle B(z), D{\psi}(z)\rangle      \nonumber 
\end{gather*}
 where $ A = \begin{pmatrix}    	0 & \mathbb{I} \\	0 & 0 
       \end{pmatrix} $, 
 $Q = \begin{pmatrix}  0 & 0 \\ 0 & \mathbb{I} \end{pmatrix} $ are $(2d \times 2d)$-matrices, $B = \begin{pmatrix}    0  \\ F    \end{pmatrix}: \R^{2d} \to \R^{2d} \,$.  We shall start  by considering the simpler equation with $B=0$, i.e.,
\begin{equation} \label{uno}
\lambda {\psi} (z) - \frac{1}{2} \triangle_v {\psi}(z) - v \cdot D_x {\psi} (z) 
= \lambda {\psi} (z) - {\cal L} {\psi}(z) = g(z), \;\;\; z \in \R^{2d}.   
\end{equation}
Recall that  $D_v \psi$ and $D_x \psi$ denote respectively the gradient of $\psi$ in the $v$-variables and in the $x$-variables; moreover, $D^2_v \psi$ indicates the Hessian
matrix of $\psi$ with respect to the $v$-variables (we have $\triangle_v {\psi} = {\text Tr}(D^2_v \psi)$).
\begin{definition}\label{deff}{\em 
The  space $X_{p,s}$ consists of all functions $f \in W^{1,p} (\R^{2d})$     
 such 
that  $D^2_v f $ and $v \cdot D_x  f$
belong to $L^p (\R^d_v ;  H^{s}_p (\R^{d}_x))$. Recall that 
$$
\| D^2_v f \|_{L^p (\R^d_v ;  H^{s}_p (\R^{d}_x))}^p = \int_{\R^d }  \sum_{i,j =1}^d \, \, \|\partial^2_{v_i v_j} f(\cdot,v)\|^p_{ H^{s}_p (\R^{d})} \dd v \, .
$$
}
\end{definition}

It turns out that $X_{p,s}$ is a Banach space endowed with the norm:
\begin{equation}
\label{xpss}
\| f \|_{X_{p,s}} =  \| f\|_{W^{1,p} (\R^{2d})} + \| D^2_v f \|_{L^p (\R^d_v ;  H^{s}_p (\R^{d}_x))} 
 + \| v \cdot D_x f \|_{L^p (\R^d_v ;  H^{s}_p (\R^{d}_x))}. 
\end{equation}
If $f \in X_{p,s}$ then
 $(\lambda f - {\cal L} f) \in L^p (\R^d_v ;  H^{s}_p (\R^{d}_x) )$ (see \eqref{uno}). 
With a slight abuse of notation, we will still write $f\in X_{p,s}$ for vector valued functions $f:\R^{2d} \to \R^{2d}$,  meaning that all components $f_i:\R^{2d}\to \R$, $i=1\dots 2d$ belong to $X_{p,s}$.

The following theorem improves results in  \cite{Bo} and \cite{BCLP}. In 
particular  it shows that there exists the weak derivative $D_x {\psi} \in L^p (\R^{2d}) $  so that 
\eqref{uno} admits a strong solution $\psi$ which solves equation \eqref{uno} in distributional sense.

\begin{theorem} \label{cos} Let $\lambda>0,$ $p \ge 2$, $s \in (1/3, 1)$  and  $g \in L^p (\R^d_v ;  H^{s}_p (\R^{d}_x))$. 
 There exists a unique  solution ${\psi} = {\psi}_{\lambda}\in X_{p,s}$ to 
equation \eqref{uno}. Moreover, we have
\begin{equation} \label{prima}
\lambda \| {\psi} \|_{L^p(\R^{2d})}  + \sqrt{\lambda} \| D_v {\psi}\|_{L^p(\R^{2d})}
 +
 \| D^2_v {\psi} \|_{L^p(\R^{2d})}  + \| v \cdot D_x \psi \|_{L^p (\R^{2d})}
 \le 
C \| g \|_{L^p (\R^{2d})}
\end{equation} 
with 
 $C =C(d,p ) >0$ and 
\begin {eqnarray} \label{prima1}
 \| D_x {\psi} \|_{L^p (\R^{2d})}
\le   C(\lambda) \, \| g \|_{L^p (\R^d_v ;  H^{s}_p (\R^{d}_x))},
\end{eqnarray}
with $C(\lambda) = C(\lambda,  s,p,d) >0$ and $C(\lambda) \to 0 $ as $\lambda \to \infty$. In addition there exists  $c = c(s,p,d) >0$ such that
\begin{gather} \label{serve}
 \lambda \| {\psi}\|_{L^p (\R^d_v ;  H^{s}_p (\R^{d}_x))}+ 
  \sqrt{\lambda} \| D_v {\psi}\|_{L^p (\R^d_v ;  H^{s}_p (\R^{d}_x))}
+ \| D^2_v {\psi} \|_{L^p (\R^d_v ;  H^{s}_p (\R^{d_x}))} 
\\ + \| v \cdot D_x {\psi} \|_{L^p (\R^d_v ;  H^{s}_p (\R^{d}_x))}
 \nonumber \le c \| g \|_{L^p (\R^d_v ;  H^{s}_p (\R^{d}_x) ) }
\end{gather} 
\end{theorem}
\begin{proof} \textit{Uniqueness.} Let ${\psi} \in X_{p,s}$ be a solution. 
Multiplying 
both sides of equation \eqref{uno} by $|{\psi}|^{p-2} {\psi} $ and integrating 
by parts we obtain
\begin{gather*}
 \lambda \| {\psi}\|_{L^p(\R^{2d})}^p +  \frac{(p-1)}{2} \sum_{k=1}^d\int_{\R^{2d}} 
|{\psi}|^{p-2} |\partial_{v_k} {\psi}|^2\,  \dd z + \int_{\R^{2d}} (v \cdot D_x {\psi} ) 
|{\psi}|^{p-2} {\psi} \, \dd z \\ =
 \int_{\R^{2d}} g |{\psi}|^{p-2} {\psi} \, \dd z
\end{gather*}
(this identity can be rigorously proved  by approximating ${\psi}$ by smooth 
functions). Note that there exists the weak derivative
$$
 D_x (|{\psi}|^p) = p |{\psi}|^{p-2} {\psi} D_x {\psi} \in L^{1}(\R^{2d})
$$
and so 
$$
 \int_{\R^{2d}} (v \cdot D_x {\psi} ) |{\psi}|^{p-2} {\psi} \, \dd z = 
 \frac{1}{p}\int_{\R^{2d}} v \cdot D_x  (|{\psi}|^{p})  \dd z =0 \, .
$$
It follows easily that 
\begin{equation} 
\label{dis}
 \| {\psi} \|_{L^p (\R^{2d})} \le \frac{1}{\lambda} \| g\|_{L^p (\R^{2d})}
\end{equation} 
which implies uniqueness of solutions for the linear equation \eqref{uno}.

\smallskip 
\noindent \textit{Existence.} {\it I Step.} We prove existence of solutions and estimates  \eqref{prima} and \eqref{prima1}.

Let us first introduce the  Ornstein-Uhlenbeck semigroup
\begin{align}
 P_t g(z) = P_t g(x,v) &=    \int_{\R^{2d}} g (  e^{tA } z +  y) N(0, Q_t) \, \dd y    \label{OU semigr}  \\
  &= \int_{\R^{2d}} g (  x+ t v + y_1, v+ y_2) N(0, Q_t)  \,  \dd y \, , 
 \;\; g \in C_c^{\infty}(\R^{2d}), \; t \ge 0,  \nonumber
\end{align}
 where 
$N(0, Q_t)$ is the Gaussian measure with mean 0 and covariance matrix 
\begin{equation} \label{dia1}Q_t = \int_0^t e^{sA} Q e^{sA^*} \dd s  = \int_0^t e^{sA} \begin{pmatrix}
 0 & 0\\
 0 & \mathbb{I}_{\R^d}
\end{pmatrix}  e^{sA^*} \dd s \, =  \begin{pmatrix}
 \frac{1}{3}  t^3 \mathbb{I}_{\R^d} &  \frac{1}{2}  t^2  \mathbb{I}_{\R^d}\\[3pt]
 \frac{1}{2}  t^2 \mathbb{I}_{\R^d} &   t \mathbb{I}_{\R^d}
\end{pmatrix} \, 
\end{equation} 
($A^*$ denotes the adjoint matrix). 
By the Young inequality (cf. the proof of Lemma 13 in \cite{P}) we know that $P_t g$  is well-defined also for any $g \in L^p(\R^{2d})$,  $z$ a.e.; moreover $P_t : L^p (\R^{2d}) \to L^p (\R^{2d})$, for any $t \ge 0$, and 
\begin{equation}
\label{e44}
\| P_t g \|_{L^p (\R^{2d})} \le \| g\|_{L^p (\R^{2d})}, \;\; g \in L^p (\R^{2d}), \;  t \ge 0. 
\end{equation}
Let us consider, for any $\lambda >0$, $z \in \R^d$, $g \in C_c^{\infty}(\R^{2d})$, 
\begin{equation} \label{gll}
{\psi} (z)= G_{\lambda}g (z) =\int_{0}^{+\infty}  e^{- \lambda t} P_t g (z) 
\, \dd t \, .
\end{equation}
Using the Jensen inequality, the Fubini theorem and \eqref{e44} it is easy to prove that $G_{\lambda} g$ is well defined for $g \in L^p(\R^{2d})$, $z$ a.e., and belongs to $ L^p(\R^{2d})$. Moreover, for any $p \ge 1$,
\begin {equation} \label{gll1}
G_{\lambda} : L^p(\R^{2d}) \to L^p(\R^{2d}),\;\;\;  \; \| G_{\lambda} g \|_{p} \le \frac{\| g\|_p}{\lambda},\;\; \; \lambda >0,\;\; g \in L^p(\R^{2d}).
\end{equation}
Note that $L^p (\R^d_v ;  H^{s}_p (\R^{d_x}) ) \subset L^p (\R^d_v ;  W^{s,p} (\R^{d}_x) )$ (see \eqref{den}).
Let us consider a sequence $(g_{n}) \in C^{\infty}_c(\R^{2d}) $ such that 
$$
g_n \to g  \; \text{in} \; L^p (\R^d_v ;  W^{s,p} (\R^{d}_x) ). 
$$ 
Arguing as in \cite[Lemma 13]{P} one can show that there 
exist classical solutions ${\psi}_n$ to \eqref{uno} with $g$ replaced by $g_n$. 
Moreover, $\psi_n = G_{\lambda } g_n$.
By \cite[Theorem 11]{P}, which is based on results in \cite{BCLP}, we have that 
 \begin{equation} \label{cf5}
\| D_v^2 {\psi}_n \|_{L^p(\R^{2d})} \le C  \| g_n\|_{L^p(\R^{2d})} ,
\end{equation}
 $\lambda>0$, $n \ge 1$, $C = C(p,d)$.  
 Using also \eqref{gll1}
we deduce easily that $({\psi}_n )$ 
 and $(D^2_v {\psi}_n) $ are both Cauchy sequences in $L^p (\R^{2d})$.
Let us denote by ${\psi}  \in L^p (\R^{2d})$
the  limit function; it holds that $\psi = G_{\lambda} g$
and 
$D^2_v {\psi} \in L^p (\R^{2d})$.

Passing to the limit in \eqref{uno} when  ${\psi}$ and $g$ are replaced by 
${\psi}_n$ and $g_n$ we obtain that ${\psi}$ solves \eqref{uno}
 in a weak sense
 ($v \cdot D_x {\psi}$ is intended  in distributional sense). 
By \eqref{cf5} as $n \to \infty$ we also get 
 \begin{gather} \label{f23} 
\| D_v^2 {\psi} \|_{L^p(\R^{2d})} \le C  \| g\|_{L^p(\R^{2d})} ,\;\;\; 
 \| {\psi} \|_{L^p (\R^{2d})} \le \frac{1}{\lambda}  \| g\|_{L^p(\R^{2d})}
\\ \nonumber \text{and}
\;\;\;
 \| v \cdot D_x {\psi}\|_{L^p(\R^{2d})} \le C \,  \| g\|_{L^p(\R^{2d})}.
\end{gather}
To prove \eqref{prima} it remains to show the estimate for $D_v {\psi}$. This 
follows from  
 \begin{equation} \label{by}
\| D_v {\psi}\|_{L^p(\R^{2d})}^p = \int_{\R^d} \| D_v \psi (x, \cdot)\|_{L^p(\R^d)}^p \dd x  
 \le  (\|  {\psi}\|_{L^p(\R^{2d})} )^{p/2} \, ( \| 
D_v^2 {\psi}\|_{L^p(\R^{2d})} )^{p/2}. 
\end{equation}
To prove that $\psi \in W^{1,p} (\R^{2d}) $ it is enough to check  that  
\begin{equation} \label{sol}
{\psi} \in L^p \big( \R^d_v ; W^{1,p} (\R^d_x) \big).
\end{equation} 
Thus we have to prove that   ${\psi}(\cdot , v) \in W^{1,p} (\R^d) $ for a.e. $v$ and 
$$ 
\int_{\R^d} \dd v \int_{\R^d} |D_x {\psi}(x,v)|^p \, \dd x < \infty \, .
$$
To this purpose we will use a result in \cite{Bo} and interpolation theory. We consider $\eta \in C_c^{\infty}(\R)$ such that Supp$(\eta) \subset [-1,1]$ and $\int_{-1}^{1} \eta(t) \dd t >0$. 

Setting $f (t,z) = \eta(t) \psi(z)$, where $\psi$ solves \eqref{uno} we have that  
$f \in L^{p}(\R \times \R^d \times \R^d) $. In order to apply  Corollary 2.2 in \cite{Bo} we  note that, for  $z=(x,v) \in \R^{2d},$ $t \in \R$,
$$
\partial_t f(t,z) +  v \cdot D_x {f} (t,z) = \eta'(t) \psi(z)
- \eta(t) g(z) + \lambda \eta(t){\psi} (z) - \frac{1}{2} \eta(t)\triangle_v 
{\psi}(z). 
$$
Since $D^2_v {\psi} \in  L^p(\R^{2d})$  we deduce that $ \partial_t f +  v \cdot D_x {f}  $ and $D^2_v {f} $ both
belong to $L^{p}(\R \times \R^d \times \R^d)$.

 By Corollary 2.2 in \cite{Bo} and \eqref{f23} we get easily that
${\psi}(\cdot ,v ) \in H^{2/3}_p (\R^d)$, for  $v \in \R^d$ a.e., and 
\begin{equation*}
 \int_{\R^d} \dd v \int_{\R^d} |{\cal F}^{-1}_x [ (1+ |\cdot|^{2/3}) {\cal F}_x 
{\psi} (\cdot, v) ] (x) |^p \, \dd x \le  \Big( \frac{\lambda +1}{\lambda} \Big)^{2p/5}\, c\,  \| g\|_{L^p (\R^{2d})}^p \, ,
\end{equation*} 
$\lambda>0$, with $c= c(p,d)$, i.e.,
\begin{equation} \label{frrr}
{\psi} = G_{\lambda} g \in L^p (\R^d_v ; H^{2/3}_p (\R^d_x)) 
\;\; \text{  and} 
\;\;
 \| G_{\lambda} g \|_{L^p (\R^d_v ; H^{2/3}_p (\R^d_x) )} \; \le \; 
 \Big( \frac{\lambda +1}{\lambda} \Big)^{2/5}\, c_1 \, \| g\|_{L^p(\R^{2d})}.
\end{equation}
By  \eqref{in2}  and \eqref{ber1} with $s_0=0$ and $s_1 = 2/3$ we can interpolate between \eqref{frrr} and the estimate $\| G_{\lambda} g \|_{L^p(\R^d; L^p(\R^d))}  \le 
 \frac{1}{\lambda}\, \| g\|_{L^p(\R^{2d})}$
 (see Proposition 1.2.6 in \cite{L}) and get, for $\epsilon \in (0, 2/3)$, 
\begin{equation}
\label{it5}
\| G_{\lambda} g \|_{ L^p (\R^d_v ; W^{2/3 - \epsilon, p} (\R^d_x))}  \le c_{\epsilon}(\lambda)  \| g\|_{L^p(\R^{2d})}.
\end{equation}   
with $c_{\epsilon}(\lambda) \to 0$ as $\lambda \to \infty$.

Suppose  now that $g \in L^p (\R^d_v ; W^{1, p} (\R^d_x))$ and fix $k =1, \ldots, d$. By approximating $g$ with regular functions, it is not difficult to prove that 
there exists the weak derivative $\partial_{x_k} {\psi}
\in L^p (\R^{2d})$, 
and 
\begin{equation} \label{ct5}
\partial_{x_k} {\psi} (z)= \partial_{x_k} G_{\lambda}g (z) =\int_{0}^{+\infty} e^{- \lambda 
t}  P_t (\partial_{x_k} g)(z) \dd t \, .
\end{equation}
Arguing as in \eqref{it5}  we obtain that  $\partial_{x_k} {\psi} \in L^p (\R^d_v ; W^{2/3 
- \epsilon, p} (\R^d_x))$ and 
\begin{equation*}
\| \partial_{x_k} {\psi} \|_{L^p (\R^d_v ; W^{2/3 - \epsilon, p} (\R^d_x))} \le  
c_{\epsilon}(\lambda) \| \partial_{x_k} g \|_{L^p (\R^{2d})},\;\;\; k = 1, \ldots, d \, ,
\end{equation*} 
so that
\begin{equation}
\label{cia4}
\|  G_{\lambda} g \|_{L^p (\R^d_v ; B^{1+ 2/3 - \epsilon, p}_{p,p} (\R^d_x))} \le  c_{ \epsilon} (\lambda) \| g \|_{ 
L^p (\R^d_v ; W^{1,p} (\R^d_x))}.
\end{equation} 
Taking into account  \eqref{ber1}, \eqref{int11} and \eqref{in2} we can interpolate  between \eqref{it5} and \eqref{cia4} (see also \eqref{hsp} and 
 \eqref{h2}) 
and 
 get 
\begin{gather} \label{ne1ee}
G_{\lambda} :  L^p (\R^d_v ; W^{s, p} (\R^d_x)) =
\big( L^p (\R^d ;  H^{0}_{p} (\R^d)),  L^p (\R^d ;   W^{1,p} (\R^d))
\big)_{s, p}\hspace{20mm} \\  \nonumber 
\hspace{20mm}  \longrightarrow  
 \, \big( L^p (\R^d ;  W^{2/3 - \epsilon,  p} (\R^d) ),
 L^p (\R^d ;  B^{ 5/3 -\epsilon}_ {p,  p} (\R^d) ))_{s,p} =
 L^p (\R^d_v ;  B^{s + 2/3 - \epsilon}_ {p,  p} (\R^d_x)).
\end{gather}
Since $L^p (\R^d_v ;  B^{s + 2/3 - \epsilon}_ {p,  p} (\R^d_x)) \subset  L^p (\R^d_v ;  
W^{1,  p} (\R^d_x))$ with $\epsilon$ small enough (recall that $s \in (1/3,1)$) we finally obtain that
\begin{equation}
\label{dd2}
G_{\lambda} :   L^p (\R^d_v ; W^{s, p} (\R^d_x))  \to
 L^p (\R^d_v ;  W^{1,p} (\R^d_x))
\end{equation}
 is linear and continuous. 
 Moreover, we have  with $\psi = G_{\lambda} g$
\begin{equation} \label{serve4}
\int_{\R^d}  \dd v \int_{\R^d} | \partial_{x_k} {\psi} (x,v)|^p \, \dd x \le 
 C'(\lambda) \, \| g\|^p_{ L^p (\R^d_v ; W^{s, p} (\R^d_x))}
 \le   C''(\lambda) \,  \| g\|^p_{ L^p (\R^d_v ; H^{s}_p (\R^d_x))} \, ,
\end{equation} 
$ k =1, \ldots, d,$ where $C'(\lambda)$ and $C''(\lambda)$ 
tend to 0 
as $\lambda \to \infty$ (recall the estimates 
\eqref{it5} and \eqref{cia4}). This proves \eqref{prima1} and \eqref{sol}.

\smallskip
\noindent {\it II Step.} We prove   the last assertion \eqref{serve}.  

The main problem is to show that 
\begin{gather*}
 \int_{\R^d}  \| D_v^2  {\psi} (\cdot, v) \|_{H^s_p (\R^{d})}^p \dd v =
 \int_{\R^d} \dd v \int_{\R^d} |{\cal F}^{-1}_x [ (1+ |\cdot|^{s}) {\cal F}_x 
D_v^2 {\psi} (\cdot, v) ] (x) |^p \, \dd x
\\ 
=  \int_{\R^d} \dd v \int_{\R^d} |D_v^2 \, \big( {\cal F}^{-1}_x [ (1+ |\cdot|^{s}) {\cal F}_x 
 {\psi} (\cdot, v) ]  \big) (x) |^p \, \dd x \le C  \| g\|^p_{ L^p (\R^d_v ; H^{s}_p 
(\R^d_x))}
\end{gather*}
 (${\cal F}_x$ denotes the Fourier transform in the $x$-variable) with ${\psi} = G_{\lambda }g$. We introduce 
$$
h_s(x,v) = {\cal F}^{-1}_x [ (1+ |\cdot|^{s}) {\cal F}_x 
 g (\cdot, v) ] (x),$$ $x, v \in \R^d$ . We know that $h_s  \in L^p (\R^{2d})$ by our hypothesis on $g$. 
A straightforward computation based on the Fubini theorem shows that  
$$
{\cal F}^{-1}_x [ (1+ |\cdot|^{s}) {\cal F}_x 
 {\psi} (\cdot, v) ] (x) = G_{\lambda} 
h_s (x,v).
$$ 
By using \eqref{f23} (with $g$ replaced by 
$ h_s$ and ${\psi} $ by $G_{\lambda } h_s$) we easily obtain that 
\begin{gather} \label{ne1}
\int_{\R^d}  \| D_v^2  {\psi} (\cdot, v) \|_{H^s_p (\R^{d})}^p \dd v 
 =\int_{\R^d} dv \int_{\R^d} |D_v^2 \, G_{\lambda} h_s (x,v)
 |^p \, \dd x \\
\nonumber
\le C  
\| h_s\|_{L^p(\R^{2d}) }^p
=  C \| g \|_{L^p (\R^d_v ; H^{s}_p (\R^d_x)) }^p,
\end{gather}
where $C = C(d,p)$. Similarly, we have 
\begin{gather} \label{ne11}
\int_{\R^d}  \| D_v  {\psi} (\cdot, v) \|_{H^s_p (\R^{d})}^p \dd v 
 =\int_{\R^d} \dd v \int_{\R^d} |D_v \, G_{\lambda} h_s (x,v)
 |^p \, \dd x 
\\
\nonumber \le \frac{C}{(\lambda)^{p/2} }   
\| h_s\|_{L^p(\R^{2d}) }^p
=  \frac{C}{(\lambda)^{p/2} } \| g \|_{L^p (\R^d_v ; H^{s}_p (\R^d_x)) }^p
\end{gather}
and 
$\| {\psi} \|_{{L^p (\R^d_v ; H^{s}_p (\R^d_x)) }} =
 \| G_{\lambda}  h_s\|_{L^p (\R^{2d})}  \le 
\frac{1}{\lambda} \| G_{\lambda}  h_s\|_{L^p (\R^{2d})} = 
\frac{1}{\lambda} \| g\|_{{L^p (\R^d_v ; H^{s}_p (\R^d_x)) }}.
$
The proof is complete.
\end{proof}

\begin{lemma} \label{neu} Assume as in  Theorem \ref{cos} 
that $g \in   L^p (\R^d_v ; H^{s}_p (\R^d_x))$, $s \in (1/3,1)$. Moreover, 
 suppose that $p>d$.
 Then the solution ${\psi} = G_{\lambda} g $ to \eqref{uno} verifies 
 also
\begin{equation}
\label{nhtt}
 \sup_{v \in \R^d}   \|D_v  \psi (\cdot, v) \|_{ H^{s}_ {p} 
(\R^d)} \le C(\lambda) \| g\|_{ L^p (\R^d_v ; H^{s}_p (\R^d_x))}, \;\; \lambda>0, 
\end{equation} 
 where 
$C(\lambda) \to 0 $ as $\lambda \to \infty$.
 \end{lemma}
\begin{proof} Using the notation introduced in the previous proof, we have 
for any $v \in \R^d$, a.e.,
\begin{gather*}
\| D_v  {\psi} (\cdot, v) \|_{H^s_p (\R^{d})}^p 
 = \int_{\R^d} |D_v \, G_{\lambda} h_s (x,v)
 |^p \, \dd x \, .
\end{gather*}
By \eqref{ne1}, \eqref{ne11} and the Fubini theorem we know that 
$$
\int_{\R^d} \dd x \int_{\R^d} |D_v \, G_{\lambda} h_s (x,v)
 |^p \, \dd v +  \int_{\R^d} \dd x \int_{\R^d} |D_v^2 \, G_{\lambda} h_s (x,v)
 |^p \, \dd v < \infty \, .
$$
It follows that, for  $x \in \R^d$ a.e., 
$$
 \| D_v \, G_{\lambda} h_s (x, \cdot)\|^p_{W^{1,p}(\R^d)}=
\int_{\R^d} |D_v \, G_{\lambda} h_s (x,v)
 |^p \, \dd v + \int_{\R^d} |D_v^2 \, G_{\lambda} h_s (x,v)
 |^p \, \dd v < \infty \, .
$$
 In order to prove \eqref{nhtt} with $C(\lambda) \to 0$,
we  consider $r \in (0,1)$ such that $rp >d$.
 Let us  fix $x \in \R^d$, a.e.; 
by the previous estimate  the mapping $ v \mapsto D_v \, G_{\lambda} h_s (x,v)$ belongs to $W^{r,p}(\R^d) \subset W^{1,p}(\R^d)$. 

We can apply  the Sobolev embedding theorem (see page 203 in \cite{T}) and get that $ v \mapsto D_v \, G_{\lambda} h_s (x, v)$ in particular  is bounded and continuous   on $\R^d$. Moreover, 
\begin{equation}
\label{emb1}
 \sup_{v \in \R^d} |D_v \, G_{\lambda} h_s (x, v)|^p \le c \;
  \| D_v \, G_{\lambda} h_s (x, \cdot)\|^p_{W^{r,p}(\R^d)},
\end{equation} 
where $c = c(p,d, r)$. Integrating with respect to $x$ we get 
$$
\int_{\R^d} \big[ \sup_{v \in \R^d} |D_v \, G_{\lambda} h_s (x, v)|^p \big] \, \dd x \, \le c  \int_{\R^d}  
  \| D_v \, G_{\lambda} h_s (x, \cdot)\|^p_{W^{r,p}(\R^d)} \dd x \, .
$$
By \eqref{ber1} we know that $ \big( L^p (\R^d), W^{1,p} (\R^d) \big)_{r, p} =  W^{r, p} (\R^d)$. Applying Corollary 1.2.7 in \cite{L} we obtain that, for any $f  \in W^{1, p} (\R^d),$
$$
\| f\|_{W^{r, p} (\R^d)} \le c(r,p) \,  \big (\| f\|_{L^{p} (\R^d)} \big)^{1-r} \cdot 
 \big(\| f\|_{W^{1, p} (\R^d)} \big)^{r}. 
$$
 It follows that 
\begin{gather*}
\sup_{v \in \R^d} \| D_v  {\psi} (\cdot, v) \|_{H^s_p (\R^{d})}^p 
 = \sup_{v \in \R^d} \int_{\R^d} |D_v \, G_{\lambda} h_s (x,v)
 |^p \, \dd x 
\\
\le \int_{\R^d} \big[ \sup_{v \in \R^d} |D_v \, G_{\lambda} h_s (x,v)
 |^p \big] \, \dd x     \le c  \int_{\R^d}  
  \| D_v \, G_{\lambda} h_s (x, \cdot)\|^p_{W^{r,p}(\R^d)} \dd x
\\ \le 
c' \;  \int_{\R^d}  
  \| D_v \, G_{\lambda} h_s (x, \cdot)\|_{L^{p}(\R^d)}^{p(1-r)} \cdot   \| D_v \, G_{\lambda} h_s (x, \cdot)
\|^{pr}_{W^{1,p} (\R^d)} \dd x
\\
\le c'  \Big ( \int_{\R^d}  
  \| D_v \, G_{\lambda} h_s (x, \cdot)\|_{L^{p}(\R^d)}^{p} \dd x \Big)^{1-r} \cdot  \Big(\int_{\R^d} \| D_v \, G_{\lambda} h_s (x, \cdot)
\|^{p}_{W^{1,p} (\R^d)} \dd x \Big)^{r}. 
\end{gather*}
Now we easily obtain \eqref{nhtt} using \eqref{ne1} and \eqref{ne11}, since $$\int_{\R^d}  
  \| D_v \, G_{\lambda} h_s (x, \cdot)\|_{L^{p}(\R^d)}^{p} \dd x =
 \int_{\R^d} dv \int_{\R^d} |D_v \, G_{\lambda} h_s (x,v)
 |^p \, \dd x \le C(\lambda) \, \| g \|_{L^p (\R^d_v ;  H^{s}_ {p} 
(\R^d_x))}
$$
with $C(\lambda) \to 0$ as $\lambda \to \infty$.
\end{proof}

\smallskip
We complete the  study of  the regularity of solutions to  equation \eqref{uno} with the next result in which we strengthen the assumptions of Lemma \ref{neu}.
Note that the next assumption on $p$ holds when  $p > 6d$ as in Hypothesis \ref{hyp-holder}.

\begin{lemma}  \label{corsa} 
 Let $\lambda>0,$ $s \in (2/3, 1)$  and  $g \in L^p (\R^d_v ;  H^{s}_p (\R^{d}_x))$. In addition assume  that  
$p (s- \frac{1}{3}) >2d$, then the following statements hold.

(i) The  solution ${\psi} = G_{\lambda } g$ (see \eqref{gll1})
is bounded  and Lipschitz continuous on $\R^{2d}$.
Moreover there exists the classical derivative $D_x {\psi}$ 
which is continuous and bounded on $\R^{2d}$ and, for $\lambda>0,$ 
 \begin{equation} \label{lin}
\| \psi \|_\infty + \| D_x {\psi} \|_{\infty} \le C(\lambda) \| g \|_{L^p (\R^d_v ;  H^{s}_ {p} 
(\R^d_x))}, \;\;\;\;\;  \text{with $C(\lambda ) \to 0 $ as $\lambda \to \infty$. }
\end{equation} 
(ii)   $D_v  {\psi} \in W^{1,p}(\R^{2d}) $ (so in particular there exist weak partial derivatives $\partial_{x_i} \partial_{v_j} \psi \in L^p(\R^{2d}) $, $i,j =1, \ldots, d$) and 
\begin{equation}
\label{es4}
\| D_v  {\psi} \|_{ W^{1,p}(\R^{2d})} \le c(\lambda)  \| g \|_{L^p (\R^d_v ;  H^{s}_ {p} 
(\R^d_x))},\;\;\; \lambda>0,\;\; \text{with} \; \, c= c(\lambda) \to 0  \;\; \text{as $\lambda \to \infty$}.
\end{equation} 
\end{lemma}
\begin{proof} (i) The 
 boundedness of $\psi $ follows easily from  estimates \eqref{prima} and \eqref{prima1}  using the Sobolev embedding since in our case $p > 2d$. Let us concentrate on proving the Lipschitz continuity.

First we recall a Fubini type theorem for fractional Sobolev spaces   (see \cite{Str}):
\begin{equation}
\label{d11}
W^{\gamma, p } (\R^{2d})
 = \Big \{  f \in L^p (\R^{2d}):  \int_{\R^d} \| f( x, \cdot) 
\|_{W^{\gamma,p} (\R^d)) }^p \dd x 
 + \int_{\R^d} \| f(  \cdot,v) \|_{W^{\gamma,p} (\R^d)) }^p \dd v 
 < \infty \Big\},
\end{equation}
 $\gamma \in (0,1]$ 
(with equivalence of the respective norms). 
 Let $\eta \in (0, s+ 2/3 -1)$ be such that 
\begin{equation}
\label{et1}
\eta p >2d \, .
\end{equation} 
We will prove that $ D_x {\psi} \in {W^{\eta,p} (\R^{2d})} $ so that by 
 the Sobolev embedding ${W^{\eta,p} (\R^{2d})} \subset C^{\eta - 2d/p}_b (\R^{2d})$ (see  page 203 in \cite{T}) we  get the assertion.
According to \eqref{d11} we check that
\begin{equation} \label{p9}
 \int_{\R^d} \| D_x {\psi}( \cdot,v) \|_{W^{\eta,p} (\R^d)) }^p \dd v 
 \le  C(\lambda) \| g \|_{L^p (\R^d_v ;  H^{s}_ {p} (\R^d_x))}^p \, ,
\end{equation}
and 
\begin{equation} \label{p10}
\int_{\R^d} \| D_x {\psi}( x, \cdot) 
\|_{W^{\eta,p} (\R^d)) }^p \dd x 
 \le  C(\lambda) \| g \|_{L^p (\R^d_v ;  H^{s}_ {p} (\R^d_x))}^p \, .
\end{equation}
with $C(\lambda) \to 0$ as $\lambda \to \infty$. 
Estimate  \eqref{p9} follows  by \eqref{ne1ee} which gives   $\psi \in L^p (\R^d_v ;  B^{\eta +1}_ {p,  p} (\R^d_x)) $ with $\eta = s- \epsilon - 1/3$.

Let us concentrate on \eqref{p10}. We still use the interpolation theory results of Section 3.2 but  here in addition to 
 \eqref{h2} we also need to identify  
$
L^p (\R^d ; H^{s}_p (\R^d))
$
  with   the Banach space $L^p (\R^d_x ; H^{s}_p (\R^d_v))$ of all measurable functions $f(x,v)$,
$f: \R^d \times \R^d \to \R$ such that $f(x, \cdot) \in H^{s}_p (\R^d)$, for  $x \in \R^d$ a.e., $ 0 \le s \le 2,$ and, moreover 
 $
\int_{\R^d} \|f(x,\cdot) \|_{H^{s}_p}^p \, \dd x < \infty$. 
 As a norm one  considers 
\begin{equation} \label{dett}
 \| f\|_{L^p (\R^d_x ; H^{s}_p (\R^d_v))}  = \Big( \int_{\R^d} \|f(x,\cdot) \|_{H^{s}_p(\R^d)}^p \dd x \Big)^{1/p}.
\end{equation} 
Similarly, we identify  $L^p (\R^d ; B^{s}_{p,p} (\R^d))$ 
with the Banach space $L^p (\R^d_x ; B^{s}_{p,p} (\R^d_v))$ of all measurable functions 
$f: \R^d \times \R^d \to \R$ such that $f(x, \cdot) \in B^{s}_{p,p} (\R^d)$, for  $x $ a.e.,  and 
 $ \int_{\R^d} \|f(x,\cdot) \|_{B^{s}_{p,p} (\R^d)}^p \, \dd x < \infty$. 
 By \eqref{prima} and \eqref{prima1} in  Theorem \ref{cos} and  using 
 \eqref{ct5} we find with $\psi = G_{\lambda } g$ 
\begin{align} \label{pr34}
\nonumber
\int_{\R^d} \dd x \int_{\R^d} | D_x  \psi(x,v) |^p \, \dd v  =
 \int_{\R^d} \| D_x  \psi(x, \cdot) \|_{L^p(\R^d)}^p \dd x \le C(\lambda)
  \| g \|_{L^p (\R^d_v ;  H^{s}_ {p} (\R^d_x))}^p,
\\
\int_{\R^d} \dd x \int_{\R^d} | D_v^2 (D_x  \psi)(x,v) |^p \, \dd v  =
\int_{\R^d} \| D^2_v (D_x \psi) (x, \cdot) \|_{L^{p}(\R^d)}^p \dd x \le c
  \| g \|_{L^p (\R^d_v ;  H_{p}^{1} (\R^d_x))}^p,
\end{align} 
with $c= c(d,p) >0$. Thus  we can consider the following linear maps 
($s' \in  (1/3,1)$ will be fixed below)
\begin{gather} \label{ins}
D_x G_{\lambda}: L^p \big(\R^d_v ;  H^{s'}_ {p} (\R^d_x)\big) \to  L^p(\R^{2d})= 
L^p\big(\R^d_x; L^p(\R^d_v)\big), 
\\ \nonumber
D_x G_{\lambda}: L^p \big(\R^d_v ;  H^{1}_ {p} (\R^d_x)\big) \to   
L^p\big(\R^d_x; H^{2}_{p}(\R^d_v)\big). 
\end{gather}
Interpolating,  choosing $s' \in (1/3,1)$ such that 
$$
s' < 2s -1,
$$
we get  (see \eqref{ber1} and \eqref{in2} with $\theta = \frac{s-s'}{1-s'} >1/2$)
\begin{gather} \label{ric0}
D_x G_{\lambda} :  L^p \big(\R^d_v ; W^{s, p} (\R^d_x)\big) =
\Big( L^p \big(\R^d_v ;  H^{s'}_{p} (\R^d_x)\big) \, , \, L^p \big(\R^d_v ;   H^{1}_p (\R^d_x)\big)
\Big)_{\theta, p}\\ \nonumber    \longrightarrow   \Big( L^p \big(\R^d_x ;  L^{  p} (\R^d_v) \big) \, , \, L^p \big(\R^d_x ;  H^2_p(\R^d_v) \big) \Big)_{\theta,p} = L^p \big(\R^d_x ; B^{2 \theta}_{p,p} (\R^d_v)\big)
\end{gather}
and  by  the  estimates in \eqref{pr34} we find
$$
 \int_{\R^d} \| D_x  (G_{\lambda} g)(x, \cdot) \|_{B^{2 \theta}_{p,p} (\R^d)}^p \dd x \le C'(\lambda)
  \| g \|_{L^p (\R^d_v ;  H^{s}_ {p} (\R^d_x))}^p
$$
(recall that  $H^{s }_ {p} (\R^d) \subset W^{s, p} (\R^d)$).
 Since $\eta < 2/3$ we have $ B^{2\theta}_{p,p} (\R^d) \subset W^{\eta, p} (\R^d)   $ (cf. \eqref{incl}) and    we finally get \eqref{p10}.

\smallskip 
\noindent (ii)  We fix $j = 1, \ldots, d$
 and prove the assertion  with $D_v \psi$ replaced by $\partial_{v_j} \psi$.

By Theorem  \ref{cos} we already know that there exists $D_v \partial_{v_j} \psi \in L^p (\R^{2d})$. 
  Therefore to show the 
assertion it is enough to check that there exists the weak derivative 
\begin{equation}
\label{es5}
 D_x (\partial_{v_j} \psi) = \partial_{v_j} (D_x {\psi}) \in L^p(\R^{2d}).
\end{equation} 
 We use again \eqref{ric0} with the same $\theta$.  
Since $2 \theta >1$ we know in particular that $D_x G_{\lambda} g \in  L^p (\R^d_x ; W^{1, p} (\R^d_v))$. Thus we  have that there exists 
the weak derivative $\partial_{v_j} D_x {\psi} (x, \cdot) $, for $x$ a.e., and  
\begin{equation} \label{f55}
 \int_{\R^d} \dd x \int_{\R^d} |\partial_{v_j} D_x {\psi} (x, v)|^p \, \dd v
=\int_{\R^d} \| \partial_{v_j} D_x {\psi} (x, \cdot) \|_{L^p (\R^{d})}^p \dd x \le C(\lambda) \,  \| g \|_{L^p (\R^d_v ;  H^{s}_ {p} (\R^d_x))}^p.
\end{equation}
This finishes the proof. 
 \end{proof}

\smallskip 
Now we study the complete  equation
\begin{equation} \label{uno1}
\lambda {\psi} (z) - \frac{1}{2} \triangle_v {\psi}(z) - v \cdot D_x {\psi} (z) 
 - F(z) \cdot D_v {\psi}(z) = g(z), \;\;\; z = (x,v) \in \R^{2d},  
\end{equation}
assuming that $F \in L^p (\R^d_v ; H^{s}_p (\R^d_x))$ (cf. \eqref{h2} and \eqref{chie}).  
From the previous results we obtain (see also Definition \ref{deff})

\begin{theorem} \label{PDE-1!} Let  $s \in (2/3, 1)$
 and $p $ be such that 
 $p (s- \frac{1}{3}) >2d$.  
Assume that   
$$
g, \, F \in  L^p (\R^d_v ;  H^{s}_p (\R^{d}_x)).
$$
Then there exists $\lambda_0 = \lambda_0 ( s,p,d, \|F\|_{L^p (\R^d_v ;  H^{s}_p (\R^d_x)}) >0$ such that for any 
$\lambda>\lambda_0$ there exists a unique  solution ${\psi}
=  \psi_{\lambda} \in X_{p,s}$ 
to  \eqref{uno1} and moreover  
\begin{gather} \label{serve2}
 \lambda \| {\psi}\|_{L^p (\R^d_v ;  H^{s}_p (\R^{d}_x))}+ 
  \sqrt{\lambda} \| D_v {\psi}\|_{L^p (\R^d_v ;  H^{s}_p (\R^{d}_x))}
+ \| D^2_v {\psi} \|_{L^p (\R^d_v ;  H^{s}_p (\R^{d}_x))} 
\\ + \| v \cdot D_x {\psi} \|_{L^p (\R^d_v ;  H^{s}_p (\R^{d}_x))}
 \nonumber \le C \| g \|_{L^p (\R^d_v ;  H^{s}_p (\R^{d}_x) ) }
\end{gather}  
with ${C = C(s,p,d, \|F\|_{L^p (\R^d_v ;  H^{s}_p (\R^{d}_x)}) >0}$.
 We also have
\begin{equation}
\label{nht}
 \sup_{v \in \R^d}   \|D_v  \psi (\cdot, v) \|_{ H^{s}_ {p} 
(\R^d)} \le C(\lambda) \| g\|_{ L^p (\R^d_v ; H^{s}_p (\R^d_x))}, \;\;
 \text{with $C(\lambda) \to 0 $ as $\lambda \to \infty$}.
\end{equation} 
 Moreover,  ${\psi} \in C^1_b(\R^{2d})$, i.e.,  $\psi $ is bounded on $\R^{2d}$ and 
 there exist the classical derivatives $D_x {\psi}$ and $D_{v} \psi$
which are bounded and continuous  on $\R^{2d}$; we  also have   
 with  $C(\lambda ) \to 0 $ as $\lambda \to \infty$
 \begin{equation} \label{lin2}
\| \psi\|_\infty + \| D_x {\psi} \|_{L^p (\R^{2d})} + \| D_x {\psi} \|_\infty+ \| D_v {\psi} \|_\infty \le C(\lambda) \| g \|_{L^p (\R^d_v ;  H^{s}_ {p} 
(\R^d_x))} . 
\end{equation} 
Finally,  $D_v  {\psi} \in W^{1,p}(\R^{2d}) $ and
\begin{equation}
\label{es41}
\| D_v  {\psi} \|_{ W^{1,p}(\R^{2d})} \le c(\lambda)  \| g \|_{L^p (\R^d_v ;  H^{s}_ {p} 
(\R^d_x))}, \;\;\; \text{  $c= c(\lambda) \to 0$
 as $\lambda \to \infty$.}
\end{equation} 
\end{theorem}

\begin{proof} 
First note that, since $p>2d$, the boundedness of $\psi$ follows by  
 the Sobolev embedding (recall also \eqref{xpss}). Similarly  the second estimate in \eqref{lin2} follows from \eqref{es41}.

 We consider the Banach space $Y = L^p (\R^d_v ;  H^{s}_p (\R^{d}_x))$  and  use an argument  similar to  the one used in the proof of Proposition 5 in \cite{DFPR}. 
 Introduce the operator $T_{\lambda} : Y \to Y$,
$$ 
T_\lambda f := F \cdot  D_v (G_{\lambda} f), \;\; f \in Y,
$$
where $G_{\lambda}$ is defined  in \eqref{gll}.  It is not difficult to check  that $T_{\lambda} f \in Y$ for $f \in Y$. Indeed  
by Lemma \ref{neu} we get
\begin{align*}
 \int_{\R^d} \|  T_\lambda f (\cdot, v) \|_{H^{s}_p (\R^{d})}^p \dd v   
&\le \sup_{v \in \R^d}\| D_v (G_{\lambda} f)(\cdot, v)\|_{H^{s}_p (\R^{d})}^p \int_{\R^d} \|  F  (\cdot, v) \|_{H^{s}_p (\R^{d})}^p \dd v \\
&\le C(\lambda) 
\| f \|_{L^p (\R^d_v ;  H^{s}_ {p} 
(\R^d_x))}^p \| F \|_{L^p (\R^d_v ;  H^{s}_ {p} 
(\R^d_x))}^p. 
\end{align*}
 It is clear that $T_{\lambda}  $ is linear and bounded. 
 Moreover   we find easily that there exists $\lambda_0 >0$ such that for any $ \lambda > \lambda_0$ we have that 
the operator norm of $T_{\lambda}$ is less than $1/2$.

Let us fix $\lambda > \lambda_0$. Since $T_{\lambda}$ is a strict contraction, there exists a unique solution $f \in Y$ to 
\begin{equation} \label{fix}
 f - T_{\lambda} f =g.
\end{equation}   
We write $f = (\mathbb{I} - T_{\lambda})^{-1}g \in Y$.

\noindent \textit{Uniqueness}. Let ${\psi}_1$ and $\psi_2$ be solutions in 
$X_{p,s}$.
 Set $w = \psi_1 - \psi_2$.  We know that 
$$
 \lambda w (z) - \frac{1}{2} \triangle_v w(z) - v \cdot D_x w (z) 
 - F(z) \cdot D_v w(z) = 0.
$$
We have $  \lambda w  - \frac{1}{2} \triangle_v w - v \cdot D_x w  = f \in Y$. By  uniqueness (see Theorem \ref{cos}) we get that 
$w = G_{\lambda} f $. Hence, for $z$ a.e.,
$$
0=  f (z) - F(z) \cdot D_v w(z)  =  f(z)  - F(z) \cdot D_v G_{\lambda} f(z).  
$$ 
Since $T_{\lambda}$ is a strict contraction we obtain that $f=0$ 
and so $\psi_1 = \psi_2$.

\noindent \textit{Existence}. It is not difficult to prove that 
\begin{equation} \label{form1}
 {\psi} = {\psi}_{\lambda}= G_{\lambda} (\mathbb{I} -T_\lambda)^{-1} g,
\end{equation} 
is the unique solution to \eqref{uno1}. 

\smallskip
\noindent \textit{Regularity of $\psi$ and estimates}. 
 All the  assertions  follow
 easily from \eqref{form1} since $( \mathbb{I} -T_\lambda)^{-1} g \in Y$ and we can 
 apply Theorem \ref{cos}, Lemmas \ref{neu} and \ref{corsa}. 
\end{proof}
\medskip

 In the Appendix we will also present a result on the stability of the PDE \eqref{uno1}, see Lemma \ref{lemma: stability PDE} .

\section{Regularity of the characteristics}\label{sec:SDE}

We will prove existence of a stochastic flow for the SDE \eqref{eq-SDE} assuming 
Hypothesis \ref{hyp-holder}.

\smallskip

We can rewrite our SDE as follows. Set $Z_t =(X_t ,V_t ) \in \R^{2d}$, $z_0=(x_0,v_0)$ and introduce the functions $b(x,v) =  A\cdot z + B(z):\R^{2d}\to \R^{2d}$, where 
\begin{equation}   \label{ABQ}
 A = \begin{pmatrix}    	0 & \mathbb{I} \\	0 & 0 
       \end{pmatrix} \ , \qquad
{R \,} = \begin{pmatrix}   0\\ \mathbb{I} \end{pmatrix},
\;\;\; \; {R \,} {R \,}^* = Q = \begin{pmatrix}  0 & 0\\ 0 & \mathbb{I}  \end{pmatrix} , \;\;\; 
B = R F = \begin{pmatrix}    0  \\ F    \end{pmatrix}: \R^{2d} \to \R^{2d} \, .
\end{equation}
With this new notation, \eqref{eq-SDE} can be rewritten as
\begin{equation} \label{eq Z}
\left\{ \begin{array}{cc}
\dd Z_t = b(Z_t) \, \dd t +R \cdot \dd W_t \\
Z_0 = z_0 \hspace{25mm} \end{array} \right. \, 
\end{equation}
or
\begin{equation}   \label{eq Z2}
\left\{ \begin{array}{cc}
\dd Z_t = \big(A\cdot Z_t  + B(Z_t) \big) \dd t +R \cdot \dd W_t \\
Z_0 = z_0 \hspace{43mm} \end{array} \right. \, .
\end{equation}
We have 
\begin{align}  
\nonumber
X_t &= x_0 + \int_0^t V_s \, \dd s = x_0 + t v_0 
+ \int_0^t (t-s) F(X_s, V_s) \, \dd s  + \int_0^t  W_s \, \dd s \,  , \label{Xt} \\ 
V_t &= v_0 + \int_0^t F(X_s, V_s) \, \dd s + W_t \, .
\end{align}

\subsection{Strong well posedness}

To prove strong  well posedness for \eqref{eq Z} we will also use 
 solutions $U$ with values in $\R^{2d}$  of 
\begin{align} \label{PDE-lambda}
& \lambda U (z)
  - \frac{1}{2} \mathrm{Tr} \big(Q D^2 U (z) \big) -   \langle Az, DU(z)\rangle  
 -  \langle B(z), DU(z)\rangle  = B(z),    \nonumber \\ &
\text{i.e.,} \;\;  \lambda U (z) - {\mathcal L} U(z) =   B(z)  
\end{align}
 (defined componentwise at least  for $\lambda$  large enough).   
Note that $U  = \begin{pmatrix}
    0 \\
   \tilde u
    \end{pmatrix}$ 
where  
$$
 \lambda \tilde u (z) - {\mathcal L} \tilde u(z) =   F(z) \, 
$$
is again defined componentwise ($\tilde u : \R^{2d} \to \R^d$).

\begin{remark}\label{rem hypoellipticity} 
In the following, according to \eqref{ABQ}, we will say that the singular diffusion $Z_t$ (the noise acts only on the last $d$ coordinates $\{e_{d+1}, \dots , e_{2d} \}$) or the 
associated Kolmogorov operator  
\begin{equation*}
\calL f(z) = \frac{1}{2} \triangle_v f(z) + \langle b(z) , D f (z)\rangle, 
\end{equation*}
$b(z) = Az +B(z)$,  are hypoelliptic to refer to the fact that the vectors
\begin{equation*}
\big\{ e_{d+1}, \dots , e_{2d}, A e_{d+1}, \dots , A e_{2d} \big\} 
\end{equation*}
generate $\R^{2d}$. Equivalently using $Q$ given in \eqref{ABQ} and the adjoint matrix $A^*$ 
we have that the symmetric matrix  $Q_t = \int_0^t 
e^{sA} Q e^{sA^*} \dd s$  is positive definite for any $t >0$ (cf. \eqref{dia1}). Note that
$$
\mathrm{det} (Q_t ) = c\, t^{4d},\;\;
t>0 \, .
$$
\end{remark}

We collect here some preliminary results, which we will later need. Recall the OU process  
\begin{equation} \label{ou1}
\left\{ \begin{array}{cc}
\dd L_t =  A L_t \, \dd t + {R \,}  \dd W_t \\
L_0 = z \in \R^{2d} \hspace{11mm} \end{array} \right. ,
 \quad\; \text{i.e.,} \ \;\;
 L_t = L_t^z =  e^{tA}z + \int_0^t e^{(t-s) A}  R \, \dd W_s \, .
\end{equation}
Using the fact that $L_t$ is hypoelliptic, for any $t>0$, one gets that the law of $L_t$ is 
equivalent to the Lebesgue measure in $\R^{2d}$ (see for example the proof of the next lemma). We also have the following result.

 \begin{lemma}\label{lem OU-Lp}
Let $(L_t^z)$  be the OU process solution of \eqref{ou1}. Let $f:\R^{2d} \to \R$ belong to
$L^q(\R^{2d})$ for  $q> 2d$. Then there exists a constant $C$ depending on $q,d$ and
$T$ such that
\begin{equation}
 \sup_{z\in\R^{2d}} \E \Big[  \int_0^T f(L_s^z) \, \dd s  \Big] \le C \| f\|_{L^q(\R^{2d})} \, . 
\end{equation}
 \end{lemma}

 \begin{proof}
We need to compute
$$  \E \Big[  \int_0^T f(L_s^z) \, \dd s  \Big] = \int_0^T P_s f(z) \, \dd s \, , $$
where $P_t$ is the Ornstein-Uhlenbeck semigroup introduced in \eqref{OU semigr}. 
By changing variable and using the H\"older inequality we find,
for $t \in [0,T] $, $z \in \R^{2d}$,
\begin{align*}
|P_t f(z)| &=  \Big| c_d \int_{\R^{2d}} f(e^{tA}z + \sqrt{Q_t} \, y) e^{- 
  \frac{|y|^2}{2}   }   \, \dd y \Big| 
  \le c_q  \Big (  \int_{\R^{2d}} |f(e^{tA}z + \sqrt{Q_t} \, y)|^q \,  \dd y \Big)^{1/q} \\
&= \frac{c_q}{ (\text{det}(Q_t))^{1/2q} }
  \Big (  \int_{\R^{2d}} |f(e^{tA}z + w)|^q \, \dd w \Big)^{1/q} 
  = \frac{c_q}{ (\text{det}(Q_t))^{1/2q} } \| f\|_{L^q(\R^{2d})} \, .
\end{align*}
with $c_q$ independent of $z$.  We now have to study when
\begin{equation} \label{q4e}
  \int_{0}^{t}  \frac{1}{ (\text{det}(Q_s))^{1/2q} } \, \dd s < \infty \, .
\end{equation}   
 By a direct computation  for $s \to 0^+$
$$
(\text{det}(Q_s))^{1/2q} \sim  c  (s^{4d} )^{1/2q} \, ,
$$ 
hence the result follows for $q>2d$.   
\end{proof}

\vv We state now the classical Khas'minskii lemma for  an OU process. 
The original version of this lemma (\cite{Kh59}, or \cite[Section 1, Lemma 2.1]{Sz98}) is stated for a Wiener process, but the proof only relies on the 
 Markov property of the process, so that its extension
 to this setting requires no modification.
 
\begin{lemma}[Khas'minskii 1959] \label{lem Khas}
 Let $(L_t^z)$ be our $2d-$dimensional OU process starting from $z$ at time $0$ and
 $f:\R^{2d}\to \R$ be a positive Borel function. Then, for any $T>0$ such that
 \begin{equation}
  \alpha = \sup_{z\in \R^{2d}} \E \Big[ \int_0^T f \big( L_t^z \big) \,\dd t \Big] < 1 \,,
 \end{equation}
we also have
 \begin{equation}
  \sup_{z\in \R^{2d}} \E \Big[ \exp\Big( \int_0^T f \big( L_t^z \big) \, \dd t
  \Big) \Big] < \frac{1}{1-\alpha} \, .
 \end{equation}
\end{lemma}

We now introduce a generalization of the previous Khas'minskii lemma which we will
use to prove the Novikov condition, allowing us to apply Girsanov's theorem.

\begin{proposition} \label{Khas}  
 Let $(L_t)$ be the OU process solution of \eqref{ou1}. Let $f:\R^{2d} \to \R$ belong to
$L^q(\R^{2d})$ for  $q>2d$. Then, there exists a constant $K_f$ depending on $d,q,T$
and continuously depending on $\|f \|_{L^q(\R^{2d})}$ such that 
\begin{equation}
 \sup_{z\in\R^{2d}} \E \Big[ \exp \Big( \int_0^T |f(L_s^z)| \dd s \Big) \Big] = K_f <\infty \, . 
\end{equation}
\end{proposition}

\begin{proof} 
From Lemma \ref{lem OU-Lp}, for any $a>1$ s.t. $q/a>2d$ we get 
\begin{equation*}
 \sup_{z\in\R^{2d}} \E \Big[  \int_0^T |f|^a (L_s^z) \, \dd s  \Big] 
 \le C \| f\|_{L^q(\R^{2d})}^a  \, .
\end{equation*}
Setting  $\varepsilon= (C \| f \|_{L^q}^a )^{-1} \wedge 1$, we apply Young's inequality: $|f(z)| \le \frac{\epsilon}{a} |f(z)|^a + C_{\epsilon} \frac{a-1}{a}$  
and Khas'minskii's Lemma \ref{lem Khas} replacing $f$ with  $\frac{\varepsilon}{a} |f|^a$ to get 
\begin{align*}
 \sup_{z\in\R^{2d}} \hspace{-1mm} \E \Big[ \exp \Big( \int_0^T |f(L_s^z)| \dd s \Big) \Big] \hspace{-0.5mm} \le \hspace{-1mm}
 \sup_{z\in\R^{2d}} \hspace{-1mm} \E \Big[ \exp \Big( \int_0^T \frac{\varepsilon}{a} |f(L_s^z)|^a
 \dd s \Big) \Big] e^{T c_{\varepsilon,a}} 
 \hspace{-0.5mm} \le \hspace{-0.5mm} \frac{1}{1-\alpha} e^{C T} \hspace{-1mm} < \hspace{-1mm}\infty  .
\end{align*}
\end{proof}

 The next result can be proved by using the Girsanov theorem (cf. \cite{IW} and \cite{LS}). 

\begin{theorem} \label{gi1} Suppose that in \eqref{eq Z} we have 
$F \in L^p(\R^{2d}; \R^{d})$ with $p>4d$. Then the  following statements hold. 

(i) Equation \eqref{eq Z} is well posed in the weak sense.

(ii) For any $z \in \R^{2d}$, $T>0$ the law in the space of continuous functions 
$C([0, T]; \R^{2d})$ of the solution $Z =(Z_t)= (Z_t^z)$ to the equation 
(\ref{eq Z}) is equivalent to the law of the OU process $L = (L_t)= (L_t^z)$.

(iii) For any $t > 0$,  $z \in \R^{2d}$, the law of $Z_t$ is equivalent to the 
Lebesgue measure in $\R^{2d}$.
\end{theorem}
\begin{proof}  \textbf{ (i)} \underline{Existence.} We argue similarly to the proof of Theorem IV.4.2 in \cite{IW}. Let $T>0.$ 
 Starting from an Ornstein-Uhlenbeck process (cf. \eqref{ou1})  
$$
L_t = L_t^z =  z + \int_0^t A L_s \, \dd s +   R W_t \, , \;\; t \ge 0
$$
defined on a stochastic basis $(\Omega, {\cal F}, ({\cal F}_t), \PP)$ on which it is defined an $\R^d$-valued Wiener process $(W_t) = W$, we can define the process 
 \begin{equation}  \label{ls}
H_t := W_t - \int_0^t F (L_r) \, \dd r \, , \;\; t \in  [0,T].
 \end{equation}
Since $p > 4d$, 
 Proposition \ref{Khas} with $f=F^2$ provides the Novikov condition ensuring that
the process 
 $$
   \Phi_t  = \exp \Big ( \int_0^t {\langle}  F(L_s), \dd W_s {\rangle}
   \, -  \, \frac{1}{2}  \int_0^t | F(L_s)
   |^2 \dd s  \Big),\;\; t \in [0,T],
$$
 is an ${\cal F}_t$-martingale. Then, by the Girsanov theorem 
$(H_t)_{t \in [0,T]} $ is a $d$-dimensional  Wiener process on
$(\Omega, {\cal F}_T, ({\cal F}_s)_{s \le T}, \Q)$, where $\Q$ is
the probability measure on $(\Omega, {\cal F}_T)$
having density $\Phi = \Phi_T $ with respect to $\PP$.
We have that on the new probability space
$$
L_t = L_t^z =   z + \int_0^t A L_s \, \dd s + \int_0^t R F (L_s) \, \dd s +   R H_t \, , \;\; t \in  [0, T]
$$ 
(cf. \eqref{ABQ}). Hence $L =  (L_t)$ is a solution to \eqref{eq Z} on 
$(\Omega,  {\cal F}_T,  ({\cal F}_s)_{s \le T}, \Q)$.

\hh \underline{Uniqueness.}
 To prove weak uniqueness we use some results from \cite{LS}. 
First note that the process 
\begin{equation}
\label{it22}
V_t = v_0 + \int_0^t F(X_s, V_s) \, \dd s + W_t \, 
\end{equation} 
(cf. \eqref{Xt}) is a process of diffusion type according to Definition 7 in page 118 
of \cite[Section 4.2]{LS}. Indeed, since 
$
X_t = x_0 $ $ + \int_0^t V_s \, \dd s
$
we have 
$$
V_t = v_0 + \int_0^t F\Big(x_0 + \int_0^s V_r \, \dd r, V_s \Big) \, \dd s + W_t
$$
and  the process $ (b_s (V))_{s \in [0,T]}
= 
(F( x_0 + \int_0^s V_r \, \dd r, V_s))_{s \in [0,T]}$ is $({\cal F}^{V}_t)$-adapted
(here ${\cal F}^{V}_t$ is the $\sigma$-algebra generated by $\{ V_s,\; s \in [0,t]\}$).

We can apply to $V = (V_t)$ Theorem 7.5 on page 257 of \cite{LS} (see also paragraph 7.2.7 in \cite{LS}): since $\int_0^T |b_s(V) |^2 \dd s < \infty$, $\PP$-a.s., we obtain that 
$$
\mu_V \sim \mu_{W} \; \; \text{on} \; \; {\cal B}\big(C([0,T]; \R^d)\big) \, ,
$$
i.e. the laws of $V= (V_t)_{t \in [0,T]}$ and $W = (W_t)_{t \in [0,T]}$ are equivalent. Moreover, by \cite[Theorem 7.7]{LS}, the Radon-Nykodim derivative $\frac{\mu_V}{\mu_W}(x)$, $x \in C([0,T]; \R^d)$, verifies 
$$
\frac{\mu_V}{\mu_W}(W) =  \exp \Big ( \int_0^T {\langle}  b_s  (W), \dd W_s {\rangle}
   \, -  \, \frac{1}{2}  \int_0^T | b_s (W)
   |^2 \, \dd s  \Big).
$$
It follows that, for any Borel set $B \in {\cal B}(C([0,T]; \R^d))$,
$$
\E \big[ 1_B (V) \big] = \E^{\mu_{W}} \Big[ 1_B \,  \frac{\mu_V}{\mu_W} \Big] =
 \E \Big[1_B(W) \, \exp \Big ( \int_0^T {\langle}  b_s (W), \dd W_s {\rangle}
   \, -  \, \frac{1}{2}  \int_0^T |  b_s (W)
   |^2 \, \dd s  \Big)\Big];
$$
this shows easily that uniqueness in law holds.

\hh Clearly  (iii) follows from (ii). Let us prove (ii). 
 
\hh {\bf (ii)}   The processes $L = (L_t)$ and $Z = (Z_t)$, $t \in [0,T]$, satisfy the
same equation \eqref{eq Z} in $(\Omega , \calF, {\cal F}_t, \Q, (H_t))$ and
$(\Omega , \calF, {\cal F}_t, \PP, (W_t)) $
 respectively. Therefore, by weak uniqueness,  the laws of  $L$ and $Z$ on 
$C([0,T]; \R^{2d})$ are the same
  (under the probability  measures $\Q$ and $\PP$ respectively).
 Hence, for any Borel set $J \subset C([0,T]; \R^{2d})$, we have
$$ 
\E [1_{J} (Z)] = \E [1_{J} (L) \, \Phi] \, .
$$
Since $W_t = ( \langle L_t, e_{d+1} \rangle, \ldots,   \langle L_t, e_{2d} 
\rangle )$ we see that each $W_s$  
is measurable with respect to the $\sigma$-algebra generated by the random 
variable $L_s$, $s \le T$. By considering $L$ as a random variable with values 
in $C([0,T]; \R^{2d})$, we obtain that  
$$
\Phi = \exp [ G(L) ]
$$
for some  measurable function  $G : M= C([0,T]; \R^{2d}) \to \R$. Using the laws 
$\mu_Z$ of $Z$ and $\mu_L$ of $L$ we find 
$$
\int_{M} 1_J(\omega) \mu_{Z} (\dd \omega) = \E [1_{J} (Z)] = \E \big[1_{J} (L) \, \exp 
[ G(L) ] \big]
 =  \int_{M} 1_J(\omega) \;  \exp [ G(\omega) ] \mu_{L} (\dd \omega) \, .
$$
Finally note that $|G (\omega)| < \infty$, for any $\omega \in M$ $\mu_L$-a.s. 
(indeed  
$\int_{M}   |G(\omega)| \mu_{L} (\dd \omega) $ $ = E [ |G(L)|] < \infty$). It 
follows that $\exp [ G(\omega) ] >0$, for any $\omega \in M$ $\mu_L$-a.s., and 
this shows that $\mu_L$ is equivalent to $\mu_Z$.
\end{proof}

\smallskip
We can now prove that the result of Lemma \ref{lem OU-Lp} holds also replacing the 
OU process $L_t$ with $Z_t$.

\begin{lemma} \label{lem Z-Lp}
  Let $Z_t^z$ be a solution of \eqref{eq Z}. Let $f:\R^{2d} \to \R$ belong to $L^q(\R^{2d})$ for some $q> 2d$. Then there exists a constant $C$ depending on $q,d$ and
$T$ such that
\begin{equation}\label{fZ - Lp}
 \sup_{z\in\R^{2d}} \E \Big[  \int_0^T f(Z_s^z) \, \dd s  \Big] \le C \| f\|_{L^q(\R^{2d})} 
\end{equation}
and a constant $K_f$ depending on $q,d, T$ and continuously depending on $\|f\|_{L^q(\R^{2d})}$ for which
\begin{equation}\label{exp-fZ - Lp}
 \sup_{z\in\R^{2d}} \E \Big[  \exp \Big( \int_0^T f(Z_s^z) \, \dd s\Big)  \Big] \le K_f  \, .
\end{equation}
 \end{lemma}
 \begin{proof}
Recall that $F\in L^p(\R^{2d})$ for $p>4d$. As seen in the previous proof, the laws of $L_t$ and $Z_t$ are the same under the
   $\Q$ and $\PP$ respectively. Then, applying H\"older's inequality with $1/a+1/a' =1$ we have
\begin{align*}
 \E^\PP \Big[ \int_0^T f(Z_s) \, \dd s \Big] = \E^\Q \Big[ \int_0^T f(L_s)
 \, \dd s \Big]
 \le \E^\PP \Big[ \int_0^T f^a(L_s) \, \dd s \Big]^{1/a}
 \E^\PP \Big[ \Phi^{a'} \Big]^{1/a'} \, .
\end{align*}
Taking $a>1$ small enough so that $q/a>2d$, we can apply Lemma \ref{lem OU-Lp} to $|f|^a$ and control the first expectation on the right hand side with a constant times
$\| f\|_{L^q(\R^{2d})} $. Then we write
\begin{align*}
   \Phi^{a'}  = \exp \Big ( \int_0^T {\langle}  a' F(L_s), \dd W_s {\rangle}
   \, -  \, \frac{1}{2}  \int_0^T | a' F(L_s)   |^2 \, \dd s + \frac{(a')^2 - a'}{2} \int_0^T | F(L_s)   |^2 \, \dd s    \Big) \, ,
\end{align*}
which has finite expectation due to Proposition \ref{Khas}. Both these estimates are uniform
in $z$, so that \eqref{fZ - Lp} follows.
 Similarly, we have
\begin{align*}
 \E^\PP \Big[ \exp\Big( \int_0^T f(Z_s) \, \dd s\Big) \Big]  \le \E^\PP \Big[\exp\Big( 2 \int_0^T f(L_s) \, \dd s \Big) \Big]^{1/2}
 \E^\PP \Big[ \Phi^{2} \Big]^{1/2} \, .
\end{align*}
Both terms on the right hand side are finite due to Proposition \ref{Khas}: this proves \eqref{exp-fZ - Lp}. \end{proof}

\medskip

Recall that we are always assuming  Hypothesis \ref{hyp-holder}.

\begin{lemma}\label{U moments estimate}
Any process $(Z_t)$  which is solution of the SDE \eqref{eq Z} has finite moments of any order, uniformly in $t\in[0,T]$: for any $q\ge2$
\begin{equation}\label{U moments}
\E \big[ | Z_t^z |^q \big] \le C_{z,q,d,T} < \infty \, .
\end{equation}
\end{lemma}

\begin{proof}
Recall that, setting $Z_t^z = Z_t$,
$$
Z_t = z + \int_0^t F(Z_s) \, \dd s + \int_0^t A Z_s \, \dd s + \int_0^t R \, \dd 
W_s \, .
$$
It follows from \eqref{exp-fZ - Lp} that for any $q \ge1$,
$
\E \big[ | \int_0^T F(Z_t) \, \dd t |^q \big] \le C.
$ 
Using  this bound, the  Burkholder  inequality for  stochastic integrals and the Gr\"onwall lemma we obtain the assertion. 
\end{proof}

\vv In the proof of strong uniqueness of solutions of the SDE \eqref{eq Z} we will have to deal with a new SDE with a Lipschitz drift coefficient, but a diffusion which only has derivatives in $L^p$. However, following un idea of Veretennikov \cite{V}, we can deal with increments of the diffusion coefficient on different solutions by means of the process $A_t$ defined in \eqref{def At}. The following lemma generalizes Veretennikov's result to our degenerate kinetic setting and even  provides bounds on the exponential of the process $A_t$. It will be a key element to prove continuity of the flow associated to  \eqref{eq Z} and will also be used in subsection \ref{subsec: derivatives} to study weak derivatives of the flow.

\begin{lemma}\label{lemma exp At}  %
Let $Z_t, \, Y_t$ be two solutions of \eqref{eq Z} starting from $z,\,y \in \R^{2d}$ respectively, $U:\R^{2d} \to \R^{2d}, \ U \in X_{p,s} \cap C^1_b$ (see Definition \ref{deff}), and set
\begin{equation}\label{def At} 
A_t =   \int_0^t 1_{ \{ Z_s  \not = Y_s \}}
 \; \frac{ \| [DU (Z_s)- DU (Y_s)]{R \,}\Big \|^2_{HS}
 } {|Z_s - Y_s|^2} \ \dd s  \, ,
\end{equation}
where $\| \cdot \|_{HS}$ denotes the Hilbert-Schmidt norm. Then, $A_t$ is a well-defined, real valued, continuous, adapted, increasing process such that 
$\E[A_T]< \infty$, for every $t\in[0,T]$
 \begin{equation}\label{At}
 \int_0^t  \Big\| \big[ DU (Z_s^z)- DU (Y_s^y) \big]{R \,} \Big\|^2_{HS}  \, \dd s = 
\int_0^t \big| Z_s^z - Y_s^y \big|^2 \, \dd A_s 
 \end{equation}
 and for any $k\in \R$, uniformly with respect to the initial conditions $z,y$:
\begin{equation}\label{exp At}
\sup_{z,y\in \R^{2d}} \E \big[ e^{k A_T} \big]  < \infty \, .
 \end{equation} 
\end{lemma}

\begin{proof}
Recall that 
$
B  = \begin{pmatrix}
    0 \\
   F
    \end{pmatrix}
$
and 
$
\big\| \big[ DU (Z_s)- DU (Y_s) \big]{R \,} \big\|^2_{HS} =
 \big| D_v \tilde u (Z_s) - D_v \tilde u(Y_s) \big|^2.
$ 
We have
\begin{align*}
\big| D_v \tilde u (Z_s)- D_v \tilde u (Y_s) \big| &=
 \Big| \sum_{i=1}^{2d} (Z_s - Y_s)^i \int_0^1 D_i D_v \tilde u (
  r Z_s + (1-r)Y_s) \, \dd r \Big|  \\
& \le |Z_s - Y_s| \int_0^1  \big| D D_v \tilde u \big(
  r Z_s + (1-r)Y_s \big) \big| \, \dd r  \, .
\end{align*}
Set $Z_t^r = r Z_t + (1-r)Y_t$ (the process $(Z_t^r)_{t \ge 0}$ depends on
 $r \in [0,1]$). We will first prove that
\begin{equation} \label{stima1}
  \E \Big[ \int_0^1 \, \dd r  \int_0^t  \big| D D_v \tilde u (
  Z_s^r) \big|^2 \dd s \Big] < \infty \, , \qquad  t>0 \, .
\end{equation}
By setting  $F_s^r = [rF(Z_s)+ (1-r)F(Y_s)] $ and $z^r = rz + (1-r)y$,
we obtain, for any $r \in [0,1]$,
$$
Z_t^r = z^r + \int_0^t   \begin{pmatrix}
    0 \\
   F_s^r
    \end{pmatrix}   \, \dd s 
 + \begin{pmatrix}
    0 \\
   W_t
    \end{pmatrix}     + \int_0^t A Z_s^r \dd s \, .
$$
Since $\int_0^T |{F}_{s}^{r}|^{2} \, \dd s \le C \int_0^T |F (Z_s)|^{2}
+ |F (Y_s)|^{2}  \, \dd s$, using H\"older's inequality and Lemma \ref{lem Z-Lp} we get for all $k\in\R$
\begin{equation}\label{eq expF uni}
\sup_{z,y} {\mathbb E}\Big [  \exp\Big( k \int_{0}^{T}\, |{F}_{s}^{r}|^{2} \, \dd s\Big) \Big]
\leq C_{k} \, < \, \infty \, ,
\end{equation}
where the constant $C_k$ depends on $k, p, T$ and $\|F\|_{L^p(\R^{2d})}$, but is uniform in
 $z,y$ and $r$. 

We can use again the Girsanov theorem (cf. the proof of Theorem  \ref{gi1}).  
 The process 
 \begin{equation*}
\tilde W_t := W_t + \int_0^t F_s^r \, \dd r \, , \qquad    t \in  [0,T] 
 \end{equation*}
 is a  $d$-dimensional  Wiener process on
$(\Omega, ({\cal F}_s)_{s \le T}, {\cal F}_T, \Q)$, where $\Q$ is
the probability measure on $(\Omega, {\cal F}_T)$
 having  the  density $\rho_r $ with respect to $\PP$,
 $$
   \rho_r  = \exp \Big ( \int_0^T {- \langle}  F_s^r  \, ,\, \dd W_s {\rangle}
   \, -  \, \frac{1}{2}  \int_0^T | F_s^r |^2 \, \dd s  \Big) \, .
$$  
Recalling the Ornstein-Uhlenbeck process $L_t$ (starting at $z^r$), 
 i.e.,
\begin{equation}  \label{ou}
 L_t = e^{tA} z^r + W_A(t),
\qquad \text{where} \quad
  W_A(t) = \int_0^t e^{(t-s) A}  R \, \dd W_s \, ,
\end{equation}
 we have:
  $$ 
Z_t^r = L_t + \int_0^t e^{(t-s) A} R \, F_s^r \, \dd s \, .
  $$
Hence
$$
Z_{t}^{r}=e^{tA} z^r\, + \, \int_{0}^{t}e^{(  t-s )  A} \, \dd \tilde W_{s}%
$$
is an OU process on $(\Omega,  ({\cal F}_s)_{s \le T}, {\cal F}_T, \rho_r \PP)$.

We now find, by the H\"older inequality,    for some $a>1$ such that   $1/a  + 1/a' =1$,
\begin{align}
 \E  \Big[  \rho^{-1/a}_r \rho ^{1/a}_r   \int_0^t |  D D_v \tilde u  (
  Z_s^r)|^2  \dd s \Big] 
 &\le c_T  \Big( \E \Big [  \rho_r  \int_0^t |  D D_v \tilde u (
  Z_s^r)|^{2a}  \dd s    \Big] \Big)^{1/a}
   \big( \E[\rho^{- a'/a}_r]   \big)^{1/a'}           \label{bo} \\
 &\le
   C_T \Big( \E \Big [  \rho_r  \int_0^t |  D D_v \tilde u (
  Z_s^r)|^{2a} \,\dd s    \Big] \Big)^{1/a} \,  ,              \nonumber
\end{align}
for any $t \in [0,T]$.  Observe that the bound on the moments of $\rho_r$ is uniform in the initial conditions $z,y\in\R^{2d}$ due to \eqref{eq expF uni}. Setting 
  $f(z) =|  D D_v \tilde u (z) |^{2a}  $ and using the Girsanov Theorem, 
assertion \eqref{stima1} follows from Lemma \ref{lem OU-Lp} if we fix $a >1$ such that $q=p/2a >2d$.

Therefore,  the process $A_t$ is well defined and $\E[A_t]<\infty$ for all $t\in[0,T]$. \eqref{At}
 and the other properties of $A_t$ follow.
\medskip

To prove the exponential integrability of the process $A_t$ we proceed in a way similar to \cite[Lemma 4.5]{FF13a}. Using the convexity of the exponential function we get
\begin{align*}
\E \Big[ e^{kA_T} \Big] \le \E \Big[ \exp \Big( k \hspace{-1mm} \int_0^T \hspace{-2mm}  \int_0^1  \hspace{-1mm} | D D_v \tilde u ( Z_s^r) |^2 \dd r  \dd s \Big) \Big] \le  \hspace{-1mm} \int_0^1  \hspace{-1mm} \E \Big[ \exp \Big( k  \hspace{-1mm} \int_0^T    \hspace{-1mm} | D D_v \tilde u ( Z_s^r) |^2 \dd s \Big) \Big] \dd r
\end{align*}
and we can continue as above (superscripts denote the probability measure used to take expectations)
\begin{align*}
\sup_{z,y} \E^\PP \Big[ e^{kA_T} \Big] &\le\sup_{z,y} \int_0^1 \E^\PP  \Big[  \rho^{-1/a}_r \rho ^{1/a}_r \exp\Big( k  \int_0^T |  D D_v \tilde u  (  Z_s^r)|^2  \dd s \Big)  \Big]  \dd r\\
& \le C_T \sup_{z,y} \int_0^1 \E^\Q  \Big[ \exp\Big( a k \int_0^T |  D D_v \tilde u  ( Z_s^r)|^2  \dd s \Big)  \Big]^{1/a}  \dd r\\
& \le C_T \sup_{z,y} \int_0^1 \E^\PP  \Big[ \exp\Big( a k \int_0^T |  D D_v \tilde u  ( L_s) |^2  \dd s \Big)  \Big]^{1/a}   \dd r \, .
\end{align*}
The last integral is finite due to Lemma \ref{Khas} because $p/2>2d$. The proof is complete.
\end{proof}

\smallskip

\begin{proposition}[It\^o formula]\label{Ito}
 If ${\varphi}:\R^{2d} \to \R$ belongs to $X_{p,s} \cap C^1_b$  
and $Z_t$ 
is a solution of (\ref{eq Z}), for any $0\le s \le t \le T$ the following It\^o 
formula holds:
\begin{equation} \label{Ito Z}
{\varphi}(Z_t) = {\varphi}(Z_s) + \int_s^t \big[  b(Z_r) \cdot D {\varphi}( Z_r)  +  
\frac{1}{2} \Delta_v {\varphi}(Z_r) \big] \dd r  + \int_s^t  D_v {\varphi}( Z_r) \, 
\dd W_r \, .
\end{equation}
\end{proposition}

\begin{proof} Note that we can use (iii) in Theorem \ref{gi1} to give a meaning
to the critical term $\int_s^t \Delta_v {\varphi}(Z_r)\,  \dd r$. The result then
follows approximating $\varphi$ with regular functions and using Lemma \ref{lem Z-Lp}.

Let $\varphi_\varepsilon \in C^\infty_c \to \varphi$ in  $X_{p,s}$.
$\varphi_\varepsilon$ satisfy the 
assumptions of the classical It\^o formula, which provides an analogue of 
\eqref{Ito Z} for $\varphi_\varepsilon (Z_t)$. For any fixed $t$, the random
variables $\varphi_\varepsilon (Z_t) \to \varphi(Z_t)$ $\PP$-almost surely.
Using that $D{\varphi}$ is bounded and almost surely $ F(Z_r)$ and $AZ_r$ 
are in $L^1(0,T)$ (this follows by Lemma \ref{lem Z-Lp} and Lemma 
 \ref{U moments estimate} respectively), 
the dominated convergence theorem gives the convergence of the first term in the 
Lebesgue integral. For the second term we use Lemma \ref{lem Z-Lp} with
$f=\Delta_v {\varphi_\varepsilon} - \Delta_v {\varphi}$ (recall that $p>6d$):
$$ 
\E \Big[ \int_s^t \Delta_v {\varphi_\varepsilon}(Z_r) - \Delta_v {\varphi} (Z_r)
\ \dd r \Big] 
\le C \| \Delta_v {\varphi_\varepsilon} - \Delta_v {\varphi} \|_{L^p(\R^{2d})} \to 0 \, .
$$
In the same way, one can show that  
$\E \big[ \int_s^t |  D_v {\varphi_\varepsilon} ( Z_r) -  D_v {\varphi} ( Z_r) |^2 \, \dd r \big]$
 converges to zero, which implies the convergence of the stochastic integral by the It\^o isometry.
\end{proof}

\begin{remark}\label{Ito2}
Using the boundedness of $\varphi$, it is easy to generalize the above It\^o formula \eqref{Ito Z} to $\varphi^a(Z_t)$ for any $a\ge2$.
\end{remark}

We can finally prove the well-posedness in the strong sense of the degenerate 
SDE (\ref{eq Z}). A different proof of this result in a H\"older setting is
contained in \cite{CdR}, 
but no explicit control on the dependence on the initial data is given there, so 
that a flow cannot be constructed. See also the more recent results of \cite{WaZh1}. 
 We here present a different, and in 
some sense more constructive, proof. This approach, based on ideas introduced in 
\cite{FGP10}, \cite{KR05}, \cite{FF13a}, will even allow us to obtain some 
regularity results on certain derivatives of the solution. 
We will use Theorem \ref{PDE-1!} from Section \ref{sec: regularity Bessel}, which provides 
the regularity $X_{p,s} \cap C^1_b(\R^{2d})$ of solutions of \eqref{PDE-lambda}.

\begin{theorem} \label{strong 1!}
Equation (\ref{eq Z}) is well posed in the strong sense.
\end{theorem}

\begin{proof}
Since we have weak well posedness by (i) of Theorem \ref{gi1}, the 
Yamada-Watanabe principle provides strong existence as soon as strong uniqueness 
holds. Therefore, we only need to prove strong uniqueness. This can be done by 
using an appropriate change of variables which transforms equation (\ref{eq Z}) 
into an equation with more regular coefficients. This method was first 
introduced in \cite{FGP10}, where it is used to prove strong uniqueness for a 
non degenerate SDE with a H\"older drift coefficient.

Here, the SDE is degenerate and we only need to regularize the 
second component of the drift coefficient, $F(\cdot)$, which is not Lipschitz 
continuous. We therefore introduce the auxiliary PDE (\ref{PDE-lambda}) with 
$\lambda$ large enough such that  
\begin{equation}\label{stima nablaU} 
\| U_{\lambda} \|_{L^\infty(\R^{2d})} + \| D U_{\lambda} \|_{L^\infty(\R^{2d})} < 1/2
\end{equation}
 holds (see  \eqref{lin2}).  
 In the following we will always use this value of $\lambda$ and to ease 
notation we shall drop the subscript for the solution $U_\lambda$ of 
(\ref{PDE-lambda}), writing $U_{\lambda} = U.$

Let $Z_t$ be one solution to \eqref{eq Z} starting from $z \in \R^{2d}$. Since
$$
 Z_t  = z  + \int_0^t B (Z_s) \, \dd s  +  \int_0^t   AZ_s  \, \dd s  +{R \,} 
W_t \, ,
$$
and $U\in X_{p,s} \cap C^1_b$ (see Theorem \ref{PDE-1!}), by the It\^o formula of Proposition \ref{Ito}
we have
\begin{align*}
U(Z_t) &= U(z)  + \int_0^t  DU (Z_s){R \,}
 \, \dd W_s   + \int_0^t {\mathcal L } U(Z_s )\, \dd s  \\
&= U(z)  + \int_0^t  DU (Z_s){R \,}
 \dd W_s   + \lambda \int_0^t U(Z_s) \,\dd s
  -  \int_0^t B(Z_s) \, \dd s \, .
\end{align*}
Using the SDE to rewrite the last term we find  
\begin{align*}
U(Z_t) =& \, U(z)  + \int_0^t  DU (Z_s){R \,}  \dd W_s   + \lambda \int_0^t 
U(Z_s)  \,\dd s - Z_t  + z  +  \int_0^t A Z_s \, \dd s +{R \,} W_t
\end{align*}
and so
\begin{align} \label{forse}
\hspace{-2mm} Z_t =  U(z)- U(Z_t)   + \int_0^t \hspace{-1mm} DU (Z_s){R \,} \dd W_s  + \lambda 
\int_0^t \hspace{-1mm} U(Z_s) \dd s  + z  + \int_0^t \hspace{-1mm} A Z_s \dd s  +{R \,} W_t  \,  .
\end{align}
Let now $Y_t$ be another solution starting from $y\in\R^{2d}$ and let  
\begin{equation}
\label{def phi}
 \gamma (x) = x + U(x), \;\; x \in \R^{2d} \, .
\end{equation}
We have $ \gamma(z) - \gamma (y) =   z-y + U(z) - U(y)$, and so   $| z-y | \le |U(z) 
- U(y) |+  |\gamma (z) - \gamma(y)|$.  Since we have chosen $\lambda$ such that $\|  
DU\|_{L^\infty(\R^{2d})} <1/2$, there exist finite constants $C,c >0$ such that
 \begin{equation} \label{dif}
c  |\gamma (z) - \gamma(y)| \le  |z-y| \le C |\gamma (z) - \gamma(y)|, \qquad \forall  z, y \in 
\R^{2d}. 
 \end{equation}
 We find
 \begin{equation} \label{eq phi(Z)}
 \dd \gamma(Z_t) = \big( \lambda U(Z_t)  + A Z_t \big) \, \dd t +   \big( DU (Z_t) + \mathbb{I} \big) {R \,} \cdot \dd W_t
 \end{equation}
 and
 \begin{align} \label{f5}
  \gamma (Z_t) - \gamma(Y_t) = & \,  z-y + U(z) - U(y) + \int_0^t  \big[ DU (Z_s)- 
DU (Y_s) \big] R \cdot \dd W_s \\
  & + \lambda \int_0^t  \big[ U(Z_s) - U(Y_s) \big] \dd s  + \int_0^t  A(Z_s - 
Y_s) \, \dd s \, .    \nonumber
\end{align}
 For $a\ge2$, let us apply It\^o formula to $\big| \gamma (Z_t) - \gamma(Y_t) \big|^a = \big( \sum_{i=1}^{2d} \big[\gamma (Z_t) - \gamma(Y_t)\big]_i^2 \big)^{a/2}$:
\begin{align*}
 \dd \big[ \big| \gamma (Z_t) - \gamma(Y_t) \big|^a \big]  & =  a \big| \gamma (Z_t) - 
\gamma(Y_t) \big|^{a-2} \big( \gamma (Z_t) - \gamma(Y_t) \big)
 \cdot \dd \big( \gamma (Z_t) - \gamma(Y_t) \big) \\
 & \hspace{-20mm} + \frac{a}{2} \big| \gamma (Z_t) - \gamma(Y_t) \big|^{a-4} \sum_{i,j=1}^{2d} \sum_{k=1}^{d}  \Big\{  
 (a-2) \big( \gamma (Z_t) - \gamma(Y_t) \big)_i  \big( \gamma (Z_t) - \gamma(Y_t) \big)_j \\
 & \hspace{31mm} + \delta_{i,j} \big| \gamma (Z_t) - \gamma(Y_t) \big|^2 \Big\} \\
 & \hspace{5mm} \times \Big[\big( DU (Z_t) - DU (Y_t) \big)R \Big]_{k,i}\Big[ \big( DU (Z_t) - DU (Y_t)  \big)R \Big]_{k,j}    \dd t\\
& \hspace{-1cm} \le a \big| \gamma (Z_t) - \gamma(Y_t) \big|^{a-2} \Big\{  \big( \gamma (Z_t) - \gamma(Y_t) \big) \cdot  \big[ DU (Z_t) - DU (Y_t) 
\big]R \cdot \dd W_t \\
& \quad +  \big( \gamma (Z_t) - \gamma(Y_t) \big) \cdot  \Big( \lambda \big[U(Z_t) - 
U(Y_t) \big] +  A(Z_t - Y_t)   \Big)  \, \dd t\\
& \quad +   C_{a,d}  \Big\| 
\big[ DU (Z_t)- DU (Y_t) \big]R \Big\|_{HS}^2 \dd t \Big\} \, .
\end{align*} 
Note that $Z_t$ has finite moments of all orders, and $U$ is bounded, so that 
also the process $\gamma(Z_t)$ has finite moments of all orders. Using also that 
$DU$ is a bounded function, we deduce that the stochastic integral 
is a martingale $M_t$:
$$
 M_t = \int_0^t a \big| \gamma (Z_s) - \gamma(Y_s) \big|^{a-2}   \big( \gamma (Z_s) - \gamma(Y_s) \big) \cdot  \big[ DU (Z_s) - DU (Y_s) 
\big]R \cdot \dd W_s
$$
As in \cite{KR05} and \cite{FF11} we now  consider  the following process 
\begin{align}\label{eq B_t}
 B_t &=  \int_0^t 1_{ \{ Z_s  \not = Y_s \}}
 \; \frac{ \Big\| [DU (Z_s)- DU (Y_s)]{R \,}\Big \|^2_{HS}
 } {|\gamma(Z_s) - \gamma(Y_s)|^2} \ \dd s \le C^2 A_t \, ,
\end{align}
where we have used the equivalence \eqref{dif} between $|Z_t-Y_t|$ and $|\gamma(Z_t) - \gamma(Y_T)|$ and $A_t$ is the process defined by \eqref{def At} and studied in Lemma \ref{lemma exp At}. Just as the process $A_t$, also $B_t$ has finite moments, and even its exponential has finite moments. With these notations at hand we can rewrite
\begin{align*}
 \dd \big[ \big| \gamma (Z_t) &- \gamma(Y_t) \big|^a \big]  \\
 & \le 
  a \big| \gamma (Z_t) - \gamma(Y_t) \big|^{a-2} \big(\gamma (Z_t) - \gamma(Y_t)\big) \cdot  \Big( \lambda \big[U(Z_t) - 
U(Y_t) \big]  +  A(Z_t - Y_t)   \Big) \, \dd t  \\
& \quad +  \dd M_t +  C_{a,d} \big| \gamma (Z_t) - \gamma(Y_t) \big|^{a} \,   \dd B_t \, .
\end{align*} 
 Again by It\^o formula   we have
\begin{align} \label{Ito 1!}
 \dd \Big( e^{-C_{a,d} B_t} \big| \gamma (Z_t) - \gamma(Y_t) \big|^a \Big) & = - C_{a,d} \, e^{-C_{a,d}  B_t} \big| \gamma (Z_t) - \gamma(Y_t) \big|^a \, \dd B_t \\
 + \, e^{-C_{a,d}  B_t} \Big\{     a \big| \gamma (Z_t) -  \gamma(Y_t) \big|^{a-2} & \big( \gamma (Z_t) -  \gamma(Y_t) \big) \cdot
   \Big( \lambda \big[U(Z_t) - U(Y_t) \big]  +  A(Z_t - Y_t)  \Big)  \, \dd t  \nonumber \\
& \quad + \dd M_t  + C_{a,d}  \big| \gamma (Z_t) - \gamma(Y_t) \big|^{a} \,  \dd B_t \Big\}\, .  \nonumber
\end{align} 
The term $e^{-C_{a,d}  B_t}\dd M_t$ is still the differential of a zero-mean martingale. Integrating and taking the expected value we find
\begin{align*}
\E  \Big[ e^{-C_{a,d}  B_t} \big| \gamma (Z_t) - \gamma(Y_t) \big|^a \Big]  =  \big| \gamma(z) - \gamma(y) \big|^a  &+ \E \Big[  \int_0^t e^{-C_{a,d}  B_s}  a \big| \gamma (Z_s) - \gamma(Y_s) \big|^{a-2} \\
&\hspace{-30mm} \times \big( \gamma (Z_s) - \gamma(Y_s) \big)  \cdot  \Big( \lambda \big[U(Z_s) - 
U(Y_s) \big]  +  A(Z_s- Y_s)    \Big) \, \dd s  \Big]   .
\end{align*} 
Using again the equivalence \eqref{dif} between $|Z_t-Y_t|$ and $|\gamma(Z_t) - \gamma(Y_T)|$ and the fact that $U$ is Lipschitz continuous, this finally provides the following estimate:
\begin{equation*}
\E \Big[ e^{-C_{a,d}  B_t} \big| Z_t - Y_t \big|^a  \Big]  \le C \Big\{ |z-y|^a +  
\int_0^t  \E \Big[ e^{- C_{a,d}  B_s} \big| Z_s - Y_s \big|^a  \Big] \, \dd s \Big\} \, .
\end{equation*}
By Gr\"onwall's inequality, there exists a finite constant $C'$ 
such that
\begin{equation}\label{Z1-Z2-a}
\E \Big[ e^{-C_{a,d}  B_t} \big| Z_t - Y_t \big|^a \Big] \le  C'  \big|z-y \big|^a \, .
\end{equation}
Using that $B_t$ is increasing and a.s. $B_T <\infty$, taking $z=y$ 
we get for any fixed $t\in[0,T]$ that $\PP \big( Z_t \neq Y_t \big) =0$. 
Strong uniqueness follows by the continuity of trajectories. This completes 
the proof.
 \end{proof}

\begin{corollary}\label{cor Z-Y}
Using the finite moments of the exponential of the process $B_t$, we can also 
prove that for any $a\ge2$,
\begin{equation}\label{Z1-Z2}
\E \Big[ \big| Z_t - Y_t \big|^a \Big] \le  C  \big|z-y \big|^a \, .
\end{equation}
\end{corollary}

\begin{proof} Using H\"older's inequality and for an appropriate constant $c$, we have
\begin{align*}
\E \Big[    \big| Z_t - Y_t \big|^a \Big] &= \E \Big[  e^{c B_t} e^{-c B_t}  
\big| Z_t - Y_t \big|^a \Big] \\
& \le C \Big( \E \Big[ e^{-2c B_t} \big| Z_t - Y_t \big|^{2a} \Big] \Big)^{1/2} 
\le  C  \big|z-y \big|^a \, .
\end{align*}  \end{proof}

 \subsection{Stochastic Flow }
Many of the proofs of results contained in this section follow closely the proofs of \cite[Chapter II.2]{Ku84} or \cite[Chapter 4.5]{Ku90}. To avoid reporting lengthy computations from those references, we will often content ourselves with describing how to adapt the classical proofs to our setting.

 We stress that the main ingredient is the quantitative control on the continuous dependence on the initial data of solutions of the SDE \eqref{eq Z}, which was already obtained in Corollary \ref{cor Z-Y}.

We will repeatedly use the transformation $\gamma$ introduced in 
\eqref{def phi} and the It\^o formula \eqref{eq phi(Z)}, which we rewrite as
$$ \dd \gamma(Z_t^z) =\widetilde b(Z_t) \, \dd t + \widetilde \sigma(Z_t) \cdot \dd W_t \, ,$$
where $\widetilde b(z) =  \lambda U(z) + A z$ is Lipschitz continuous and $ \widetilde \sigma (z) = \big(D U (z)  + \mathbb{I}_{2d} \big) R$ is bounded.

\subsubsection{Continuity}

\begin{lemma}\label{lemma 2.3 Kunita}
Let $a$ be any real number.  
Then there is a positive constant $C_{a}$ independent of  $t\in[0,T]$ and $z\in\R^{2d}$ such that 
\begin{equation}
\E\Big[ \Big(1 + \big| Z_{t}^z\big|^2 \Big)^a \Big] \le C_{a,d} \big(1 + |z|^2 \big)^a.
\end{equation}
\end{lemma}

\begin{proof}
Using the boundedness of the solution $U$ of the PDE \eqref{PDE-lambda} (see \eqref{stima nablaU}) one can show the equivalence 
\begin{equation*}\label{equivalence Z-gamma}
c (1+|z|^2) \le 1+ |\gamma (z)|^2 \le C(1+|z|^2)\, .
\end{equation*}
Set $\gamma_t = \gamma(Z_t^z)$. Then, it is enough to prove that $\E\big[ \big(1 + | \gamma_{t}|^2 \big)^a \big] \le C_{a,d} \big(1 + |\gamma(z)|^2 \big)^a$. Set $f(z):=(1 + |z|^2)$. The idea is to apply It\^o formula to $g(\gamma_{t})$, where $g(z)=f^a(z)$. Since
\begin{align*}
\frac{\partial g}{\partial z_i} (z) &= 2a f^{a-1}(z) z_i \,, \qquad
\frac{\partial^2 g}{\partial z_i\, \partial z_j} (z) = 4a(a-1)f^{a-2}(z) z_i z_j + 2a f^{a-1}(z) \delta_{i,j} \,,
\end{align*}
we see that
\begin{align}\label{F(Y)}
g(\gamma_{t}) - g(\gamma(z)) =& \ 2a \int_0^t f^{a-1}(\gamma_{s})\, \gamma_s \cdot  \widetilde\sigma(\gamma_s) \cdot \dd W_s + 2a \int_0^t f^{a-1}(\gamma_s) \, \gamma_s \cdot \widetilde b(\gamma_s) \, \dd s \\
&+ \sum_{i,j,k=1}^{2d} \int_0^t 2a f^{a-2}(\gamma_s) \Big[ (a-1) \gamma_s^i \gamma_s^j + \delta_{i,j} f(\gamma_s) \Big]\widetilde\sigma^{k,i}(\gamma_s) \widetilde\sigma^{k,j}(\gamma_s)\,  \dd s. \nonumber  
\end{align}
Here we have used the relation $
\dd \langle \gamma_t,\gamma_t \rangle = \widetilde\sigma(\gamma_t)\, \widetilde\sigma^t(\gamma_t)\, \dd t \, . 
$
Since $\gamma_t$ has finite moments, the first term on the right hand side of (\ref{F(Y)}) is a martingale with zero mean. Note that $f(z)\ge1$, so that $f^{a-1}\le f^a $ and $|z|\le f^{1/2}(z)$. Moreover, since $\widetilde\sigma$ is bounded and $\widetilde b$ is Lipschitz continuous, $|\widetilde b(z)|\le C(1+|z|)\le C f^{1/2}(z)$. Using all this, we can see that the second and third term on the right hand side of (\ref{F(Y)}) are dominated by a constant times $\int_0^t g(\gamma_s)\, \dd s$. Therefore, taking expectations in (\ref{F(Y)}) we have
\begin{equation*}
\E\big[ g(\gamma_t) \big] - g(\gamma_0) \le C_{a,d} \int_0^t \E\big[ g(\gamma_s)\big] \dd s \, ,
\end{equation*}
and the result follows by Gr\"onwall's lemma. \end{proof}
\smallskip

\begin{proposition}\label{Flow-continuity}
Let $Z_t^z$ be the unique strong solution to the SDE \eqref{eq Z} given by Theorem \ref{strong 1!} and starting from the point $z\in \R^{2d}$. For any $a>2$, $s,t\in[0,T]$ and $z,y \in \R^{2d} $ we have
\begin{equation} \label{stima flow-continuity}
\E\Big[ \big| Z_t^z - Z_s^y \big|^a \Big] \le C_{a,d,\lambda,T} \Big\{ |z-y|^a + \big( 1+ |z|^a+|y|^a \big) |t-s|^{a/2} \Big\} \, .
\end{equation}
\end{proposition}

\begin{proof}
Assume $t>s$. It suffice to show that
\begin{align}
\E\Big[  \big| Z_s^z - Z_s^y \big|^a \Big] &\le C  |z-y|^a      \label{st1} \\
\E\Big[  \big| Z_t^z - Z_s^z \big|^a \Big] &\le C  \big( 1+ |z|^a \big)  |t-s|^{a/2} \, .  \label{st2}
\end{align}
The first inequality was obtained in Corollary \ref{cor Z-Y}. To prove the second inequality we use the equivalence \eqref{dif} between $Z_t$ and $\gamma(Z_t)$. We use the It\^o formula  \eqref{eq phi(Z)} for $\gamma(Z_t)$ and $\gamma(Z_s)$: we can control the differences of the first and last term using the fact that $U$ and $DU$ are bounded, together with Burkholder's inequality 
\begin{align*}
\E\Big[ \big| Z_t^z & - Z_s^z \big|^a \Big]  \le C \E \Big[ \big| \gamma(Z_t^z) - \gamma(Z_s^z) \big|^a \Big] \\
 &\le C_{a,d} \Big\{  \Big| \int_s^t \Big[  \| \lambda U \|_{\infty}^2  \dd r \Big|^{a/2} \hspace{-1mm} +  \E \Big[ \Big|  \int_s^t  A Z_r^z \, \dd r \Big|^{a} \Big]  + \|DU\|_{\infty} \, \E \Big[\big| R (W_t - W_s) \big|^a  \Big] \Big\} 
 \end{align*} 
and for the linear part we use H\"older's inequality and Lemma \ref{lemma 2.3 Kunita}:
$$     \E \Big[ \Big|  \int_s^t  A Z_r^z \, \dd r \Big|^{a} \Big]  \le (t-s)^{a/2} \,  \E \Big[   \int_s^t  |A Z_r^z|^{a} \, \dd r \Big]  \le C (t-s)^{a/2} \int_s^t (1+|z^2|)^{a/2} \, \dd r  \, .$$
\end{proof}

Applying Kolmogorov's regularity theorem (see \cite[Theorem I.10.3]{Ku84}), we immediately obtain the following
\begin{theorem}\label{teo continuity}
The family of random variables $(Z_t^z)$, $t \in [0,T]$, $z \in \R^d,$ admits a modification which is locally $\alpha$-H\"older continuous in $z$ for any $\alpha<1$ and $\beta$-H\"older continuous in $t$ for any $\beta<1/2$. 
\end{theorem}

From now on, we shall always use the continuous modification of $Z$ provided by this theorem.

\subsubsection{Injectivity and Surjectivity}

The proofs of the injectivity and surjectivity are inspired by \cite{Ku84} and are  similar to the ones given in Section 5 of  \cite{FF13a}. Thus  proofs of the main results in this section are   given in Appendix. 

\vv

To obtain the injectivity of the flow, we review the computations of Proposition \ref{Flow-continuity}: we now want to allow the exponent $a$ to be negative. The proofs of the following two lemmas are given in Appendix.

\begin{lemma}\label{lemma 2.4 Kunita}
Let $a$ be any real number and $\varepsilon>0$. Then there is a positive constant $C_{a,d}$ (independent of $\varepsilon$) such that for any $t\in[0,T]$ and $z,y \in\R^{2d}$
\begin{equation} \label{w INI flow}
\E\Big[ \Big(\varepsilon + \big|Z_{t}^z - Z_{t}^y\big|^2 \Big)^a\Big] \le C_{a,d} \big( \varepsilon + |z-y|^2 \big)^a.
\end{equation}
\end{lemma}
 \smallskip

\begin{corollary}\label{cor flow INI}
Let $\varepsilon$ tend to zero in Lemma \ref{lemma 2.4 Kunita}. Then, by monotone convergence, we have:
\begin{equation}\label{w INI flow,e=0}
\E\Big[\, \big|Z_{t}^z - Z_{t}^y \big|^a \Big]\le C_{a,d} |z-y|^a \, .
\end{equation}
Considering the case $a<0$ we get that  $z\neq y \Rightarrow 
Z_{t}^z\neq  Z_{t}^y$ a.s. for any $t\in[0,T]$. 
\end{corollary}
Kunita in \cite{Ku84} calls the previous property ``weak injectivity''. This intermediate result allows to obtain the following lemma.

\begin{lemma}\label{lemma 4.1 Kunita}
For $t\in[0,T]$ set $\eta_t(z,y):= |Z_{t}^z - Z_{t}^y|^{-1}$. 
Then for any $a>2$ there exists a constant $C=C_{a,d,\lambda,T}$ such that for any $\delta>0$
\begin{align*}\label{tesi lemma 4.1 kunita}
\E\Big[ \big|\eta_t(z,y) - \eta_{t'}(z',y')\big|^a \Big ] \le C\, \delta^{-2a} \Big[ & |z-z'|^a + |y-y' |^a \\
&+ \big(1+ |z|^a + |z'|^a + |y|^a + |y'|^a \big) |t-t'|^{a/2} \Big]  \nonumber
\end{align*}
holds for any $t,t'\in[0,T]$ and $|z-z'|\ge \delta$,  $|y-y'|\ge \delta$.
\end{lemma}
  
\smallskip

\begin{theorem}\label{teo flow 1-1}
The map $Z_{t}:\R^{2d}\rightarrow\R^{2d}$  is one to one, for any $t\in[0,T]$ almost surely.
\end{theorem}

\begin{proof} Take $a/2 > 2 (d+1)$ in Lemma \ref{lemma 4.1 Kunita}. Kolmogorov's theorem states that $\eta_t(z,y)$ is continuous in $(t,z,y)$ in the domain $\{(t,z,y)|\, t\in[0,T], |z-y|\ge\delta\}$. Since $\delta$ is arbitrary, it is also continuous in the domain $\mathcal D:=\{(t,z,y)|\,t\in[0,T], z\neq y\}$. Note that $\mathcal D$ has at most two connected components, both intersecting the hyperplane $\{t=0\}$. Then, since $\eta_0$ is finite, $\eta_t$ must be finite on all of $\mathcal D$. Therefore, if $z\neq y$, $Z_{t}^z\neq Z_{t}^y$, and the theorem is proved.
\end{proof}

\vv Surjectivity will follow from the next lemma which is similar to \cite[Lemma II.4.2]{Ku84}.  Theorem \ref{teo flow onto} below can be proved using an homotopy argument, as in \cite[pag. 226]{Ku84}. Both proofs will be given in Appendix.

\begin{lemma}\label{lemma 4.2 Kunita}
Let $\widehat{\R}^{2d}$ be the one point compactification (Alexandrov compactification) of $\R^{2d}$. For $z \in\R^{2d}\backslash \{0\}$ set $\widehat{z}:=z/|z|^2$, $\widehat{z}:=\infty$ for $z=0$ and define for every $t\in[0,T]$ 
\begin{equation*}
\eta_t(\widehat{z}):=\left \{ \begin{array} [2]{c}
\dfrac{1}{1+|Z_t^z|}\quad if \ z\in\R^{2d},\\
\quad 0\quad\quad \quad if \ \widehat{z}=0 \, .
\end{array} \right. 
\end{equation*}
Then, for any positive $a$ there exists a constant $C=C_{a,d,\lambda,T}$ such that
\begin{equation*}
\E\Big[ \big| \eta_t(\widehat{z})-\eta_{t'}(\widehat{y}) \big|^a \Big] \le C \Big( |\widehat{z} - \widehat{y}|^a + |t-t'|^{a/2} \Big) \, ,\;\;\; \widehat y,\,\widehat z \in \widehat\R^{2d}, \; t,\, t' \in [0,T] \, .
\end{equation*}
\end{lemma}

\begin{theorem}\label{teo flow onto}
The map $Z_{t}:\R^{2d}\rightarrow\R^{2d}$ is onto for any $t\in[0,T]$ almost surely.
\end{theorem}

\subsubsection{The flow}

We resume the results we have obtained so far for the flow associated to the SDE  \eqref{eq Z} in the following theorem.

\begin{theorem}\label{teo-flow}
The unique strong solution $Z_t = (X_t, V_t)$ of the SDE \eqref{eq-SDE} defines a stochastic flow of H\"older continuous homeomorphisms $\phi_t$.
\end{theorem}

\begin{proof}
The map $Z_{t}(\omega)$ is H\"older continuous, see Theorem \ref{teo continuity}, it is one to one by Theorem \ref{teo flow 1-1} and it is onto by Theorem \ref{teo flow onto}. Hence the inverse map $\big(Z_{t}(\omega)\big)^{-1}$ is well defined, one to one and onto. We claim that it is also continuous. Indeed, since the map $\widehat{Z}_{t}(\omega)$ is one to one and continuous from the compact space $\widehat{\R^{2d}}$ into itself, it is a closed map. Hence the inverse map $\big(\widehat{Z}_{t}(\omega)\big)^{-1}$ is continuous, and so is its restriction to $\R^{2d}$.
\end{proof}

\subsection{Regularity of the derivatives}      \label{subsec: derivatives}
Although F is not even weakly differentiable, from the reformulation \eqref{forse} of equation \eqref{eq Z} it is reasonable to expect differentiability of the flow, since the derivatives $D X_t$, $D V_t$ with respect to the initial conditions $(x,v)$ formally solve suitable SDEs with well-defined, integrable coefficients. We have the following result.

\begin{theorem}\label{teo SDE-differentiable flow}
Let $\phi_t(z)$ be the flow associated to \eqref{eq Z} provided by Theorem \ref{teo-flow}.  
Then, for any $t \in [0,T]$, $\PP$-a.s., 
  the random variable  $\phi_t(z)$ admits a weak distributional  derivative
 with respect to $z$; moreover
 $D_z \phi_t \in L^p_{loc}(\Omega \times \R^{2d}) $, for any  $p \ge1$. 
\end{theorem}

\begin{proof}  
{\it Step 1.} (Bounds on difference quotients) It is sufficient to prove the existence and regularity of $D_{z_i}  \phi_t$ for some fixed $i\in \{1, \dots ,2d\}$. We omit to write $i$ and set $e=e_i$. 

Introduce for every $h>0$ the stochastic processes
\begin{equation} \label{eq theta-xi}
\theta_t^h (z) = \frac{\phi_t(z+he) - \phi_t(z)}{h} \, ,\qquad  \qquad \xi_t^h (z) = \frac{ \gamma\big(\phi_t(z+he)  \big) -  \gamma\big( \phi_t(z) \big) }{h} \, ,
\end{equation}
where $\gamma (z)= z+U(z)$ as in \eqref{def phi}. It is clear that they have finite moments of all orders because $\phi$ and $\gamma(\phi)$ do. The two processes are also equivalent in the sense that there exist constants $C_1,C_2$ such that
\begin{equation}\label{xi-theta}
 C_1 | \xi_t^h (z) |  \le |\theta_t^h (z) |  \le C_2 | \xi_t^h (z) | \,. 
 \end{equation}
This follows from \eqref{dif}. To fix the ideas, consider the case $i>d$. We have
\begin{align}
 \xi_t^h &= e + \frac{1}{h} \Big[ U(z+he) - U(z) \Big]    \label{xi-h}    \\  
& \quad +  \int_0^t \frac{\lambda}{h}\Big[ U\big( \phi_s (z+he) \big) - U\big(   \phi_s (z) \big)\Big] + A \theta_t^h (z) \ \dd s     \nonumber    \\
&\quad + \frac{1}{h} \int_0^t \Big[ D_v U\big(  \phi_s (z+he) \big)  - D_v U\big(   \phi_s (z) \big)  \Big] R \cdot \dd W_s \, .    \nonumber      
\end{align} 
Proceeding as in the proof of Theorem \ref{strong 1!} above we have
\begin{align*}
\dd \big| \xi_t^h \big|^p & \le p \big| \xi_t^h \big|^{p-2} \xi_t^h \cdot \Big( \frac{\lambda}{h} \big[ U\big( \phi_t (z+he) \big) - U\big(   \phi_t (z) \big) \big] + A \theta_t^h (z) \Big) \, \dd t \\
& \quad + \frac{p}{h} \big| \xi_t^h \big|^{p-2} \xi_t^h \cdot \Big[ D_v U\Big(  \phi_t (z+he) \Big)  - D_v U\Big(   \phi_t (z) \Big)  \Big]  R \cdot \dd W_t \\
&\quad + \frac{C_{p,d}}{h^2}  \big| \xi_t^h \big|^{p-2} \Big\| D_v U\Big(  \phi_t (z+he) \Big)  - D_v U\Big(   \phi_t (z) \Big) R \Big\|_{HS}^2 \dd t \\
&= p \big| \xi_t^h \big|^{p-2} \xi_t^h \cdot  \Big( \frac{\lambda}{h} \big[ U\big( \phi_t (z+he) \big) - U\big(   \phi_t (z) \big) \big] + A \theta_t^h (z) \Big) \dd t \\
&\quad + \dd M_t^h +   C_{p,d}  \big| \xi_t^h \big|^{p-2} \big| \theta_t^h \big|^2 \, \dd A_t \, ,
\end{align*}
where the process $A_t$ is defined as in \eqref{def At}, but with $Z=\phi(z+he)$ and $Y=\phi(z)$, and for every $h>0$, $\dd M_t^h$ is the differential of a martingale because $DU$ is bounded and $\xi_t^h$ has finite moments. Setting  $C_p = (C_2)^2 C_{p,d}$ we get
\begin{align}
\dd \Big( e^{- C_p A_t} \big| \xi_t^h \big|^p \Big) &= - C_p e^{-C_p A_t} \big| \xi_t^h \big|^p \dd A_t 
+ e^{- C_p A_t} \dd \big( \big| \xi_t^h \big|^p \big) \label{bound derivative}\\
& \hspace{-22mm} \le e^{-C_p A_t} p \big| \xi_t^h \big|^{p-2} \xi_t^h \cdot \Big( \frac{\lambda}{h} \big[ U\big( \phi_t (z+he) \big) - U\big(   \phi_t (z) \big) \big] + A \theta_t^h (z) \Big) \dd t +  e^{-C_p A_t} \dd M_t^h \, .  \nonumber
\end{align}
After integrating and taking expectations we find
\begin{align*}
\E \Big[  e^{- C_p A_t} \big| \xi_t^h \big|^p \Big] & \le \Big| e + \frac{1}{h} \big[ U(z+he) - U(z) \big] \Big|^p \\
& \hspace{-7mm} + \int_0^t \E \Big[ e^{-C_p A_s} p \big| \xi_s^h \big|^{p-2}  \xi_t^h \cdot \Big( \frac{\lambda}{h} \big[ U\big( \phi_s (z+he) \big) - U\big(   \phi_s (z) \big) \big] + A \theta_s^h (z) \Big) \Big] \dd s \\
& \hspace{-15mm} \le C \big( 1+ \|DU\|_{L^\infty(\R^{2d})}^p \big)  +  \int_0^t  C\big(\lambda \|DU\|_{L^\infty(\R^{2d})} + |A| \big) \E \Big[ e^{-C_p A_s} p \big| \xi_s^h \big|^p \Big] \, \dd s \, .
\end{align*}
A similar estimate holds for the case $i\le d$. We now apply Gr\"onwall's inequality and proceeding as in the proof of Corollary \ref{cor Z-Y} we finally get that
\begin{equation}\label{bound theta}
\E \Big[ \big|\theta_t^h \big|^p \Big] \le C \E \Big[ \big|\xi_t^h \big|^p \Big] \le C_{p,d,T,\lambda} <\infty \, .
\end{equation}

{\it Step 2.} (Derivative of the Flow)
Remark that, due to the boundedness of $DU$, the bound \eqref{bound theta} is uniform in $h$ and $z$, and we get
\begin{equation}\label{stima theta}
\sup_{z\in \R^{2d}} \sup_{h\in(0,1]} \E \Big[ \big|\theta_t^h \big|^p \Big]  \le C_{p,d,T,\lambda} <\infty \, .
\end{equation} 
We can then apply \cite[Corollary 3.5]{BF13} and obtain the existence of the weak derivative for the flow $D\phi_t \in L^p_{loc}(\Omega\times \R^{2d})$.
\end{proof}

\begin{remark} 
Since the bound \eqref{bound theta} is also uniform in time, applying \cite[Theorem 3.6]{BF13} one would also get the existence of the weak derivative as a process $D\phi_t$ belonging to $L^p_{loc}( [0,T] \times \R^{2d})$ with probability one, and the weak convergence $\theta_t^h \rightharpoonup D\phi_t$ in $L^p_{loc}(\Omega\times[0,T]\times\R^{2d} )$.
\end{remark}

\section{Stochastic kinetic equation}\label{sec: SKE}

We present here results on the stochastic kinetic equation \eqref{SPDEintro}. The first result concerns existence of  solutions with a certain Sobolev regularity (see Theorem \ref{Theo regul}). The second one is about  uniqueness of solutions (see Theorem \ref{teo uniqueness}).

We will use the results of the previous sections together with results similar to the ones given in  \cite{FF13b} to approximate the flow associated to the equation of characteristics. We report them in the Appendix for the sake of completeness. To prove that some degree of Sobolev regularity of the initial condition is preserved  on has to deal with weakly differentiable solutions, according to the definition introduced in \cite{FF13b} for solutions of the stochastic transport equation. 

Recall that, as observed in Section \ref{sec preliminaries}, by point 2 of the next definition and Sobolev embedding, weakly differentiable solutions of the stochastic kinetic equation are a.s. continuous in the space variable, for every $t\in[0,T]$; this is in contrast with  the deterministic kinetic equation, where solutions can be discontinuous (see Proposition \ref{Prop 2 Appendix}). In the sequel, given a Banach space $E$ we denote by 
$C^0\big( [0,T] ; E   \big) $ the Banach space of all continuous functions from $[0,T]$ into $E$ endowed with the supremum norm.

\begin{definition}\label{def: weak diff sol} 
Assume that $F$ satisfies Hypothesis \ref{hyp-holder}. We say
that $f$ is a weakly differentiable solution of the stochastic kinetic equation \eqref{SPDEintro} if

\begin{enumerate}
\item $f:\Omega\times\left[  0,T\right]  \times\mathbb{R}^{2d}\rightarrow
\mathbb{R}$ is measurable, $\int_{\R^{2d}}f\left(  t,z\right)  \varphi\left(  z\right)
\dd z$ (well defined by property 2 below) is progressively measurable for each
$\varphi\in C_{c}^{\infty}\left(  \mathbb{R}^{2d}\right) ; $

\item $\PP\left(  f\left(  t,\cdot\right)  \in\cap_{r\geq1}W_{loc}^{1,r}\left(
\mathbb{R}^{2d}\right)  \right)  =1$ for every $t\in\left[  0,T\right]  $ and 
both $f$ and $D f$ are in $\cap_{r\geq1} C^0\big( [0,T] ; L^{r} (\Omega \times \R^{2d}) \big) ;$

\item setting $b(z) =  A\cdot z + B(z), $ $b:\R^{2d}\to \R^{2d}$, see \eqref{ABQ}, 
for every $\varphi\in C_{c}^{\infty}\left(  \mathbb{R}^{2d}\right)  $ and
$t\in\left[  0,T\right]  $, with probability one, one has%
\begin{align*}
&  \int_{\R^{2d}}f(t,z)  \varphi(z) \, \dd z + \int_{0}^{t}\int_{\R^{2d}} b(z)  \cdot D f(s,z) \varphi(z) \,  \dd z \dd s\\
&  =\int_{\R^{2d}}f_{0} (z)  \varphi(z) \, \dd z +  \sum_{i=1}^{d} \int_{0}^{t}\left(  \int_{\R^{2d}}f \left(  s,z\right)  \partial_{v_{i}}  \varphi\left(
z\right) \, \dd z \right)  \dd W_{s}^{i}\\
&  +\frac{1}{2}\int_{0}^{t}\int_{\R^{2d}}f(s,z) \Delta_v \varphi (z) \, \dd z \dd s \, .
\end{align*}
\end{enumerate}
\end{definition}

\begin{remark}
The process $s\mapsto Y_{s}^{i}:=\int_{\R^{2d}}f\left(  s,z\right)  \partial_{v_{i}} \varphi ( z)\, \dd z$ is progressively measurable by property 1 and $\int_{0}^{T}\left\vert Y_{s}^{i}\right\vert ^{2} \dd s < \infty$ $\PP$-a.s. by property 2, hence the It\^{o} integral is well defined.
\end{remark}

\begin{remark}
The term $\int_{0}^{t}\int_{\R^{2d}}b\left( z \right)  \cdot D f\left(
s,z\right)  \varphi\left(  z\right) \, \dd z \dd s$ is well defined with probability
one because of the integrability properties of $b$ (assumptions) and $D f$ (property 2).
\end{remark}

In the next result the inverse of $\phi_t$ will be denoted by $\phi_0^t$.

\begin{theorem}\label{Theo regul}
If $F$ satisfies Hypothesis \ref{hyp-holder} and $f_{0}\in \cap_{r\geq1} W^{1,r} ( \mathbb{R}^{2d})$, then $f\left(  t,z\right)  :=f_{0}\left(  \phi_{0}^{t}(z) \right)  $ is a weakly differentiable solution of the stochastic kinetic equation \eqref{SPDEintro}.
\end{theorem}

\begin{proof}
The proof follows the one of \cite[Theorem 10]{FF13b}. We divide it into several steps.

\medskip

\textbf{Step 1} (preparation). The random field $(\omega,t,z)
\mapsto f_{0}\left(  \phi_0^t(z) (\omega) \right)  $ is jointly
measurable and $( \omega,t)  \mapsto\int_{\R^{2d}}f_0 \left( \phi_0
^t(z)(\omega) \right)  \varphi\left(z\right) \,\dd z$ is progressively
measurable for each $\varphi\in C_{c}^{\infty}(  \mathbb{R}^{2d})
$. Hence part 1 of Definition \ref{def: weak diff sol} is true. To prove part 2 and 3 we approximate $f\left(  t,z\right)  $ by smooth fields $f_{n}\left(  t,z\right)  $.

Let $f_{0,n}$ be a sequence of smooth functions which converges to $f_{0}$ in $W^{1,r}(\R^{2d})$, for any $r \ge 1$, and so  uniformly on $\mathbb{R}^{2d}$ by the Sobolev embedding. This can be done for instance by using standard convolution with mollifiers. Moreover suppose that $F_n$ are smooth approximations converging to $F$ in $L^p(\R^{2d})$ ($p$ is given in Hypothesis \ref{hyp-holder}), let $\phi_{t,n}$ be the regular stochastic flow generated by the SDE \eqref{eq Z2} where $B$ is replaced by $B_n = R F_n$ and let $\phi_{0,n}^t$ be the inverse flow.  Then
$f_{n}\left( t,z \right)  :=  f_{0,n}\left(  \phi_{0,n}^{t}\left( z \right) \right)$
 is a smooth solution of
$$ 
 \dd f_n = - \big( v\cdot D_x f_n + F_n \cdot D_v f_n \big) \, \dd t - D_v f_n \circ \dd W_t  \, 
$$ 
 and thus for every $\varphi\in C_{c}^{\infty}(  \mathbb{R}^{2d})  $, $t\in\left[  0,T\right]  $ and bounded r.v. $Y$, it satisfies  
\begin{align}
  \E \bigg[  Y & \int_{\R^{2d}}f_{n} (  t,z )  \varphi ( z ) \, \dd z \bigg]
+ \E \left[  Y\int_{0}^{t}\int_{\R^{2d}}b_{n}\left( z \right)  \cdot D f_{n}\left(
s,z\right)  \varphi\left(  z\right)  \dd z \dd s \right]     \nonumber \\
&  =  \E \Big [ Y \int_{\R^{2d}}f_{0,n}\left(  z\right)  \varphi\left(  z\right)  \dd z \Big]+  \sum_{i=1}
^{d}  \E \left[  Y\int_{0}^{t}\left(  \int_{\R^{2d}}f_{n}\left(  s,z\right)  \partial
_{v_{i}}\varphi\left(  z\right)  \dd z \right)  \dd W_{s}^{i} \right]     \nonumber \\
&\hspace{5mm} +\frac{1}{2} \E \left[  Y\int_{0}^{t}\int_{\R^{2d}}f_{n}\left(  s,z\right)
\Delta_v \varphi\left(  z\right)  \dd z \dd s \right]  .     \label{very weak formulation}%
\end{align}
We shall pass to the limit in each one of these terms. We are forced to use
this very weak convergence due to the term
\begin {equation} \label{div1}
\E \left[  Y \int_{0}^{t}\int_{\R^{2d}}b_{n}\left( z\right)  \cdot D f_{n}\left(
s,z\right)  \varphi\left(  z\right)  \dd z \dd s \right],
\end{equation}
where we may only use weak convergence of $D f_{n}$.

\medskip

\textbf{Step 2} (convergence of $f_{n}$ to $f$). We claim that, uniformly in $n$ and for every $r\ge1$,
\begin{align}
\sup_{t\in[0,T]} \int_{\R^{2d}} \E \Big[ |f_n(t,z)|^r \Big] \dd z \le C_r \ , \label{stim unif un}\\
\sup_{t\in[0,T]} \int_{\R^{2d}} \E \Big[ |D f_n(t,z)|^r \Big] \dd z \le C_r \ . \label{sima unif nabla un}
\end{align}
Let us show how to prove the second bound; the first one can be obtained in the same way. The key estimate is the bound \eqref{property 2} on the derivative of the flow, which is proved in Appendix.
We use the representation formula for $f_n$ and the H\"older inequality to obtain
\begin{equation*}
 \bigg(\int_{\R^{2d}}   \E \Big[ |D f_n(t,z)|^r \Big] \dd z \bigg)^2   \le   \sup_{z\in \R^{2d}} \E \big[ | D \phi_{0,n}^t(z)|^{2r} \big]    \int_{\R^{2d}}  \E \Big[ \big| D f_{0,n} \big( \phi_{0,n}^{t}(z) \big) \big|^{2r} \Big]  \dd z \, .
\end{equation*}
The first term on the right-hand side can be uniformly bounded using Lemma \ref{lemma Df}. Also the last integral can be bounded uniformly: changing variables (all functions are regular) we get 
\begin{equation*}
\int_{\R^{2d}}   \E \Big[ \big| D f_{0,n} \big( \phi_{0,n}^{t}(z) \big) \big|^{2r} \Big]  \dd z = \int_{\R^{2d}}   \big| D f_{0,n} (y) \big|^{2r} \, \E\big[  J_{\phi_{t,n}}(y)  \big]  \dd y \, ,  
\end{equation*}
where $J_{\phi_{t,n}} (y) $ is the Jacobian determinant of $\phi_{t,n}(y)$. Then we conclude using again the H\"older inequality, (\ref{property 2}) and the boundedness  of $(f_{0,n})$ in $W^{1,r}(\R^{2d})$ (for every $r\ge1$). Remark that all the bounds obtained are uniform in $n$ and $t$.

We can now consider the convergence of $f_n$ to $f$. Let us first prove that, given $t\in\left[  0,T\right]  $ and $\varphi\in C_{c}^{\infty}(\mathbb{R}^{2d})  $,
\begin{equation}
\PP-\lim_{n\rightarrow\infty}\int_{\mathbb{R}^{2d}} f_{n}\left(  t,z\right)
\varphi\left(  z\right)  \dd z  =  \int_{\mathbb{R}^{2d}}  f\left(  t,z\right)
\varphi\left(  z\right)  \dd z \, \label{convergence 1}
\end{equation}
(convergence in probability).  
Using the representation formulas $f_n = f_{0,n}(\phi_{0,n}^t)$, $f = f_0(\phi_{0}^t)$ and
Sobolev embedding $W^{1,4d}\hookrightarrow C^{1/2}$  we have (Supp$(\varphi) \subset B_R$ where $B_R$ is the ball of radius $R>0$ and center $0$) 
\begin{align*}
\left\vert \int_{\mathbb{R}^{2d}}\left(  f_{n}\left(  t,z\right)  - f\left(
t,z\right)  \right)  \varphi\left(  z\right)  \dd z \right\vert  &  \leq\left\Vert
f_{0,n} - f_{0} \right\Vert _{L^{\infty}(\R^{2d})}\left\Vert \varphi\right\Vert _{L^{1}(\R^{2d})}\\
&  + C \left\Vert \varphi\right\Vert _{L^{\infty}(\R^{2d})}\int_{B_R
}\left\vert \phi_{0,n}^{t}\left(  z\right)  -\phi_{0}^{t}\left(  z\right)
\right\vert^{1/2} \dd z \, .  
\end{align*} 
The first term converges to zero by the uniform convergence of $f_{0,n}$ to 
$f_{0}$. From Lemma \ref{lemma convergenza SDE} we get 
\[
\lim_{n\rightarrow\infty}  \E \left[  \int_{B_R  }\left\vert
\phi_{0,n}^{t}\left(  z\right)  -\phi_{0}^{t}\left(  z\right)  \right\vert
\dd z \right]  = 0 \, ,
\]
and the convergence in probability (\ref{convergence 1}) follows.  
This allows to pass to the limit in the first	and in the last term of equation (\ref{very weak formulation}) using the uniform bound (\ref{stim unif un}) and the Vitali convergence theorem.  
Similarly, we can show that, given $\varphi\in C_{c}^{\infty}\left(
\mathbb{R}^{2d}\right)  $,
\begin{equation}
\PP-\lim_{n\rightarrow\infty}\int_{0}^{T}\left\vert \int_{\mathbb{R}^{2d}}\left(
f_{n}\left(  t,z\right)  - f\left(  t,z\right)  \right)  \varphi\left(
z\right)  \dd z \right\vert ^{2}  \dd t  = 0
\label{property 3b}
\end{equation}
and  allows to pass to the limit in the stochastic integral term of  (\ref{very weak formulation}). Hence, one can easily show   convergence of all terms in \eqref{very weak formulation} 
except for the one in \eqref{div1} 
which will be treated in Step 4. 
\smallskip

\textbf{Step 3} (a bound for  $f$). Let us prove property 2 of Definition
\ref{def: weak diff sol}. The key estimate is property \eqref{sima unif nabla un} obtained in the previous step.

Recall we have already obtained the convergence \eqref{convergence 1} and the uniform bound  \eqref{sima unif nabla un} on $D f_n$. We can then apply \cite[Lemma 16]{FF13b} which gives $\PP\big(f(  t,\cdot )  \in W_{loc}^{1,r}(\mathbb{R}^{2d}) \big)  =1$  for any  $r\geq1$ and $t\in\left[  0,T\right]  $, and
\[
\E \left[  \int_{B_R  }\left\vert D f\left(  t,z\right)
\right\vert ^{r}  \dd z \right]  \leq  \underset{n\rightarrow\infty}{\limsup}  \ \E \left[
\int_{B_R  }\left\vert D f_{n}\left(  t,z \right)
\right\vert ^{r}  \dd z \right]  \leq C_r\,,
\]
for every $R>0$ and $t\in[0,T]$. Hence, by monotone convergence we have
\begin{equation} \label{sima unif nabla u}
\sup_{t\in[0,T]} \E \left[  \int_{\R^{2d}  }\left\vert D f\left(  t,z\right)
\right\vert ^{r}  \dd z \right]  \leq C_r \, .
\end{equation}
A similar bound can be proved for $f$ itself using \eqref{stim unif un}, the convergence in probability \eqref{convergence 1} and the  Vitali convergence theorem.

\smallskip

\textbf{Step 4} (passage to the limit). Finally, we prove that we can 
pass to the limit in equation (\ref{very weak formulation}) and deduce that
$f$ satisfies property 3 of Definition \ref{def: weak diff sol}. 
 It remains to consider the term
$\E\big[  Y \int_{0}^{t}\int_{\R^{2d}}b_{n}\left(  s,z\right)  \cdot D f_{n}\left(
s,z \right)  \varphi\left( z \right)  \dd z \dd s \big]  $.  
Since $F_{n}\rightarrow F$ in $L^{p}(  \mathbb{R}^{2d})  $, it is
sufficient to use a suitable weak convergence of $D f_{n}$ to $D f$. Precisely, for $t \in [0,T]$,
\begin{align*}
 \E\Big[  Y \int_{0}^{t}\int_{\R^{2d}} & b_{n}\left( z \right)  \cdot D f_{n}\left(
s,z \right)  \varphi\left(  z \right)  \dd z \dd s \Big] \\
& -  \E \left[  Y  \int_{0}^{t}\int_{\R^{2d}} 
b\left( z \right)  \cdot D f\left(  s,z \right)  \varphi\left( z \right)
\dd z \dd s \right]  =   I_{n}^{\left(  1\right)  }\left(  t\right)  +   I_{n}^{\left(  2\right)
}\left(  t\right) \ ; \\
I_{n}^{\left(  1\right)  }\left(  t\right)    & = \E \left[  Y \int_{0}^{t}%
\int_{\R^{2d}} \big(  F_{n}\left( z \right)  - F\left( z \right)  \big)
\cdot D_v f_{n}\left(  s,z \right)  \varphi\left( z \right)  \dd z \dd s \right] ;  \\
I_{n}^{\left(  2\right)  }\left(  t\right)    & = \E \left[  Y \int_{0}^{t}
\int_{\R^{2d}} \varphi\left(  z \right)  b\left( z \right)  \cdot   \big( Df_{n}\left( s,z \right) - D f \left( s,z \right)  \big)  \dd z \dd s \right]  .
\end{align*}
We have to prove that both $I_{n}^{\left(  1\right)  }\left(  t\right)  $ and
$I_{n}^{\left(  2\right)  }\left(  t\right)  $ converge to zero as
$n\rightarrow\infty$.  By the H\"{o}lder inequality, for all $t\in[0,T]$
\[
I_{n}^{\left(  1\right)  }\left(  t\right)  \leq C\left\Vert F_{n}
- F \right\Vert _{L^{p}(  \mathbb{R}
^{2d}) } \, \sup_{t \in [0,T]}  \E \left[  \big\| D f_{n} (t, \cdot) \big\|_{L^{p^{\prime}}(  \R^{2d} ) }\right] 
\]
where $1/p+1/p^{\prime}=1$ and $C = C_{Y, T, \varphi}$. Thus, from (\ref{sima unif nabla un}), $I_{n}^{\left(  1\right)  }\left(
t\right)  $ converges to zero as $n \to \infty$.
Let us treat $I_{n}^{\left(  2\right)  }\left(  t\right)  $. Using the
integrability properties shown above we can change the order of integration.  The function%
\[ 
h_{n}\left(  s\right)  :=  \E \left[  \int_{\R^{2d}}Y \varphi\left( z \right)   b \left(
z \right)  \cdot   \big(  D f_{n}\left( s,z \right) - D f \left( s,z \right)  \big) \, \dd z \right], \;\; s \in [0,T],
\]
converges to zero as $n\rightarrow\infty$ for almost every $s$ and satisfies
the assumptions of the Vitali convergence theorem (we shall prove these two claims
in Step 5 below). Hence $I_{n}^{\left(  2\right)  }\left(  t\right)  $
converges to zero. 

Now we may pass to the limit in equation (\ref{very weak formulation}) and 
from the arbitrariness of $Y$ we obtain property 3 of Definition \ref{def: weak diff sol}.

\medskip

\textbf{Step 5} (auxiliary facts). We have to prove the two properties of $h_{n}(  s)  $ claimed in Step 4.  
For every $s\in [0,T]$ \cite[Lemma 16]{FF13b} gives
\begin{equation}
\E \left[  \int_{\mathbb{R}^{2d}}\partial_{z_{i}} f\left( s,z \right)
\varphi\left( z\right)  Y \, \dd z \right]  =  \lim_{n\rightarrow\infty}  \E \left[
\int_{\mathbb{R}^{2d}} \partial_{z_{i}} f_{n}\left(  s,z\right)  \varphi\left(z\right)  
Y  \, \dd z \right] \, , \label{convergence 4}%
\end{equation}
for every $\varphi\in C_{c}^{\infty}( \mathbb{R}^{2d})$ and
bounded r.v. $Y$. 
 Since the space $C_{c}^{\infty}(\mathbb{R}^{2d})$ is dense in $L^{p}(  \mathbb{R}^{2d})  $, 
we may extend the convergence property (\ref{convergence 4}) to all
$\varphi\in L^{p}(  \mathbb{R}^{2d})  $ by means of the bounds
(\ref{sima unif nabla un}) and (\ref{sima unif nabla u}), which proves the first claim.

Moreover, for every $\varepsilon>0$ there is a constant $C_{Y, \varepsilon}$ such that (Supp$(\varphi) \subset B_R$)
\begin{align*}
\sup_{n \ge 1}\int_{0}^{T}  h_{n}^{1+\varepsilon}(s) \, \dd s  
 \leq C_{Y,\varepsilon} \left\Vert b \varphi \right\Vert _{L^p}^{1+\varepsilon} \bigg\{ & \left(  \E \int_{0}^{T}\int_{B_R }\big\vert
D f_{n}\left( s,z \right) \big\vert ^{r} \dd z \dd s \right)^{\frac{1+\varepsilon}{r}} \\
& + \left(  \E \int_{0}^{T}\int_{B_R }\big\vert D f \left( s,z\right)  \big\vert ^{r}  \dd z \dd s \right)  ^{\frac{1+\varepsilon}{r}} \bigg\}%
\end{align*}
for a suitable $r$ depending on $\varepsilon$ (we have used H\"{o}lder
inequality; cf. page 1344 in \cite{FF13b}). The bounds (\ref{sima unif nabla un}) and (\ref{sima unif nabla u})
 imply that $\int_{0}^{T}h_{n}%
^{1+\varepsilon}\left(  s\right)  \dd s$ is uniformly bounded, and the Vitali
theorem can be applied. The proof is complete. \ \ \ 
\end{proof}

\smallskip

We now present the uniqueness result for weakly differentiable solutions. The proof seems to be of independent interest.

\begin{theorem}\label{teo uniqueness}
If $F$ satisfies Hypothesis \ref{hyp-holder} and, moreover, $\mathrm{div}_v F \in L^\infty(\R^{2d})$ ($\mathrm{div}_v F$ is understood in distributional sense)
 weakly differentiable solutions are unique.
\end{theorem}
\begin{proof}
By linearity of the equation we just have to show that the only solution starting from $f_0 = 0$ is the trivial one. 
\medskip

{\bf Step 1 } ($f^2$ is a solution).   We  prove that for any solution $f$, the function $f^2$ is still a weak solution of the stochastic kinetic equation. Take test functions of the form $\varphi^n_\zeta (z) = \rho_n (\zeta - z)$, where $(\rho_n)_n$ is a family of standard mollifiers ($\rho_n$ has support in $B_{1/n}$). Let $\zeta = (\xi,\nu) \in \R^{2d}$, $f_n(t,\zeta) = \big(f(t,\cdot) \star \rho_n \big) (\zeta)$. By definition of  solution we get that, $\PP$-a.s.,
\begin{align*}
f_n(t,\zeta)  &+ \int_0^t  b(\zeta) \cdot D f_n (s,\zeta) \, \dd s + \int_0^t D_{v} f_n (t,\zeta)  \circ \dd W_s = \int_0^t R_n (s,\zeta) \, \dd s \, ,  \\
R_n (s,\zeta) &= \int_{\R^{2d}} \big( b(\zeta) - b(z) \big) \cdot D_{z}  f (s,z) \, \rho_n (\zeta-z) \  \dd z\,.
\end{align*}
The functions $f_n$ are smooth in the space variable. For any fixed $\zeta \in \R^{2d}$, by the It\^o formula we get
\begin{align*}
\dd f_n^2 & = 2 f_n \dd f_n \\
	& = -2  f_n b \cdot D f_n \, \dd t   - 2 f_n D_{v} 
	f_n  \circ \dd W_t + 2 f_n R_n \, \dd t \, .
\end{align*}
Now we multiply by   $\varphi \in C_c^{\infty}(\R^{2d})$ and integrate over $\R^{2d}$. Using the It\^o integral we  pass to the limit as $n \to \infty$ and find, $\PP$-a.s.,
\begin{gather}
\int_{\R^{2d}} \hspace{-2mm} f_n^2 (t,\zeta) \varphi (\zeta) \, \dd \zeta   
  - \frac{1}{2}\int_0^t \int_{\R^{2d}} \hspace{-2mm}  f_n^2 (s, \zeta)
  \triangle_v \varphi (\zeta)  \, \dd \zeta \dd s  +  
  \int_0^t \int_{\R^{2d}}  \hspace{-2mm} \varphi(\zeta)\, b(\zeta) \cdot D f_n^2 (s, \zeta) \,  \dd \zeta \dd s   \nonumber  \\
	 -  \int_0^t \int_{\R^{2d}} f_n^2(s, \zeta) D_v \varphi ( \zeta)\, \dd \zeta \cdot \dd W_s  =  2 \int_0^t \int_{\R^{2d}} f_n
(s, \zeta) \,  R_n(s, \zeta)  \varphi(\zeta) \, \dd \zeta \dd s \, .     \label{weak f2}
\end{gather}
Recall that 
$$
b(z) =  A\cdot z + \begin{pmatrix}    0  \\ F (z)   \end{pmatrix} \in \R^{2d}.
$$
Let us fix $t \in [0,T]$. By definition of weakly differentiable solution it is not difficult to  pass to the limit in probability as $n \to \infty$ in all the  terms in the left hand side of \eqref{weak f2}.
Indeed, we can use that,  for every $t\in[0,T]$, $r \ge 1$,  $f_n(t,\cdot) \to f(t,\cdot)$ in $W^{1,r}_{loc}(\R^{2d})$, $\PP$-a.s., together with the bounds 
\begin{align}
\sup_{t\in[0,T]} \int_{\R^{2d}} \E \Big[ |f_n(t,z)|^r \Big] \dd z \le C_r \ , \;\;\; 
\sup_{t\in[0,T]} \int_{\R^{2d}} \E \Big[ |D f_n(t,z)|^r \Big] \dd z \le C_r \ , \label{sera}
\end{align}
and the Vitali theorem.
 For instance, if Supp$(\varphi) \subset B_R$ we have
\begin{gather*}
J_n(t) = \E \int_0^t \int_{\R^{2d}} | f_n^2 (s, \zeta) -  f^2 (s, \zeta)|
  |\triangle_v \varphi (\zeta)|    \, \dd \zeta \dd s \\ \le 
C_{\varphi} \int_0^T  \, \E \int_{B_R} | f_n^2 (s, \zeta) -  f^2 (s, \zeta)|   \, \dd \zeta ds,
\end{gather*}
and, for any $s \in [0,T]$, $\PP$-a.e $\omega$, $ k_n (\omega, s)= \int_{B_R} | f_n^2 (\omega, s, \zeta) -  f^2 (\omega, s, \zeta)|   \, \dd \zeta  \to 0$ as $n \to \infty$.
 From \eqref{sera}  we deduce easily that 
$$
\sup_{n \ge 1} \E \Big[ \int_0^T k_n^2 (s) \, \dd s \Big] < \infty
$$
and so by the Vitali theorem we get
$
\int_0^T  \, \E \big[ \int_{B_R} | f_n^2 (s, \zeta) -  f^2 (s, \zeta)|   \, \dd \zeta \big] \dd s \to 0
$, 
as $n \to \infty$, which implies  $\lim_{n \to \infty} J_n (t)=0$. In order to show that  
$$
\E \Big[ \int_0^t \int_{\R^{2d}} | f_n
(s, \zeta) \,  R_n(s, \zeta)  \varphi(\zeta)| \, \dd \zeta \dd s \Big] 
\le C_{\varphi} \E \Big[ \int_0^t \int_{B_{R}} | f_n
(s, \zeta) \,  R_n(s, \zeta) | \, \dd \zeta \dd s \Big]   \to 0 
$$
as $ n \to \infty $, it is enough to prove that for fixed $\omega$, $\PP$-a.s., and $s \in [0,T]$ we have 
\begin{equation} \label{vita}
 \int_{B_{R}} | f_n
(s, \zeta) \,  R_n(s, \zeta) | \, \dd \zeta  \,
\to 0 \,, 
\end{equation} 
as $n \to \infty$. Indeed once \eqref{vita} is proved, using the bounds  \eqref{sera} and the H\"older inequality we get 
\begin{equation*}
\sup_{n \ge 1} \E \Big[ \int_0^t \Big| \int_{B_{R}} \big| f_n
(s, \zeta) \,  R_n(s, \zeta) \big| \, \dd \zeta \Big|^2 \dd s \Big]
\le C_R \sup_{n \ge 1} \E \Big[ \int_0^t \int_{B_{R}} | f_n
(s, \zeta) \,  R_n(s, \zeta) |^2 \, \dd \zeta \dd s \Big] < \infty \, .
\end{equation*}
Thus we can apply the Vitali theorem and deduce the assertion. 
Let us check \eqref{vita}.  

By Sobolev regularity of weakly differentiable solutions we know that $$
\sup_{n \ge 1} \sup_{\zeta \in B_R}| f_n (s, \zeta)|
= M < \infty \, .
$$
Hence it is enough to prove that 
 $\int_{B_{R}} |  R_n(s, \zeta) |  \dd \zeta  
\to 0 $. 
 Recall that 
$$
R_n (s,\zeta) =  b(\zeta)  \cdot D  f_n (s, \zeta) - 
 [(b   \cdot D  f (s, \cdot)) * \rho_n] (s, \zeta) \,.
$$
Using the fact that $b \in L^p_{loc}(\R^{2d})$, with $p$ given in Hypothesis \ref{hyp-holder}, the H\"older inequality and basic properties of convolutions we have
\begin{gather*}
\int_{B_R} \big|  b(\zeta)  \cdot D  f_n (s, \zeta )  -
b(\zeta)  \cdot D  f (s, \zeta) \big| \, \dd \zeta \;\;  \to\; 0 \, ,
\\
\int_{B_R} \Big| 
\big[(b   \cdot D  f (s, \cdot)) * \rho_n \big] (s, \zeta) -
 b(\zeta)  \cdot D  f (s, \zeta) \Big| \, \dd \zeta \;\; \to\; 0 \, ,
\end{gather*}
as $n \to \infty$. This shows that \eqref{vita} holds. We have proved that  also $f^2$ is a weakly differentiable solution of the stochastic kinetic equation.

\medskip 
{\bf Step 2} ($f$ is identically zero). Due to the integrability properties of $f$, the stochastic integral in It\^o's form is a martingale; it follows that  the function $g(t,z) = \E [f^2(t,z) ]$  belongs to $C^0 ([0,T] ; W^{1,r} (\R^{2d}))$ for any $r\ge1$ and satisfies, for any $\varphi \in C_c^{\infty}(\R^{2d})$,
$$
\int_{\R^{2d}} g(t,z) \varphi(z) \, \dd z + \int_0^t \int_{\R^{2d}} b(z) \cdot D g(s,z) \varphi(z) \, \dd z \dd s = \frac{1}{2} \int_0^t \int_{\R^{2d}} g(s,z) \Delta_v \varphi(z) \, \dd z \dd s \, .
$$
We have, for any $s \in [0,T]$,
\begin{align*}
\int_{\R^{2d}} b(z) \cdot D g(s,z) \varphi(z) \, \dd z
&=
\int_{\R^{2d}} v \cdot D_x g(s,z) \varphi(z) \, \dd z
+
\int_{\R^{2d}} F(z) \cdot D_v g(s,z) \varphi(z) \, \dd z \\
& \hspace{-5mm} = -\int_{\R^{2d}} v \cdot D_x \varphi (z) g(s,z)  \, \dd z
+
\int_{\R^{2d}} F(z) \cdot D_v g(s,z) \varphi(z) \, \dd z.
\end{align*}
Now we fix $\eta \in C_c^{\infty}(\R^{d})$ such that 
$\eta = 1$ on the ball $B_1$ of center 0 and radius 1. By considering the   test functions:
$$
\varphi_{nm} (x,v) = \eta(x/n) \eta (v/m),\;\; (x,v) =z \in \R^{2d},
$$
$n, m \ge 1$, we obtain
\begin{gather*}
 \int_{\R^{2d}} g(t,z) \eta(x/n) \eta (v/m) \, \dd z 
 - \frac{1}{n}\int_0^t  \int_{\R^{2d}} \eta(v/m) v \cdot D \eta (x/n) g(s,z)  \, \dd z \dd s
\\ 
 + \hspace{-0.5mm}
 \int_0^t  \hspace{-0.5mm}
  \int_{\R^{2d}} \hspace{-2mm}  F(z) \cdot D_v g(s,z) \eta(x/n) \eta (v/m)  \dd z \dd s =  \frac{1}{2 m^2} \int_0^t \hspace{-0.5mm} \int_{\R^{2d}} \hspace{-2mm} \eta (x/n) g(s,z) \Delta \eta(v/m)  \dd z \dd s.
\end{gather*}
Now we fix $m \ge 1$ and pass to the limit as $n \to \infty$ by the Lebesgue theorem. We infer
\begin{gather*}
 \int_{\R^{2d}} g(t,z)  \eta (v/m) \, \dd z 
 + 
 \int_0^t 
  \int_{\R^{2d}} F(z) \cdot D_v g(s,z)  \eta (v/m) \, \dd z \dd s
\\ = \frac{1}{2 m^2} \int_0^t \int_{\R^{2d}}  g(s,z) \Delta \eta(v/m) \, \dd z \dd s \, .
  \end{gather*}
Passing to the limit as $m \to \infty$ we arrive at
$$
 \int_{\R^{2d}} g(t,z)  \, \dd z 
 =
 - \int_0^t  
  \int_{\R^{2d}} F(z) \cdot D_v g(s,z)  \, \dd z \dd s \, .
$$
Since in particular $g(t,z) \in C^0 ([0,T] ; W^{1,r} (\R^{2d}))$, with $r = \frac{p}{p-1}$, we obtain
$$
 \int_{\R^{2d}} g(t,z)  \, \dd z 
 =
  \int_0^t    \int_{\R^{2d}} \text{div}_v F(z)  g(s,z)  \, \dd z \dd s
   \le 
\| \text{div}_v F(z) \|_{\infty}    \int_0^t    \int_{\R^{2d}}  g(s,z)  \, \dd z \dd s \, .
$$
Applying the Gr\"onwall lemma we get that $g $ is identically zero and this proves uniqueness for the kinetic equation.
\end{proof}

\section{Appendix}

\begin{proof} [Proof of Proposition \ref{df3}]
  By \eqref{ma1} we only have to prove the first inclusion. 
 Let $f\in W^{s',p} (\R^{2d})$. Recall that 
$$
[f]^2_{ B^{s}_{p, 2}} =\int_{\R^{2d}} \frac{\dd h}{ |h|^{2d + 2 s } } \Big( \int_{\R^{2d}} 
 |f(x+h) - f(x)|^p \, \dd x
 \Big)^{2/p}  .
$$
We have with $\epsilon = \frac{s' -s}{2} >0$ 
\begin{align*}
\int_{|h|\le 1}& \frac{1}{ |h|^{2d + 2 s } } \Big( \int_{\R^{2d}} 
 |f(x+h) - f(x)|^p \, \dd x
 \Big)^{2/p}  \dd h \\
&= \int_{|h|\le 1} \frac{1}{ |h|^{ 2d (\frac{p - 2} {p})  - 
\epsilon} } \Big( \int_{\R^{2d}} 
 \frac{|f(x+h) - f(x)|^p} { |h|^{(\frac{2}{p} 2d +  2s + \epsilon) \frac{p}{2} } } \dd x
 \Big)^{2/p} \dd h \\
&\le c  \Big ( \int_{|h|\le 1}   \dd h \int_{\R^{2d}} 
 \frac{|f(x+h) - f(x)|^p} { |h|^{2d +  (s + \epsilon/2) p  } } \, \dd x
   \Big)^{2/p}
 \, \cdot \, 
\Big (\int_{|h|\le 1} \frac{1}{ |h|^{2d - \, 
 \frac{\epsilon p}{p-2} } } \, \dd h \Big)^{\frac{p-2}{p}} \\
&\le C  [f]^2_{W^{s + \frac{\epsilon}{2},p } }
\end{align*}
Similarly, we have with $0<\epsilon < \min (2 s, \frac{s'-s}{2}) $
\begin{align*}
\int_{|h| > 1} &\frac{1}{ |h|^{2d + 2 s } } \Big( \int_{\R^{2d}} 
 |f(x+h) - f(x)|^p \, \dd x   \Big)^{2/p} \dd h \\
& = \int_{|h|> 1} \frac{1}{ |h|^{ 2d (\frac{p - 2} {p})  + 
\epsilon} } \Big( \int_{\R^{2d}} 
 \frac{|f(x+h) - f(x)|^p} { |h|^{(\frac{2}{p} 2d +  2s - \epsilon) \frac{p}{2} } } \, \dd x
 \Big)^{2/p} \dd h \\
& \le c  \Big ( \int_{|h| > 1}   \dd h \int_{\R^{2d}} 
 \frac{|f(x+h) - f(x)|^p} { |h|^{2d +  (s - \epsilon/2) p  } } \, \dd x    \Big)^{2/p}
 \, \cdot \, 
\Big (\int_{|h|> 1} \frac{1}{ |h|^{2d + \, 
 \frac{\epsilon p}{p-2} } } \, \dd h \Big)^{\frac{p-2}{p}} \\
& \le C  [f]^2_{W^{s - \frac{\epsilon}{2},p } } \le C' \, \| f\|^2_{W^{s + \frac{\epsilon}{2},p } }
\end{align*}
(cf. \eqref{incl} and \eqref{rft}). The proof is complete.
 \end{proof}

\hh
\begin{proof} [Proof of Lemma \ref{lemma 2.4 Kunita}]
  Remark that, due to \eqref{dif}, 
$$
\Big(\varepsilon + \big|Z_{t}^z - Z_{t}^y\big|^2 \Big)^a \le C \Big(\varepsilon + \big|\gamma(Z_{t}^z) - \gamma(Z_{t}^y)\big|^2 \Big)^a \, .
$$
Therefore, we can prove \eqref{w INI flow} for $\gamma(Z_t)$ instead of $Z_t$.

We proceed as in \cite[Lemma II.2.4]{Ku84} or \cite[Lemma 5.4]{FF13a}. Fix any $t\in[0,T]$ and set for $z,y\in\R^{2d}$: $g(z):=f^a(z)$, $f(z):=(\varepsilon + |z|^2)$ and $\eta_t:=\gamma(Z_{t}^z)- \gamma(Z_{t}^y)$. Then, applying It\^o formula we obtain as in the proof of Lemma \ref{lemma 2.3 Kunita}
\begin{align*}
g(\eta_t) - g(\eta_0) 
=  & \, 2a \hspace{-0.5mm} \int_0^t  \hspace{-1mm} f^{a-1}(\eta_s)\,  \eta_s \cdot \Big[\widetilde  b\big(Z_{s}^z\big) - \widetilde b\big(Z_{s}^y\big) \Big] \, \dd s \\
& + 2a \hspace{-0.5mm}  \int_0^t  \hspace{-1mm} f^{a-1}(\eta_s) \, \eta_s \cdot \Big[ \widetilde\sigma\big(Z_{r}^z\big) - \widetilde\sigma\big(Z_{s}^y\big) \Big] \cdot \dd W_s \\
&+ a \sum_{i,j} \int_0^t f^{a-2}(\eta_s) \Big[ f(\eta_s) \delta_{i,j} + 2(a-1)\, \eta_s^i\eta_s^j \Big]\\
&\hspace{2cm} \times \Big[ \Big(\widetilde\sigma\big(Z_{s}^z\big) - \widetilde\sigma\big(Z_{s}^y\big) \Big)\Big( \widetilde\sigma\big(Z_{s}^z\big) - \widetilde\sigma\big(Z_{s}^y\big) \Big)^t\Big]^{i,j} \,\dd s \, .
\end{align*}
Recall that $|z|\le f^{1/2}(z)$ and that the coefficient $\widetilde b$ is Lipschitz continuous: 
$$|\widetilde b(z)-\widetilde b(y) |\le L |z-y| \le C |\gamma(z) - \gamma(y) | \le C f^{1/2}\big( |\gamma(z) - \gamma(y) | \big) \, .$$ 
We can continue with the estimates and obtain  
\begin{align} \label{Mar4}
g(\eta_t) - g (\eta_0) 
&\le 2C|a| \int_0^t f^a(\eta_s)\, \dd s + 2\,a\int_0^t f^{a-1}(\eta_s) \eta_s \Big[ \widetilde\sigma\big(Z_{s}^z\big) - \widetilde\sigma\big(Z_{s}^y\big) \Big] \, \dd W_s \\ 
&\hspace{0,5cm} + C_{a,d}  |a| \int_0^t f^{a-1}(\eta_s) 
\,|\eta_s|^2  \,\dd A_s \, .\nonumber 
\end{align}
Here, $A_t$ is the process introduced and studied in Lemma \ref{lemma exp At}:
 \begin{equation*}
 \int_0^t \big\|\widetilde{\sigma}(Z_{s}^{z}) - \widetilde{\sigma}(Z_{s}^{y})\big\|^2_{HS} \dd s = \int_0^t \big|Z_{s}^{z}-Z_{s}^{y}\big|^2 dA_s \, .
 \end{equation*}
The stochastic integral in (\ref{Mar4}) is a martingale with zero mean ($\widetilde \sigma$ is bounded). Proceeding as in \eqref{Ito 1!}, we get 
\begin{equation*}
 \E\Big[ e^{-A_t} g(\eta_t) \Big] - e^{-A_0}g(\eta_0) \le C_{a,d} \int_0^t \E\Big[ e^{-A_s}g(\eta_s)\Big] \,\dd s \, .
\end{equation*}
By Gr\"onwall's inequality applied to the function $h(t):=\E\big[e^{-A_t} g(\eta_t)\big]$, it follows
\begin{align}\label{Mar3}
\E\Big[ e^{-A_t} \Big(\varepsilon + \big|Z_{t}^z - Z_{t}^y\big|^2 \Big)^a\Big] &\le C \E\Big[ e^{-A_t} g(\eta_t) \Big] \le C_{a,d} \, g(\eta_0) = C_{a,d} \big( \varepsilon + |\gamma(z)-\gamma(y)|^2 \big)^a      \nonumber \\
&\le C_a \big( \varepsilon + |z-y|^2 \big)^a \, .
\end{align}
To complete the proof of the lemma, we manipulate (\ref{Mar3}) using H\"older's inequality and we conclude invoking Lemma \ref{lemma exp At} to bound the term $\E[e^{2A_T}]$: 
\begin{equation*}
\E\Big[  \Big(\varepsilon + \big|Z_{t}^z - Z_{t}^y\big|^2 \Big)^a\Big]^2 \le  \E\Big[ e^{2A_t}\Big] \E\Big[ e^{-2A_t} g^2(\eta_t) \Big] 
 \le C_{a,d} \big(\varepsilon + |z-y|^2 \big)^{2a} .
\end{equation*}
\end{proof}

\begin{proof} [Proof of Lemma \ref{lemma 4.1 Kunita}] We have set here $\eta_t(z,y) := |Z_t^z - Z_t^y|$. To ease notation in the following computations we will write $\eta$ and $\eta'$ for $\eta_t(z,y)$ and $\eta_{t'}(z',y')$ respectively and set $\xi=\eta^{-1}$ and $\xi'=\eta'^{-1}$. Observe that
\begin{align*}
\Big| \eta_t(z,y) - \eta_{t'}(z',y')\Big|^a &=\Big| \frac{1}{\xi} - \frac{1}{\xi'} \Big|^a = \Big| \frac{\xi' - \xi}{\xi\xi'} \Big|^a = |\eta|^a |\eta'|^a\, |\xi'-\xi|^a \nonumber \\
     & =  |\eta|^a |\eta'|^a\, \big| |Z_{t'}^{z'}-Z_{t'}^{y'}| - |Z_t^z-Z_t^y|\,\big|^a \nonumber \\
     & \le |\eta|^a |\eta'|^a\, \big| |Z_{t'}^{z'}-Z_{t}^{z}| + |Z_{t'}^{y'}-Z_t^y|\,\big|^a \nonumber \\
     & \le C_a |\eta|^a |\eta'|^a\,\big( |Z_{t'}^{z'}-Z_{t}^{z}|^a + |Z_{t'}^{y'}-Z_t^y|^a\big). \nonumber
\end{align*}
The first inequality above follows from the triangular inequality as in \cite[Lemma II.4.1]{Ku84}: every side of a quadrilateral is shorter that the sum of the other three. We now take expectations and use H\"older's inequality:
\begin{align*}
\E\Big[\big| \eta_t(z,y) &- \eta_{t'}(z',y')\big|^a \Big]\\
 &\le C_a\, \E\big[ |\eta|^{4a} \big]^{1/4} \E\big[ |\eta'|^{4a} \big]^{1/4} \bigg (\E\Big[ |Z_{t'}^{z'}-Z_{t}^z|^{2a}\Big]^{1/2} + \E\Big[ |Z_{t'}^{y'}-Z_{t}^{y}|^{2a}\Big]^{1/2} \bigg)\\
  & \le C_{a,d,\lambda,T} |z-y|^{-a} |z'-y'|^{-a} \Big( |z-z'|^a + |y-y'|^a \\
  & \hspace{41mm} + \big(1+ |z|^a + |z'|^a + |y|^a + |y'|^a \big) |t-t'|^{a/2} \Big).
\end{align*}
For the last inequality we have used (\ref{w INI flow,e=0}) to estimate the first two terms and Proposition \ref{Flow-continuity} for the last ones.
\end{proof}

\hh
\begin{proof} [Proof of Lemma \ref{lemma 4.2 Kunita}] Consider first the case $z,y\neq\infty$ ($\widehat{z},\widehat{y}\neq0$). In this case, proceeding just as in the proof of Lemma \ref{lemma 4.1 Kunita} we obtain the estimate
\begin{equation*}
\big| \eta_t(\widehat{z})-\eta_{t'}(\widehat{y}) \big|^a = \bigg| \frac{1+|Z_{t'}^y|-1-|Z_{t}^z|}{\big( 1+|Z_{t}^z| \big)\big( 1+|Z_{t'}^y| \big)}\bigg|^a \le \big|\eta_t(\widehat{z})\eta_{t'}(\widehat{y})\big|^a\big| Z_{t}^z-Z_{t'}^y \big|^a .
\end{equation*}
Use again H\"older's inequality, Lemma \ref{lemma 2.3 Kunita} and Proposition \ref{Flow-continuity}:
\begin{align*}
\E\Big[\big| \eta_t(\widehat{z})-\eta_{t'}(\widehat{y}) \big|^a \Big] &\le \E\Big[[\big| \eta_t(\widehat{z})\big|^{4a} \Big]^{1/4}\E\Big[[\big| \eta_{t'}(\widehat{y})\big|^{4a} \Big]^{1/4} \E\Big[ \big| Z_{t}^z-Z_{t'}^y \big|^{2a} \Big]^{1/2}\\
      &\le C (1+|z|)^{-a}(1+|y|)^{-a}\Big( |z-y|^a + \big(1+|z|^a + |y|^a \big)|t-t'|^{a/2} \Big)  \\
      &\le C \Big( |\widehat{z}-\widehat{y}|^a + |t-t'|^{a/2} \Big).
\end{align*}
The last inequality above holds if $z$ and $y$ are both finite and is a consequence of the inequality
\begin{equation*} 
(1+|z|)^{-1} (1+|y|)^{-1} |z-y| \le |\widehat{z}-\widehat{y}|=\Big| \frac{z}{|z|^2}-\frac{y}{|y|^2} \Big|.
\end{equation*} 
The case $z=\infty$ (or $y=\infty$) is even easier, since $\eta_t(\widehat{z})=0$ and Lemma \ref{lemma 2.3 Kunita} again implies
\begin{equation*}
\E\Big[\big| \eta_{t'}(\widehat{y}) \big|^a \Big] \le C_{a,d} \big(1+|y|\big)^{-a} \le C_{a,d} |\widehat{y}|^a.
\end{equation*}
The lemma is proved.
\end{proof}

\hh 
 \begin{proof}[Proof of Theorem \ref{teo flow onto}] 
Take $a>2(2d+3)$ in Lemma \ref{lemma 4.2 Kunita}. Then, by Kolmogorov's theorem, $\eta_t(\widehat{z})$ is continuous at $\widehat{z}=0$. Therefore, $Z_{t}^z$ can be extended to a continuous map from $\widehat{\R^{2d}}$ into itself for any $t\in[0,T]$ almost surely and the extension $\widehat{Z}_{t}^z(\omega)$ is continuous in $(t,z)$ almost surely. For all $\omega$ such that $\widehat{Z}$ is continuous, the map $\widehat{Z}_{t}(\omega):\widehat{\R}^{2d}\rightarrow\widehat{\R}^{2d}$ is homotopically equivalent to the identity map $\widehat{Z}_{0}(\omega)$. Proceeding by contradiction, assume that $\widehat{Z}_{t}(\omega)$ is not surjective. Then it takes values in $\widehat{\R^{2d}}$ without one point, which is a contractible space, so that is must be homotopically equivalent to a constant. This implies that also the map $Id_{\widehat{\R^{2d}}}=\widehat{Z}_{0}(\omega)$ is homotopically equivalent to a constant, and the space $\widehat{\R^{2d}}$ would be contractible, which is absurd (because, for example, $\pi_{2d}(\widehat{\R^{2d}})=\mathbb{Z}$). The contradiction found shows that the function $\widehat{Z}_{t}(\omega)$ needs to be an onto map. Since $\widehat{Z}_{t}^\infty(\omega)=\infty$, the restriction of $\widehat{Z}_{t}(\omega)$ to $\R^{2d}$ is again onto. The theorem is proved.
\end{proof}

\medskip
We now present some results on the convergence and regularity of approximations $\phi_{0,n}^t$ of the inverse flow $\phi_0^t$ associated to the SDE \eqref{eq Z}. Note that $\phi_{0,n}^t$ are solutions of SDEs with regular coefficients, see the proof of Theorem \ref{Theo regul}. These results are adapted from \cite{FF13b} and based on the following lemma on the stability of the PDE  \eqref{PDE-lambda}, which is of independent interest.

\begin{lemma}[Stability of the PDE \eqref{PDE-lambda}]\label{lemma: stability PDE}
Let $U_n$ be the unique solutions provided by Theorem \ref{PDE-1!} to the PDE \eqref{PDE-lambda} with smooth approximations $B_n(z)= (0, F_n(z))$ of $B(z) = (0, F(z))$ and some $\lambda$ large enough for \eqref{lambda n} to hold. If $F_n(z) \to F(z)$ in $L^p(\R^d_v ; H^{s}_p (\R^d_x))$, with $s,p$ as in Hypothesis \ref{hyp-holder}, then $U_n$ and $D_v U_n$ converge pointwise and locally uniformly to the respective limits. In particular, for any $r>0$ there exists a function $g(n) \to 0$ as $n\to \infty$ s.t.
\begin{align}
\sup_{z\in B_r} \big| U_n (z) - U(z) \big| \le g(n) \, ,  \nonumber\\
\sup_{z\in B_r} \big| D_v U_n (z) - D_v U(z) \big| \le g(n) \, .  \label{stab DU}
\end{align}
Moreover, there exists a $\lambda_0$ s.t. for all $\lambda>\lambda_0$
\begin{equation}\label{lambda n}
\big\| D_v U_n \big\|_{\infty} \le 1/2 \, .
\end{equation}
\end{lemma}
\begin{proof} Setting $V_n = (U_n -U)$ we write for $\lambda$ large enough (cf. \eqref{PDE-lambda})
\begin{align*}
 \lambda V_n (z)
  - \frac{1}{2} \mathrm{Tr} \big(Q D^2 V_n (z) \big) -   \langle Az, DV_n(z)\rangle  
 &-  \langle B(z), DV_n(z)\rangle  \\
  \nonumber =  B_n(z) - B(z)
 & +  \langle B_n (z) - B(z), D_vU_n (z)\rangle. 
\end{align*}
By \eqref{serve2} we know that
  \begin{gather*} 
 \sqrt{\lambda} \| D_v U_n \|_{L^p (\R^d_v ;  H^{s}_p (\R^{d}_x))}
\le C \| B_n \|_{L^p (\R^d_v ;  H^{s}_p (\R^{d}_x) ) }
\le C \| B \|_{L^p (\R^d_v ;  H^{s}_p (\R^{d}_x) ) }, \;\; n \ge 1,
\end{gather*}  
 with ${C = C(s,p,d, \|F\|_{L^p (\R^d_v ;  H^{s}_p (\R^{d}_x)}) >0}$. Hence applying \eqref{nht}, \eqref{lin2},
 \eqref{es41}
and Sobolev embedding we obtain \eqref{stab DU} with 
$
g(n)$ $ = C \| B - B_n \|_{L^p (\R^d_v ;  H^{s}_p (\R^{d}_x) ) }.  
$
 On the other hand the last assertion follows from \eqref{es41}.
\end{proof}

\medskip


\begin{lemma}[{\cite[Lemma 3]{FF13b}}]\label{lemma convergenza SDE} 
For every $R>0$, $a\ge 1$ and $z\in B_R$, 
\begin{equation*}\label{property 1}
\lim_{n\rightarrow\infty}\sup_{t\in\left[  0,T\right]  }\sup_{z\in B_R } \E\left[  \left\vert \phi_{0,n}^{t}\left( z \right)  -\phi_{0}%
^{t}\left( z \right)  \right\vert ^{a}\right]  = 0 \, .
\end{equation*}
\end{lemma}

\begin{proof}
To ease notation, we shall prove the convergence result for the forward flows $\phi_{t,n} \to \phi_{t}$. This in enough since the backward flow solves the same equation with a drift of opposite sign. Since the flow $\phi_t$ is jointly continuous in $(t,z)$, the image of $[0,T]\times B_R$ is contained in $[0,T] \times B_r$ for some $r<\infty$. Thus for $z\in B_R$, from Lemma \ref{lemma: stability PDE} we get $|U_n(\phi_{t,n}) - U(\phi_t)| \le g(n) + 1/2 |\phi_{t,n} - \phi_t|$ and $|D_v U_n(\phi_{t,n}) - D_v U(\phi_t)| \le g(n) + |D_v U_n(\phi_{t,n}) - D_v U_n(\phi_t)|$. Extending the definition \eqref{def phi} to $\gamma_n (z) = z+U_n(z)$ we have the approximate equivalence
\begin{align*}
\frac{2}{3} \Big( \big| \gamma_{n}(\phi_{t,n}) - \gamma(\phi_t)  \big| - g(n) \Big) \le  \big| \phi_{t,n} - \phi_t \big| \le 2\Big( \big| \gamma_{n}(\phi_{t,n}) - \gamma(\phi_t)  \big|  + g(n) 
\Big)\, .
\end{align*}
Therefore, it is enough prove the convergence result for the transformed flows $\gamma_{t,n} = \gamma_{n}  (\phi_{t,n}) \to \gamma (\phi_t)=\gamma_t$. Proceeding as in the proof of Theorem \ref{strong 1!} we get, for any $a\ge2$
\begin{align}
\frac{1}{a} \dd \big| \gamma_{t,n} - \gamma_{t} \big|^a \le \big| \gamma_{t,n} - \gamma_{t} \big|^{a-2} \bigg\{ &\big(  \gamma_{t,n} - \gamma_t  \big) \cdot \Big[ \lambda \big( U_n(\phi_{t,n}) - U(\phi_t) \big)  +
 A(\phi_{t,n} - \phi_t) \, \Big] \dd t     \nonumber   \\
& + \big(  \gamma_{t,n} - \gamma_t  \big) \cdot  \big(D U_n(\phi_{t,n}) - D U(\phi_t) \big)R \cdot \dd W_t    \nonumber  \\
& + C_{a,d} \big\| \big(D U_n(\phi_{t,n}) - D U(\phi_t)\big) R \big\|_{HS}^2 \, \dd t   \bigg\} \, .  \label{approx gamma}
\end{align}
The stochastic integral is a martingale. Since
$$
 \frac{|\phi_{t,n} - \phi_t|}{|\gamma_{t,n}-\gamma_t|} \le C \Big( 1 + \frac{g(n)}{|\gamma_{t,n} - \gamma_t|} \Big)\, ,
$$
 the term on the last line in \eqref{approx gamma} can be bounded using \eqref{stab DU} by a constant times $ |\gamma_{t,n} - \gamma_t|^a  \dd B_{t,n} + |\gamma_{t,n} - \gamma_t|^{a-2} g^2(n) ( \dd B_{t,n} + \dd t )$, where for every $n$ the process $B_{t,n}$ is defined as in \eqref{eq B_t} but with $D U_n(\phi_{t,n})$ and $D U_n(\phi_t)$ in the place of $DU(Z_t)$ and $DU(Y_t)$ respectively. One can show that $B_{t,n}$ share the same integrability properties of the process $A_t$ studied in Lemma \ref{lemma exp At}, uniformly in $n$, see \cite[Lemma 14]{FF13b}. Computing $\E[ e^{-B_{t,n}} |\gamma_{t,n} - \gamma_t|^a ]$ using It\^o formula and taking the supremum over $t\in[0,T]$ leads to
\begin{align*}
\sup_{t\in[0,T]} \E \Big[ e^{-B_{t,n}} \big| \gamma_{t,n} &- \gamma_t \big|^a \Big] \le 
C \E \Big[ \int_0^T e^{-B_{s,n}} \big| \gamma_{s,n} - \gamma_s \big|^a \dd s \Big] \\
&+  C g(n) \E \Big[ \int_0^T e^{-B_{s,n}} \big( \big| \gamma_{s,n} - \gamma_s \big|^{a-1} + g(n) \big| \gamma_{s,n} - \gamma_s \big|^{a-2} \big) \dd s \Big] \\
&+ g^2(n) \E \Big[ \int_0^T e^{-B_{s,n}} \big| \gamma_{s,n} - \gamma_s \big|^{a-2} \dd B_{s,n} \Big]\, .
\end{align*}
Using the integrability properties of $\phi_t, \phi_{t,n}, U(\phi_t), U_n(\phi_{t,n})$ one can see that all terms are bounded, uniformly in $n$. To conclude the proof we can pass to the limit
$$
\limsup_n \sup_{t\in[0,T]} \E \Big[ e^{-B_{t,n}} \big| \gamma_{t,n} - \gamma_t \big|^a \Big] \le C\int_0^T \limsup_n \sup_{t\in[0,s]} \E \Big[ e^{-B_{t,n}} \big| \gamma_{t,n} - \gamma_t \big|^a \Big] \dd s \, ,
$$
apply Gr\"onwall's lemma and proceed as in Corollary \ref{cor Z-Y} to get rid of the exponential term.
\end{proof}

\begin{lemma}[{\cite[Lemma 5]{FF13b}}]\label{lemma Df}
For every $a\geq1$, there exists $C_{a,d,T}>0$ such that
\begin{equation}
\sup_{t\in\left[  0,T\right]  }\sup_{z\in \R^{2d} } \E\big[ \big\vert D \phi_{0,n}^{t}( z )  \big\vert ^{a} \big]  \le C_{a,d,T}  \label{property 2}
\end{equation}
uniformly in $n$.
\end{lemma}

\begin{proof} Let us show the bound for the forward flows $\phi_{t,n}$. These are regular flows: let $\theta_{t,n}$ and $\xi_{t,n}$ denote the weak derivative of $D \phi_{t,n}$ and $D \gamma_{t,n} = D \gamma_n( \phi_{t,n})$, respectively. They are equivalent in the sense of \eqref{xi-theta}, so we shall prove the bound for $\xi_{t,n}$ instead of $\theta_{t,n}$. Proceeding as in the proof of Theorem \ref{teo SDE-differentiable flow} we obtain as in \eqref{bound derivative}
\begin{align*}
\dd e^{-C_1 B_{t,n}} \big| \xi_{t,n} \big|^a  \le e^{-C_1 B_{t,n}} \Big[ C_2 \big| \xi_{t,n} \big|^a \dd t + \dd M_t \Big] \, ,
\end{align*}
where the process $B_{t,n}$ is simply given by $\int_0^t | D D_v U_n(\phi_{s,n}) |^2 \dd s$. We can integrate, take expected values, the supremum over $t\in[0,T]$ and apply Gr\"onwall's inequality to get
$$  \sup_{t\in[0,T]} \E \big[ e^{-C_1 B_{t,n} } | \xi_{t,n}|^a \big]  \le C_{T} |\xi_{0,n}|^a = C_{a,d,T}\, . $$
Observe that this bound is uniform in $n$ and $z\in \R^{2d}$. Proceeding as in Corollary \ref{cor Z-Y} we can get rid of the exponential term and obtain the desired uniform bound on $\xi_{t,n}$.
\end{proof}

\section*{Acknowledgements}

The authors Ennio Fedrizzi and Julien Vovelle are supported by the LABEX MILYON (ANR-10-LABX-0070) of Université de Lyon, within the program ``Investissements d'Avenir" (ANR-11-IDEX- 0007) operated by the French National Research Agency (ANR).

The Author Enrico Priola is supported by the Italian PRIN project 2010MXMAJR.

The author Julien Vovelle is supported by the ANR STOSYMAP (ANR-11-BS01-0015) and the ANR STAB (ANR-12-BS01-0019).

\end{document}